\documentclass{amsart}

\usepackage{amscd,amsmath,amssymb,amsfonts,bbm, calligra, xspace}
\usepackage[T1]{fontenc}
\usepackage{lmodern}
\usepackage{mathtools}
\usepackage{enumerate}
\usepackage{multicol}
\usepackage[all]{xy}
\usepackage{mathrsfs}
\usepackage{footmisc}

\usepackage{stackrel}

\usepackage{xcolor}
\definecolor{darkblue}{rgb}{0,0,0.8}
\definecolor{darkgreen}{rgb}{0,0.4,0}
\usepackage[colorlinks=true,linkcolor=darkblue, citecolor=darkgreen,  urlcolor=darkgreen, unicode,pdfborder={0 0 0},final]{hyperref}

\newtheorem{thm}{Theorem}[section]
\newtheorem{prop}[thm]{Proposition}

\newtheorem{lem}[thm]{Lemma}
\newtheorem{cor}[thm]{Corollary}

\theoremstyle{definition}

\theoremstyle{remark}
\newtheorem{rem}[thm]{Remark}
\newtheorem{rems}[thm]{Remarks}

\newtheorem{Step}{Step}

\numberwithin{equation}{section}
\newcommand{\Pic}{\mathrm{Pic}}
\newcommand{\Br}{\mathrm{Br}}

\newcommand{\oor}{\mathrm{or}}
\newcommand{\Id}{\mathrm{Id}}
\newcommand{\Fitt}{\mathrm{Fitt}}
\newcommand{\Clop}{\mathrm{Clop}}
\newcommand{\Isom}{\mathrm{Isom}}

\newcommand{\Ker}{\mathrm{Ker}}

\newcommand{\Hom}{\mathrm{Hom}}
\newcommand{\cl}{\mathrm{cl}}

\newcommand{\N}{\mathrm{N}}
\newcommand{\cd}{\mathrm{cd}}

\newcommand{\colim}{\mathrm{colim}}
\newcommand{\rr}{\mathrm{r}}
\newcommand{\et}{\mathrm{\acute{e}t}}
\newcommand{\red}{\mathrm{red}}

\newcommand{\per}{\mathrm{per}}
\newcommand{\ind}{\mathrm{ind}}

\newcommand{\pt}{\mathrm{pt}}

\newcommand{\ab}{\mathrm{ab}}
\newcommand{\an}{\mathrm{an}}

\newcommand{\Specan}{\mathrm{Specan}}
\newcommand{\Spf}{\mathrm{Spf}}
\newcommand{\Spec}{\mathrm{Spec}}
\newcommand{\Sper}{\mathrm{Sper}}
\newcommand{\Frac}{\mathrm{Frac}}

\newcommand{\Gal}{\mathrm{Gal}}

\newcommand{\RR}{\mathrm{R}}

\newcommand{\isoto}{\myxrightarrow{\,\sim\,}}
\makeatletter
\def\myrightarrow{{\setbox\z@\hbox{$\rightarrow$}\dimen0\ht\z@\multiply\dimen0 6\divide\dimen0 10\ht\z@\dimen0\box\z@}}
\def\myrightarrowfill@{\arrowfill@\relbar\relbar\myrightarrow}
\newcommand{\myxrightarrow}[2][]{\ext@arrow 0359\myrightarrowfill@{#1}{#2}}
\makeatother

\def\loccit{\emph{loc}.\kern3pt \emph{cit}.{}\ }
\def\eg{e.g.\kern.3em}
\def\ie{i.e.,\ }
\def\resp {\text{resp.}\kern.3em}

\newcommand{\Homrond}{\mathscr{H}\mkern-4muom}

\def\A{\mathbb A}

\def\Z{\mathbb Z}
\def\C{\mathbb C}
\def\F{\mathbb F}
\def\G{\mathbb G}
\def\LL{\mathbb L}
\def\M{\mathbb M}

\def\P{\mathbb P}
\def\R{\mathbb R}
\def\N{\mathbb N}
\def\bS{\mathbb S}

\def\cN{\mathcal{N}}
\def\cM{\mathcal{M}}
\def\cL{\mathcal{L}}
\def\cO{\mathcal{O}}

\def\cF{\mathcal{F}}

\def\cM{\mathcal{M}}

\def\cU{\mathcal{U}}

\def\cI{\mathcal{I}}
\def\ci{\mathcal{C}^{\infty}}

\def\kp{\mathfrak{p}}

\def\km{\mathfrak{m}}

\def\kF{\mathfrak{F}}
\def\kL{\mathfrak{L}}

\def\kX{\mathfrak{X}}

\def\ta{\tilde{a}}

\def\th{\tilde{h}}

\def\tp{\tilde{p}}

\def\talpha{\tilde{\alpha}}

\def\tpsi{\tilde{\psi}}
\def\tphi{\tilde{\phi}}

\def\oT{\overline{T}}
\def\oZ{\overline{Z}}

\def\of{\overline{f}}

\def\op{\overline{p}}

\def\oX{\overline{X}}
\def\oU{\overline{U}}

\def\oK{\overline{K}}

\def\wS{\widetilde{S}}
\def\wT{\widetilde{T}}

\def\wD{\widetilde{D}}

\def\wX{\widetilde{X}}

\def\whT{\widehat{T}}
\def\whS{\widehat{S}}

\begin{document}

\title[]{On the field of meromorphic functions on a Stein surface}

\author{Olivier Benoist}
\address{D\'epartement de math\'ematiques et applications, \'Ecole normale sup\'erieure, CNRS,
45 rue d'Ulm, 75230 Paris Cedex 05, France}
\email{olivier.benoist@ens.fr}

\renewcommand{\abstractname}{Abstract}
\begin{abstract}
We prove that fields of meromorphic functions on Stein surfaces have cohomological dimension $2$, and solve the period-index problem and Serre's conjecture II for these fields. 
We obtain analogous results for fields of real meromorphic functions on Stein surfaces equipped with an antiholomorphic involution.  We deduce an optimal quantitative solution to Hilbert's $17$th problem on analytic surfaces.
\end{abstract}
\maketitle

\section*{Introduction}

  The function field $\C(X)$ of an integral algebraic variety of dimension~$n$ over $\C$ is a finitely generated field of transcendence degree $n$ over $\C$.  Its arithmetic properties impact the geometry of $X$ and are therefore important to investigate.

  The complex-analytic analogues of these fields are the fields $\cM(S)$ of meromorphic functions on a connected normal complex space $S$ of dimension $n$. They are not always interesting: they may be reduced to the field~$\C$ of constants even if~$n>0$.
To ensure that meromorphic functions on $S$ are abundant, one must restrict the class of complex spaces under consideration. Two cases of interest are projective varieties and Stein spaces.  If $S$ is projective,  then all meromorphic functions on $S$ are algebraic by GAGA,  and one is reduced to the case of algebraic function~fields.

We henceforth consider the case of \textit{Stein spaces} (analytic analogues of affine varieties, characterized by the vanishing of higher cohomology groups of coherent sheaves, see \cite{GRStein}). Their fields of meromorphic functions are much bigger than algebraic function fields, and correspondingly harder to study. A typical example is $\cM(\C^n)$: the fraction field of the ring of convergent power series on $\C^n$.  

In this article, we focus on the case where $S$ is a Stein surface, and we investigate various \textit{arithmetic} questions concerning the field of \textit{analytic} origin $\cM(S)$.

\subsection{Cohomological dimension}
\label{cohodim}

Let $F$ be a field with absolute Galois group~$\Gamma_F$. 
The \textit{cohomological dimension} of $F$ is the largest integer $n$ such that there exists a finite $\Gamma_F$-module~$M$ with $H^n(\Gamma_F,M)\neq 0$ (or $+\infty$ if there is no upper bound on these integers).
Let $X$ be an integral algebraic variety of dimension $n$ over $\C$.  It is a consequence of Tsen's theorem that the field~${F=\C(X)}$ of rational functions on~$X$ has cohomological dimension~$n$ (see \mbox{\cite[II.4.2, Proposition 11]{CG}}).

That the field of meromorphic functions on a connected normal Stein curve (that is,  of a connected noncompact Riemann surface) has cohomological dimension $1$ has been known for a long time and is attributed by Guralnick to M. Artin (see \cite[Proposition 3.7]{Guralnick}). Our first result deals with the case of Stein surfaces.

\begin{thm}[Theorem \ref{thcohodim}]
\label{thcohodimintro}
Let $S$ be a connected normal Stein surface. Then the field $\cM(S)$ of meromorphic functions on $S$ has cohomological dimension $2$.
\end{thm}

When the dimension $n$ of $S$ is arbitrary,  a positive result in this direction has been obtained in \cite{Stein}, under an additional compactness hypothesis.  More precisely,  it is shown in \cite[Theorem 0.4]{Stein} that if a connected compact set $K\subset S$ is Stein (\ie admits a basis of Stein open neighborhoods), then the field $\cM(K)$ of germs of meromorphic functions along $K$ has cohomological dimension $n$.  Theorem \ref{thcohodimintro} removes this compactness hypothesis (and hence overcomes the difficulties related to the singularities at infinity of meromorphic functions on $S$),  when $n=2$.

\subsection{The real case}
\label{cohodimR}

Let $X$ be an integral algebraic variety of dimension~$n$ over~$\R$.  If $X(\R)$ is Zariski-dense in $X$, then $-1$ is not a sum of squares in~$\R(X)$.  It follows from E. Artin and Schreier's work that the field $\R(X)$ can be ordered compatibly with the field structure (see \cite[Satz 7b]{AS}), and hence that it has infinite cohomological dimension (see \eg \cite[Remark 7.5]{Scheiderer}). 
Conversely, a result attributed to Ax by Colliot-Th\'el\`ene and Parimala states that if $X(\R)$ is not Zariski-dense in $X$ (for instance if $X(\R)=\varnothing$), then $\R(X)$ has cohomological dimension $n$ (see \mbox{\cite[Proposition~1.2.1]{CTP}}). 

With applications to real-analytic geometry in mind, it is important to investigate analogues of this result in Stein geometry.
Let $G:=\Gal(\C/\R)\simeq \Z/2$ be the group generated by the complex conjugation.
A $G$\nobreakdash\textit{-equivariant Stein space} is a Stein space equipped (as a locally ringed space) with an action of $G$ such that the complex conjugation acts $\C$-antilinearly on the structure sheaf (in other words,  it is a Stein space equipped with an antiholomorphic involution).
We prove the following real counterpart of Theorem \ref{thcohodimintro} (in dimension $1$, see \cite[Proposition~A.8]{Tight}).

\begin{thm}[Theorem \ref{thcohodimR}]
\label{thcohodimRintro}
Let $S$ be a normal $G$-equivariant Stein surface with~$S/G$ connected.  Let $\cM(S)^G$ be the field of $G$-invariant meromorphic functions on~$S$. The following assertions are equivalent:
\begin{enumerate}[(i)]
\item the field $\cM(S)^G$ has finite cohomological dimension;
\item the field $\cM(S)^G$ has cohomological dimension $2$;
\item the field $\cM(S)^G$ admits no field orderings;
\item $S^G$ is a discrete subset of $S$.
\end{enumerate}
\end{thm}

In view of Theorem \ref{thcohodimRintro}, to overcome difficulties related to the failure of~$\cM(S)^G$ having finite cohomological dimension (equal to $2$) in general, it is critical to control the field orderings of $\cM(S)^G$. This is the purpose of the next theorem.

Recall that if $A$ is a ring, the \textit{real spectrum} $\Sper(A)$ of $A$ is the set of pairs~$(\kp,\prec)$ where $\kp$ is a prime ideal of $A$ and $\prec$ is a field ordering of $\Frac(A/\kp)$. It is endowed with the \textit{spectral topology} \cite[Definition 7.1.3]{BCR}, which is generated by open sets of the form $\{(\kp,\prec)\in\Sper(A)\mid a_1,\dots,a_m\succ 0\}$ for some $a_1,\dots, a_m\in A$.
If $S$ is a Stein surface, we identify $s\in S^G$ with $(\km_s,\prec_s)\in \Sper(\cO(S)^G)$,  where $\km_s=\{a\in\cO(S)^G\mid a(s)=0\}$ and $\prec_s$ is the unique field ordering of $\cO(S)^G/\km_s=\R$.

\begin{thm}[Theorem \ref{thorder2}]
\label{thmorderings}
If $S$ is a normal $G$-equivariant Stein surface such that~$S/G$ is connected,  the closure of $S^G$ in~$\Sper(\cO(S)^G)$ contains $\Sper(\cM(S)^G)$.
\end{thm}

Theorem \ref{thmorderings} is the so-called Artin--Lang property for the field $\cM(S)^G$.  For real-analytic surfaces (which amounts to only considering the germ of $S$ along~$S^G$), the Artin--Lang property was proven by Castilla \cite[Theorem~1.1]{Castilla} in the nonsingular case and by Andradas, D\'iaz-Cano and Ruiz \cite[Theorem~2]{ADR} in general.

\subsection{The period-index problem}

Let $F$ be a field with separable closure $F^s$ and absolute Galois group $\Gamma_F$. The \textit{Brauer group} $\Br(F)$ of $F$ is the Galois cohomology group $H^2(\Gamma_F,(F^s)^*)=H^2_{\et}(\Spec(F),\G_m)$. It may equivalently be defined as
the set of isomorphism classes of central division algebras $D$ over~$F$,  with law group given by
$[D_1]\cdot[D_2]=[D_3]$ if and only if $D_1\otimes_F D_2\simeq M_N(D_3)$ for some~$N\geq 0$.

Two important invariants of a class $\eta\in\Br(F)$ are its period $\per(\eta)$ (its order in the torsion group~$\Br(F)$) and its index $\ind(\eta)$ (the smallest degree of a finite field extension $F'/F$ splitting $\eta$; equivalently the gcd of these degrees; equivalently the integer $\sqrt{\dim_F(D)}$ where $D$ is a central
division algebra over $F$ representing~$\eta$).  
These two integers share the same prime factors, and $\per(\eta)\mid\ind(\eta)$.
We refer to~\cite{GS} for a textbook account of this theory.

The period-index problem aims at controlling $\ind(\eta)$ when $\per(\eta)$ is known. Its difficulty increases with the arithmetic complexity of~$F$. If $F$ is perfect of cohomological dimension~$1$,  then $\Br(F)=0$ (see \cite[II.3.1, Proposition 6]{CG}) and the problem is vacuous.  
If~$F$ has cohomological dimension~$2$, it is often reasonable to expect that $\ind(\eta)=\per(\eta)$.
This may fail (see Merkurjev's counterexamples \mbox{\cite[\S3]{Merkurjev}}), but it holds for many fields of geometric or arithmetic interest,  such as function fields of complex algebraic surfaces (de Jong's period-index theorem~\cite{deJong, Lieblich}).
The survey~\cite{BourCT} contains a detailed discussion and many more examples. 
Our next result covers the case of fields of meromorphic functions on Stein surfaces.

\begin{thm}[Theorem \ref{thpiC}]
\label{piintro}
Let $S$ be a connected normal Stein surface. For any~$\eta\in\Br(\cM(S))$, one has $\ind(\eta)=\per(\eta)$.
\end{thm}

Our arguments also prove that any class $\eta\in\Br(\cM(S))$ is cyclic (see Remark~\ref{remcyclic}).

Real analogues of de Jong's period-index theorem were investigated in \cite{pi}.  
In particular, it was shown in \cite[Theorem 0.3]{pi} that if $X$ is an integral algebraic surface over $\R$ such that $X(\R)$ is not Zariski-dense in $X$, then $\ind(\eta)=\per(\eta)$ for all classes~${\eta\in\Br(\R(X))}$.
Here is the counterpart of this result in Stein geometry.

\begin{thm}[Theorem \ref{corpi}]
\label{piRintro}
Let $S$ be a normal $G$-equivariant Stein surface with $S/G$ connected and $S^G$ discrete.  For any $\eta\in\Br(\cM(S)^G)$, one has $\ind(\eta)=\per(\eta)$.
\end{thm}

One cannot remove the hypothesis that $S^G$ be discrete in Theorem~\ref{piRintro}.  Indeed, the sum of quaternion algebras $(-1,-1)+(x,y)\in\Br(\cM(\C^2)^G)$, where~$G$ acts naturally on $\C^2$, has period $2$ and index $4$ (it still has index $4$ in $\Br(\R((x))((y)))$ by \cite[VI, Example 1.11]{Lam}; such examples go back to Albert \cite[Theorem~2]{Albert}).

In the algebraic setting,  the hypothesis that $X(\R)$ is not Zariski-dense was weakened in \cite[Theorem 0.4]{pi} by only requiring that $\eta$ vanishes in restriction to the real points of some Zariski-dense open subset of $X$ (see also \cite[Theorem~4.1]{CTOP} in the local case).  We improve Theorem~\ref{piRintro} by weakening the hypothesis that~$S^G$ be discrete in a similar way.  

If $F$ is a field and $\xi\in\Sper(F)$ is a field ordering, we let $F_{\xi}$ be the associated \textit{real closure} of $F$ (its biggest ordered algebraic extension).  We then define 
\begin{equation}
\label{Brnul}
\Br(F)_0:=\Ker\big[\Br(F)\to \prod_{\xi\in\Sper(F)}\Br(F_{\xi})\big].
\end{equation}

\begin{thm}[Theorem \ref{thpi}]
\label{pi3}
Let $S$ be a normal $G$-equivariant Stein surface with~$S/G$ connected.  If $\eta\in\Br(\cM(S)^G)_0$, then $\ind(\eta)=\per(\eta)$.
\end{thm}

A further application of these results to the $u$-invariant is presented in \S\ref{paru}.

\subsection{Serre's conjecture II}

It is a theorem of Steinberg \cite[Theorem 1.9]{Steinberg} (previously Serre's conjecture I) that if $H$ is a connected linear algebraic group over a perfect field $F$ of cohomological dimension $1$, then all $H$-torsors over $F$ are trivial (that is, $H^1(F,H)=0$). Serre's conjecture II asserts that if $H$ is a simply connected semisimple algebraic group over a perfect field of cohomological dimension~$2$, then $H^1(F,H)=0$.  Partial results are known for particular groups (see \cite{BP, Chernousov, Gille}) or particular fields (such as function fields of complex algebraic surfaces, thanks to de Jong, He and Starr \cite[Theorem 1.5]{dJHS}). We refer to the survey~\cite{Gillesurvey} for more information.  We solve Serre's conjecture II for fields of meromorphic functions of possibly $G$-equivariant Stein surfaces.

\begin{thm}[Theorem \ref{thSerre1}]
\label{thmII1}
Let $S$ be a connected normal Stein surface.  If $H$ is a simply connected semisimple algebraic group over $\cM(S)$, then $H^1(\cM(S),H)=0$.
\end{thm}

\begin{thm}[Theorem \ref{thSerre2}]
\label{thmII2}
Let $S$ be a normal $G$-equivariant Stein surface with~$S/G$ connected and $S^G$ discrete.  If $H$ is a simply connected semisimple algebraic group over~$\cM(S)^G$, then ${H^1(\cM(S)^G,H)=0}$.
\end{thm}

These theorems follow from known results on Serre's conjecture II, from Theorems \ref{thcohodimintro}, \ref{thcohodimRintro},~\ref{piintro} and \ref{piRintro}, and from the following theorem of independent interest (whose proof is inspired by the analogous result of Colliot-Th\'el\`ene, Ojanguren and Parimala \cite[Theorem 2.2]{CTOP} in the local case).  

\begin{thm}[Theorem \ref{thab}]
Let $S$ be a connected normal Stein surface. The maximal abelian extension $\cM(S)^{\ab}$ of $\cM(S)$ has cohomological dimension $1$. 
\end{thm}

When $S^G$ is not discrete, the field $\cM(S)^G$ only has virtual cohomological dimension $2$ (i.e.\,$\cM(S)^G[\sqrt{-1}]$ has cohomological dimension $2$), and Serre's conjecture II should be replaced with a Hasse principle proposed by Colliot-Th\'el\`ene \cite{CTreal} and Scheiderer \cite{ScheidererHasse} (see \cite[p.\,652]{BP2} for a precise statement).  We do not know how to prove this Hasse principle in general for the fields $\cM(S)^G$.

%classical in [Bayer, Parimala, Classical groups and the Hasse principle]
%nonclassical non E_8 qsplit in [Chernousov,  the kernel...]
%To reduce to qsplit mimic Chernousov, Theorem 8.3? ou Gille IV.2 ? [Suivre plutôt Philippe, faut vérifier que à chaque fois les Brauer classes qui apparaissent sont nuls sur les clôtures réelles.]
%Remains E_8. Then all autos are interiors, so we may assume split in fact (argument as in dJHeStarr). 
%Alors, on voudrait adapter [CTGP, Arithmetic of linear algebraic groups...,  Proof of Th 1.2 (iv)] avec l'hyp vcd(kabtot)<=1 (tot=totalement réel) ??? %(The hyp that cdab<=1 is OK (tuer ramification + extraire sqrt(-1))). 
%Il reste à adapter [Gille,  Cohomologie galoisienne des groupes quasi-déployés...,  Theorem 11] en mettant partout des hyps nul sur les clôtures réelles. Only need p=2 (killed by sqrt(-1) by cd=2 case).
%Cas p=2, il faut adapter Corollaire 2 (a),  à son tour il faut adapter Prop 2 (a).  Possiblement il faut remonter à Prop 1 et là remplacer l'application de MS par l'énoncé correspondant de Bayer-Parimala ??? Urgh

\subsection{Hilbert's \texorpdfstring{$17$}{17}th problem}
\label{parsos}

E. Artin's positive answer to Hilbert's $17$th problem (\mbox{\cite[Satz 4]{Artin}}, see also \cite[Theorem 9]{LAng}) shows that if $X$ is an integral variety of dimension $n$ over~$\R$ and $h\in\cO(X)$ takes nonnegative values on~$X(\R)$,  then $h$ is a sum of squares in~$\R(X)$. 
A quantitative refinement of Artin's theorem was obtained by Pfister (\cite[Theorem 1]{Pfister}, see also \cite[Theorem 2]{PfisterICM}), who showed that $h$ is then a sum of $2^n$ squares in $\R(X)$.
We prove an analogue of the results of Artin and Pfister for $G$-equivariant holomorphic functions on $G$-equivariant Stein surfaces.

\begin{thm}[Theorem \ref{3squares}]
\label{3squaresintro}
Let $S$ be a normal $G$-equivariant Stein surface. Fix $h\in\cO(S)^G$. The following assertions are equivalent:
\begin{enumerate}[(i)]
\item there exists a closed discrete subset $\Sigma\subset S^G$ such that $h\geq 0$ on $S^G\setminus\Sigma$;
\item $h$ is a sum of squares in $\cM(S)^G$;
\item $h$ is a sum of $3$ squares in $\cM(S)^G$.
\end{enumerate}
\end{thm}

One deduces the following consequence in the more classical real-analytic setting.

\begin{thm}[Theorem \ref{corsquares}]
\label{3squares2}
Let $M$ be a normal real-analytic variety of pure dimension~$2$ and fix ${h\in\cO(M)}$. The following assertions are equivalent:
\begin{enumerate}[(i)]
\item $h\geq 0$ on $M$;
\item $h$ is a sum of squares in $\cM(M)$;
\item $h$ is a sum of $3$ squares in $\cM(M)$.
\end{enumerate}
\end{thm}

The qualitative statement Theorem \ref{3squares2} (i)$\Leftrightarrow$(ii) was already known \cite[Theorem 1]{ADR}. Its quantitative refinement (Theorem~\ref{3squares2}~(i)$\Leftrightarrow$(iii)) was known if $M$ is a manifold \cite[Corollary 2]{Jaworski1},  or with the weaker bound $5$ on the number of squares (see \cite[Theorem 1.2]{ABFR1} or \cite{Fernando}). Bounding the required number of squares by $3$, as we do,  is optimal in general (see \cite[Corollary 2]{Jaworski1}).  In higher dimensions,  even the qualitative statement is not known to hold (unless $M$ is compact,  see~ \cite[Theorem 1]{Ruiz}, \cite[Theorem 1]{Jaworski2} and \cite[Theorem~0.1]{Stein}).

Theorem \ref{3squaresintro} is entirely new.  Its qualitative part (Theorem \ref{3squaresintro} (i)$\Leftrightarrow$(ii)) is a direct application of Theorem \ref{thmorderings}, in view of the relation between field orderings and sums of squares discovered by Artin \cite[Satz 1]{Artin}.

\subsection{A generic comparison theorem over Stein surfaces}

Let $S$ be a connected normal Stein surface.  The main difficulty in proving the theorems stated in~\S\S\ref{cohodim}--\ref{parsos} is to find a way to compute or control the Galois cohomology of~$\cM(S)$.  Our strategy is to relate it to the singular cohomology of (Zariski-open subsets of)~$S$. 
More generally, we hope to compare the \'etale cohomology of an~$\cO(S)$\nobreakdash-scheme~$X$ of finite presentation, and the singular cohomology of its analytification~$X^{\an}$ (in the sense of Bingener~\cite{Bingener}, see~\S\ref{paranal}).
In a perfect world, the comparison morphisms
$$H^k_{\et}(X,\LL)\to H^k(X^{\an},\LL^{\an})$$
would be isomorphisms for all constructible \'etale sheaves $\LL$ on $X$ and all $k\geq 0$. Unfortunately, this is false in general (see Remark \ref{remcomp} (iv)). Our way out is to prove the following weaker \textit{generic} comparison theorem.

\begin{thm}[Theorem \ref{compC}]
\label{compintro}
Let $S$ be a reduced Stein space of dimension $\leq 2$.  
Let $\LL$ be a constructible \'etale sheaf on an $\cO(S)$-scheme of finite presentation $X$.  If one lets~$a\in\cO(S)$ run over all non\-zero\-di\-vi\-sors, the comparison morphisms
\begin{equation}
\label{geniso}
\underset{a}{\colim }\,H^k_{\et}(X_{\cO(S)[\frac{1}{a}]},\LL)\to \underset{a}{\colim}\, H^k((X_{\cO(S)[\frac{1}{a}]})^{\an},\LL^{\an})
\end{equation}
are isomorphisms for all $k\geq 0$.
\end{thm}

The left-hand side of (\ref{compintro}) is equal to $H^k_{\et}(X_{\cM(S)},\LL)$ (see  \cite[Lemma~\href{https://stacks.math.columbia.edu/tag/03Q6}{03Q6}]{SP}), and hence computes the Galois cohomology of $\cM(S)$ when $X=\Spec(\cO(S))$.

Theorem \ref{compintro} and its $G$-equivariant extension Theorem \ref{compG} allow us to reduce our main results to problems pertaining to the singular cohomology of Stein surfaces. These are solved using that Stein surfaces have the homotopy type of $2$\nobreakdash-dimensional CW complexes, and that their higher coherent cohomology groups~vanish.

When $S$ is a point, Theorem \ref{compintro} amounts to M. Artin's comparison theorem between the \'etale cohomology of a complex algebraic variety and the singular cohomology of its analytification \cite[XVI, Th\'eor\`eme 4.1]{SGA43}. 

The first extension of Artin's theorem to Stein geometry appeared in \cite[Theorem 0.5]{Stein}. The comparison theorem of \loccit was restricted to algebraic varieties over Stein compacta and therefore avoided complications related to the singularities at infinity of holomorphic functions on~$S$, as well as most difficulties related to the non\-noetherian\-ness of $\cO(S)$.  It was however valid in all dimensions, and did not require to invert the non\-zero\-di\-vi\-sors of~$\cO(S)$. 

In contrast, the map~(\ref{geniso}) would fail to be an isomorphism in general if one did not take the colimit over all non\-zero\-di\-vi\-sors (see Remark \ref{remcomp} (iv)).
It also fails to be an isomorphism in general if $S$ is nonreduced, or if $X$ is only of finite type over~$\cO(S)$, or if $\LL$ is not constructible (see Remarks \ref{remcomp} (i), (ii) and (iii)). 
We do not know if Theorem~\ref{compintro} remains valid with no restriction on the dimension of $S$.

\subsection{Proof of the comparison theorem}
\label{parproof}

Using the relative comparison theorem \cite[Theorem 3.7]{Stein} and standard d\'evissage arguments, one reduces the proof of Theorem \ref{compintro} to the case where $X=\Spec(\cO(S))$ and $\LL=\Z/m$ for some $m\geq 1$.  In this crucial case, the morphism (\ref{geniso}) takes the particularly simple form
\begin{equation}
\label{genisosimple}
\underset{a}{\colim }\, H^k_{\et}(\Spec(\cO(S)[\tfrac{1}{a}]),\Z/m)\to\underset{a}{\colim }\,H^k(S\setminus\{a=0\},\Z/m).
\end{equation}
%whose left-hand side is nothing but the Galois cohomology group $H^k(\Gamma_{\cM(S)},\Z/m)$.
Inspired by the proof of Artin's comparison theorem in the algebraic case,  we introduce the change of topology morphism $\varepsilon:S_{\cl}\to (\Spec(\cO(S)))_{\et}$ (see~\S\ref{paranal} for more details).
One would ideally want to prove that $\Z/m\isoto\varepsilon_*\Z/m$ and that~${\RR^s\varepsilon_*\Z/m=0}$ for~$s>0$. 
If that were true, we would deduce from the Leray spectral sequence for $\varepsilon$ that (\ref{genisosimple}) is an isomorphism (even without taking the colimit on $a$). Unfortunately, this is not quite true (which explains why our comparison theorem only holds generically).  We prove weaker versions of these assertions that suffice to deduce the validity of Theorem \ref{compintro}.

The morphism of sheaves $\Z/m\to\varepsilon_*\Z/m$ and the sheaf $\RR^1\varepsilon_*\Z/m$ are investigated in Sections~\ref{secH0} and \ref{secH1} respectively.  
This amounts to studying closed and open subsets (\resp cyclic finite \'etale covers) of \'etale $\cO(S)$-schemes.
In doing so, we impose as few restrictions as possible on the base Stein space $S$ (see Theorem~\ref{thmclopen}, Proposition~\ref{morphcov} and Theorem~\ref{thmconstr}). In particular, the dimension of $S$ may be arbitrary. In addition, in Section~\ref{secH1}, we consider arbitrary finite \'etale covers that may not be cyclic. These results are complemented by examples showing their optimality (see Remarks \ref{remH0} and \ref{remH1}).
This part of our work is in the spirit of (and relies on) classical algebraic results concerning rings of meromorphic or holomorphic functions on Stein spaces, such as Iss'sa's \cite{Isssa} and Forster's \cite{Forster}.

To control the sheaf $\RR^2\varepsilon_*\Z/m$, we study degree $2$ singular cohomology classes on
Stein surfaces in Section \ref{secH2} (see notably Propositions \ref{killram} and \ref{GLefschetz11}). 
In contrast with the results of Sections \ref{secH0} and \ref{secH1},  we rely in a crucial way on~$S$ being of dimension~$\leq~2$.
In addition, for later use in the proofs of Theorems \ref{piRintro} and \ref{pi3}, it is essential 
that we work $G$-equivariantly there.

As for the sheaves $\RR^s\varepsilon_*\Z/m$ for $s\geq 3$,  
it is easy to show that they vanish on the nose when $S$ has dimension $\leq 2$, using that Stein spaces of dimension~$\leq 2$ have the homotopy type of CW complexes of dimension $\leq 2$.

\subsection{Structure of the article}

Generalities on Stein spaces and the analytification functor are gathered in Section \ref{secanal}.
As already indicated in \S\ref{parproof}, Sections~\ref{secH0},~\ref{secH1} and~\ref{secH2} contain the results that we need to control cohomology classes of respective degrees~$0$,~$1$ and~$2$ in the proof of the generic comparison theorem (Theorem~\ref{compintro}).  They are combined in Section \ref{seccomp} to prove Theorem \ref{compintro} (and its $G$-equivariant companion Theorem \ref{compG}).  

We deduce our theorems on the cohomological dimension and on the orderings of fields of (possibly $G$-invariant) meromorphic functions in Section \ref{seccohodim}.  Applications to the period-index problem, to Serre's conjecture II, and to Hilbert's $17$th problem appear in Sections \ref{secpi}, \ref{secSerre} and \ref{secH17} respectively.

\subsection{Acknowledgements}

I thank James Hotchkiss for an interesting email correspondence, and an anonymous referee for their useful comments.

\section{Stein spaces and algebraic varieties over them}
\label{secanal}

\subsection{Generalities on complex spaces}

Complex spaces (in the sense of \cite[1, \S 1.5]{GRCoherent}) are assumed to be second-countable, but may not be Hausdorff, reduced, or finite-dimensional. A complex space $S$ is said to be \textit{Stein} if~${H^k(S,\cF)=0}$ for all coherent sheaves $\cF$ on $S$ and all $k>0$ (see \cite{GRStein}). Stein spaces are Hausdorff.

We set $G:=\Gal(\C/\R)\simeq \Z/2$ and we let $\sigma\in G$ denote the complex conjugation.  As in \cite[Appendix A]{Tight}, we define a $G$\nobreakdash\textit{-equivariant complex space} to be a complex space endowed (as a $\C$-ringed space) with an action of $G$ such that the complex conjugation acts $\C$\nobreakdash-antilinearly on the structure sheaf. It is said to be \textit{Stein} (\resp reduced, normal, nonsingular...) if so is the underlying complex space.

If $S$ is a complex space, we let $S^{\sigma}$ denote the \textit{complex conjugate} of $S$. It is equal to $S$ as a ringed space, but its structural morphism $\mu:\C\to\cO_S$ is replaced with~$\mu\circ \sigma$. 
The disjoint union $S\sqcup S^{\sigma}$ then has a natural structure of $G$-equivariant complex space,  obtained by letting $\sigma$ exchange the two factors.

\begin{lem}
\label{lemrealpoints}
(i)
If $S$ is a $G$-equivariant complex manifold of dimension $n$, then~$S^G$ is a $\ci$ manifold of dimension $n$.

(ii)
Let $S$ be a $G$-equivariant complex space. If $S^G$ is included in a nowhere dense complex subspace of $S$, then $S^G$ only contains singular points of~$S$.
\end{lem}

\begin{proof}
Let $s\in S^G$ be nonsingular.  Using a $G$\nobreakdash-in\-vari\-ant local system of coordinates ${z_1,\dots, z_n\in(\cO_{S,s})^G}$, we can $G$\nobreakdash-equivariantly identify some open neighborhood $\Omega$ of $s$ in $S$ with an open neighborhood of the origin in $\C^n$. This proves~(i).

To prove (ii),  note that if moreover $a\in\cO(\Omega)$ vanishes on $\Omega^G=\Omega\cap\R^n$, then the coefficients of the expansion of~$a$ in powers of~$z_1,\dots, z_n$ vanish, so $a=0$ in~$\cO_{S,s}$.
\end{proof}

A proper holomorphic map $p:T\to S$ between reduced complex spaces is a \textit{modification} (\resp an \textit{alteration}) if there exists a nowhere dense closed analytic subset $\Sigma\subset S$ such that $p^{-1}(\Sigma)$ is nowhere dense in $T$ and ${p|_{p^{-1}(S\setminus \Sigma)}:p^{-1}(S\setminus \Sigma)\to S\setminus \Sigma}$ is an isomorphism (\resp a local biholomorphism). A finite alteration is called an \textit{analytic covering} (unlike in \cite[7, \S 2.1]{GRCoherent},  we do not insist that~$p$ be surjective).  An alteration is said to be \textit{of degree} $d$ (\resp \textit{of bounded degree}) if the fibers of~$p|_{p^{-1}(S\setminus \Sigma)}$ have cardinality $d$ (\resp bounded cardinality).

We let $\cM(S)$ denote the ring of meromorphic functions on a complex space $S$. If $p:T\to S$ is an analytic covering of degree $1$ between reduced complex spaces (for instance the normalization map), then $\cM(S)\isoto\cM(T)$ (see \cite[8, \S 1.3]{GRCoherent}). If~$S$ is Stein and reduced, then $\cM(S)$ is the total ring of fractions of $\cO(S)$ (use the coherence of sheaves of denominators \cite[6, \S 3.2]{GRCoherent}).
%for embedded points, see [Siu, Noether-Lasker decomposition of coherent analytic subsheaves]
If $S$ is a reduced Stein space with finitely many irreducible components, there is an equivalence of categories  
\begin{equation}
\label{eqext}
 \left\{  \begin{array}{l}
\textrm{analytic coverings }
T\to S\\\textrm{\hspace{2.2em}with $T$ normal}
  \end{array}\right\}\to
 \left\{  \begin{array}{l}
    \textrm{finite \'etale }\cM(S)\textrm{-algebras}
  \end{array}\right\}
\end{equation}
sending a degree $d$ analytic covering $T\to S$ with $T$ normal to the degree $d$ finite \'etale $\cM(S)$\nobreakdash-al\-ge\-bras $\cM(T)$ (see \cite[\S 1.5]{Stein}).
Similarly, it is shown in \cite[Proposi\-tion~5.4]{Stein} that if $S$ is a reduced $G$-equivariant Stein space with finitely many irreducible components, there is an equivalence of categories
\begin{equation}
\label{eqextG}
 \left\{  \begin{array}{l}
G\textrm {-equivariant analytic coverings}
\\ \hspace{2.2em}T\to S\textrm{ with $T$ normal}
  \end{array}\right\}\to
 \left\{  \begin{array}{l}
    \textrm{finite \'etale }\cM(S)^G\textrm{-algebras}
  \end{array}\right\}.
\end{equation}

\subsection{The analytification functor}
\label{paranal}

Let $S$ be a Stein space.  Beware that the ring~$\cO(S)$ of holomorphic functions on $S$ is not noetherian in general.
If~$X$ is an~$\cO(S)$\nobreakdash-scheme locally of finite presentation,  Bingener \cite[Satz 1.1]{Bingener} has defined the \textit{analytification} $X^{\an}$ of $X$ to be the complex space over $S$  endowed with a morphism~$i_X:X^{\an}\to X$ of locally ringed spaces such that 
\begin{equation}
\label{defan}
\Hom_S(S',X^{\an})\xrightarrow{i_X\circ\, -}\Hom_{\cO(S)-\textrm{locally ringed spaces}}(S',X)
\end{equation}
is bijective for all complex spaces $S'$ over $S$.
Then $X\mapsto X^{\an}$ is a functor \mbox{\cite[p.\,2]{Bingener}} respecting fiber products \cite[p.\,3]{Bingener} and compatible with change of the base Stein space \cite[(1.2)]{Bingener}.  It is also compatible with restriction of scalars as we now show.

\begin{lem}
\label{lemBing}
Let $S'\to S$ be a holomorphic map between Stein spaces.  Let~$X'$ be an $\cO(S')$-scheme locally of finite presentation.  Let $X$ be the scheme $X'$ viewed as an $\cO(S)$-scheme.  Assume that the $\cO(S)$-scheme $X$ is locally of finite presentation. Then the analytifications $X^{\an}$ and $(X')^{\an}$ are canonically isomorphic.
\end{lem}

\begin{proof}
When $S'$ is finite-dimensional, a proof is given in \cite[(1.3)]{Bingener}. This proof works in general, replacing the reference \cite{Forster} by \cite[Proposition 3.4]{Steinalgebra}.
\end{proof}

To give just a few examples, one can compute that $\Spec(\cO(S))^{\an}=S$, that $\Spec(\cO(S)[\frac{1}{a}])^{\an}=S\setminus\{a=0\}$, that $(\A^N_{\cO(S)})^{\an}=S\times\C^N$, etc.

\begin{lem}
\label{lemCpoints}
Let $S$ be a Stein space. Let $X$ be an $\cO(S)$-scheme locally of finite presentation. The map $i_X:X^{\an}\to X$ induces a bijection $X^{\an}\isoto X(\C)$ (where~$X(\C)$ is the set of $\C$-points of $X$ viewed as a $\C$-scheme).
\end{lem}

\begin{proof}
When $X=\Spec(\cO(S))$, the lemma follows from \cite[\S 1]{Forster} if $S$ is finite-dimensional, and from \cite[Theorem 0.1]{Steinalgebra} in general. The general case of the lemma now follows from (\ref{defan}) applied with $S'$ equal to a point of $S$.
\end{proof}

It is shown in \cite[Satz 3.1]{Bingener} that if $f:X\to Y$ is a morphism of finite presentation between~$\cO(S)$\nobreakdash-schemes locally of finite presentation, and if~$f$ is separated (\resp proper,  finite, flat, \'etale) then ${f^{\an}:X^{\an}\to Y^{\an}}$  is separated (\resp proper, finite, flat, a local biholomorphism). 

Let $X$ be an $\cO(S)$-scheme locally of finite presentation. Let $X_{\et}$ be the small \'etale site of $X$, and let $(X^{\an})_{\cl}$ be the site of local isomorphisms of its analytification~$X^{\an}$ (see \cite[XI, \S 4.0]{SGA43}). There are natural site morphisms $\varepsilon:(X^{\an})_{\cl}\to X_{\et}$ and $\delta:(X^{\an})_{\cl}\to X^{\an}$. 
By \cite[III, Th\'eor\`eme 4.1]{SGA41}, the morphism $\delta_*$ induces an equivalence of topoi. As a consequence, for cohomological purposes, we will not distinguish between $(X^{\an})_{\cl}$ and $X^{\an}$.  If $\LL$ is an \'etale sheaf on $X$, we let~$\LL^{\an}:=\varepsilon^*\LL$ denote its \textit{analytification}. For $k\geq 0$, we consider the comparison morphisms 
\begin{equation}
\label{bcmorphisms}
H^k_{\et}(X,\LL)\to H^k(X^{\an},\LL^{\an}).
\end{equation}

As explained in \cite[\S 5.3]{Stein},  these constructions admit $G$\nobreakdash-equi\-vari\-ant variants. If $S$ is a $G$-equivariant Stein space, there is an analytification functor associating with an $\cO(S)^G$-scheme $X$ locally of finite presentation a $G$-equivariant complex space $X^{\an}$ over $S$.  

\begin{lem}
\label{lemRpoints}
Let $S$ be a $G$-equivariant Stein space. Let $X$ be an $\cO(S)^G$-scheme locally of finite presentation. There is a natural  bijection $(X^{\an})^G\isoto X(\R)$ (where $X(\R)$ is the set of $\R$-points of $X$ viewed as an $\R$-scheme).
\end{lem}

\begin{proof}
The bijection $X^{\an}\isoto X(\C)$ obtained by applying Lemma \ref{lemCpoints} to $X_{\cO(S)}$ is $G$-equivariant. The lemma follows by Galois descent.
\end{proof}

The analytification of an \'etale sheaf $\LL$ on~$X$ is a $G$\nobreakdash-equivariant sheaf $\LL^{\an}$ on $X^{\an}$ and, for $k\geq 0$,  there are comparison morphisms
\begin{equation}
\label{Gbcmorphisms}
H^k_{\et}(X,\LL)\to H^k_G(X^{\an},\LL^{\an}),
\end{equation}
where the right-hand side of (\ref{Gbcmorphisms}) is a $G$-equivariant cohomology group (see \S\ref{parGeq}).

\subsection{Finite \texorpdfstring{$\cO(S)$}{O(S)}-schemes of finite presentation}

If $(S,\cO_S)$ is any ringed space and $\pi:(S,\cO_S)\to(\pt,\cO(S))$ is the natural morphism of ringed spaces,  there is a monoidal
%meaning all arrows implied are monoidal
 adjunction
%in the sense that pi_* is lax but pi^* is then necessarily strong (here because tensor products commute between themselves). 
\begin{equation}
\label{monoidaladj}
 \left\{  \begin{array}{l}
    \textrm{sheaves of } \cO_S\textrm{-modules}
  \end{array}\right\}
    \stackrel[\pi^*]{\pi_*}{\rightleftarrows} 
 \left\{  \begin{array}{l}
    \cO(S)\textrm{-modules}
  \end{array}\right\}.
\end{equation}
A sheaf of  $\cO_S$-modules $\cF$ will be said to be \textit{finitely presented} if there exists a short exact sequence $\cO_S^{\oplus N'}\to\cO_S^{\oplus N}\to\cF\to 0$ for some integers $N,N'\geq 0$.

\begin{prop}
Let $S$ be a Stein space and let $\pi:(S,\cO_S)\to(\pt,\cO(S))$ be as above. The functors $\pi_*$
%R
and $\pi^*$ 
%L
induce a monoidal adjoint equivalence
%MacLane p. 91
 of categories
\begin{equation}
\label{monoidaleq}
 \left\{  \begin{array}{l}
    \textrm{\hspace{1.3em}finitely presented}\\ \textrm{sheaves of } \cO_S\textrm{-modules}
  \end{array}\right\}
    \stackrel[\pi^*]{\pi_*}{\rightleftarrows} 
 \left\{  \begin{array}{l}
    \textrm{finitely presented}\\ \hspace{.6em}\cO(S)\textrm{-modules}
  \end{array}\right\}
\end{equation}
\end{prop}

\begin{proof}
On the one hand, the functor $\cF\mapsto\pi_*(\cF)=H^0(S,\cF)$  is exact on the category of sheaves of finite presentation on $S$, because $S$ is Stein and these sheaves are coherent.  On the other hand, the functor $M\mapsto \pi^*(M)$ is exact on the category of $\cO(S)$\nobreakdash-modules, because the natural ring morphisms $\cO(S)\to\cO_{S,s}$ are flat  for all~$s\in S$ (\eg apply \cite[Lemma 1.8]{Stein} with $K=\{s\}$). It follows that $\pi_*$ and $\pi^*$ preserve the condition of being finitely presented, whence the diagram (\ref{monoidaleq}).

The tensor product of $\cO(S)$-modules, or of sheaves of $\cO_S$-modules,  preserves the condition of being finitely presented \cite[II, \S3.6, Proposition 6]{BourbAlgebre}.  Since (\ref{monoidaladj}) is a monoidal adjunction, we deduce that so is (\ref{monoidaleq}).

Let $\cF$ be a finitely presented sheaf of $\cO_S$-modules. We claim that the counit 
$\pi^*(\pi_*(\cF))\to\cF$ is an isomorphism.  Making use of a presentation of~$\cF$,  it suffices to prove the claim for $\cF=\cO_S$, which reduces us to the obvious fact that ${\pi^*(\cO(S))\isoto\cO_S}$.
Let $M$ be a finitely presented $\cO(S)$-module.  We also claim that the unit $M\to \pi_*(\pi^*(M))$ is an isomorphism.  Using a presentation of $M$, we reduce to the case $M=\cO(S)$, which boils down to the tautology
${\cO(S)\isoto H^0(S,\cO_S)}$.
It follows from these claims that (\ref{monoidaleq}) is an adjoint equivalence of categories. 
\end{proof}

\begin{rem}
\label{remseqcat}
(i) The equivalence (\ref{monoidaleq}) restricts to an equivalence between the categories of finitely presented sheaves of $\cO_S$-modules $\cF$ that are locally free of rank~$r$ and of $\cO(S)$-modules $M$ that are locally free of rank~$r$ (hence of finite presentation by \cite[Lemma 00NX]{SP}).
%loc libre de rg fini = fin gen projective = fin pres flat. 
Clearly,  if $M$ is locally free of rank~$r$, then so is~$\cF$.  
Conversely, if $\cF$ is locally free of rank $r$, then so is $M$ by \cite[S\" atze~6.2 und~6.3]{Forster} (the irreducibility and finite-dimensionality hypotheses in \cite[Satz~6.3]{Forster} are only used to ensure that $M$ is finitely generated; here, it is even finitely presented).  

%(ii) The equivalence (\ref{monoidaleq}) induces an equivalence between the categories of finitely presented Azamuya algebras of degree $d$ over the ringed space $(S,\cO_S)$
%and isomorphisms
%(see \cite[I, \S 2 et Remarque 5.12]{GrothBrauer}), and 
%of Azumaya algebras of degree $d$ 
%Mover $\cO(S)$ (use that a finite locally free algebra $A$ is Azumaya if and only $A\otimes A^{\opp}\isoto\End(A)$),
%C'est justifié dans le Remarque 5.12.
%which respects Morita equivalence.

(ii) If $S$ is finite-dimensional, then any coherent sheaf of $\cO_S$-modules~$\cF$ that is locally free of rank $r$ is finitely presented. Indeed, by~\cite[Theorem~1]{Kripke}, one can find a surjection $\cO_S^{\oplus N}\to\cF$ for some $N\geq 0$. Its kernel is still locally free, of rank~$N-r$, so applying \cite[Theorem~1]{Kripke} again concludes.
\end{rem}

\begin{prop}
\label{eqfinitepres}
Let $S$ be a Stein space.  Then (\ref{monoidaleq}) induces an adjoint equivalence
\begin{equation}
\label{adjeq}
 \left\{  \begin{array}{l}
    \textrm{finite }\cO(S)\textrm{-schemes }X\\ \hspace{.3em}\textrm{of finite presentation}
  \end{array}\right\}
    \rightleftarrows
 \left\{  \begin{array}{l}
    \textrm{\hspace{.2em}finite holomorphic maps $p:T\to S$}\\ \textrm{such that }p_*\cO_T \textrm{ is finitely presented}
  \end{array}\right\}
\end{equation}
whose involved functors are $X\mapsto X^{\an}$ and $T\mapsto\Spec(\cO(T))$.
The $\cO(S)$\nobreakdash-scheme~$X$ is flat (\resp \'etale) if and only if $p$ is flat (\resp a local biholomorphism).
\end{prop}

\begin{proof}
The monoidal adjoint equivalence (\ref{monoidaleq}) induces an adjoint equivalence
\begin{equation}
\label{adjeqbis}
 \left\{  \begin{array}{l}
     \hspace{1.2em}\textrm{finitely presented}\\\textrm{sheaves of $\cO_S$-algebras}
  \end{array}\right\}
    \stackrel[\pi^*]{\pi_*}{\rightleftarrows} 
 \left\{  \begin{array}{l}
    \textrm{finitely presented}\\ \hspace{.7em}\cO(S)\textrm{-algebras}
  \end{array}\right\},
\end{equation}
where finitely presented is meant as sheaves of $\cO_S$-modules and as $\cO(S)$-modules.  The adjoint equivalence (\ref{adjeq}) is obtained by passing to opposite categories, using the spectrum construction $\Spec$ (and that an $\cO(S)$-algebra is finitely presented as an~$\cO(S)$\nobreakdash-module if and only if it is finite of finite presentation; see \cite[Lemmas~\href{https://stacks.math.columbia.edu/tag/0D46}{0D46} and \href{https://stacks.math.columbia.edu/tag/0564}{0564}]{SP}) on one side, and the analytic spectrum construction $\Specan$ (see \cite[Theorem 1.15 b)]{Fischer}) on the other side.

We now identify the right arrow $X\mapsto \Specan(\pi^*\cO(X))$ of (\ref{adjeq}) with ${X\mapsto X^{\an}}$. 
For any holomorphic map $q:S'\to S$ of complex spaces, 
\begin{alignat*}{4}
\label{eqcd}
\Hom_S(S',\Specan(\pi^*\cO(X)))&=\Hom_{\cO_S}(\pi^*\cO(X),q_*\cO_{S'})\\
&=\Hom_{\cO(S)}(\cO(X),\cO(S'))\\
&=\Hom_{\Spec(\cO(S))}(S',X),
\end{alignat*}
where the first equality is the definition of the analytic spectrum given in \cite[\S 1.14]{Fischer},  the second equality is the adjunction between $\pi^*$ and $\pi_*$, and the third one is \cite[Lemma \href{https://stacks.math.columbia.edu/tag/01I1}{01I1}]{SP}. Comparing with (\ref{defan}) concludes.

The direct implications of the last assertion follow from the general properties of the analytification functor.  
If $p$ is flat, then~$p_*\cO_T$ is locally free (of bounded rank as it is finitely presented), and hence the
$\cO(S)$-module~$\cO(T)$ is flat by Remark \ref{remseqcat}~(i).

Assume that~$p$ is a local biholomorphism.  The locus $Z\subset \Spec(\cO(S))$ over which the finite locally free morphism $X\to\Spec(\cO(S))$ (see~\cite[Lemma~\href{https://stacks.math.columbia.edu/tag/02KB}{02KB}]{SP}) is not \'etale has a natural schematic structure locally defined by the vanishing of a single equation (see~\cite[Lemma~\href{https://stacks.math.columbia.edu/tag/0BJF}{0BJF}]{SP}). It is therefore a closed subscheme of finite presentation of~$\Spec(\cO(S))$. Since $Z^{\an}=\varnothing$ (as $p$ is a local biholomorphism), one has~$Z=\varnothing$ by the equivalence~(\ref{adjeq}). It follows that $X$ is an \'etale $\cO(S)$-scheme.
\end{proof}

\begin{cor}
\label{surjO}
Let $S$ be a Stein space. Let $X$ be a finite $\cO(S)$-scheme of finite presentation.  The natural morphism $\cO(X)\to\cO(X^{\an})$ is an isomorphism.
%if only finite type, the morphism is still surjective.
\end{cor}

\begin{proof}
By Proposition \ref{eqfinitepres}, the counit map $\Spec(\cO(X^{\an}))\to X$ of~(\ref{adjeq}) is an isomorphism. Applying the global sections functor shows that $\cO(X)\isoto\cO(X^{\an})$.
\end{proof}

\begin{rem}
\label{remnotfp}
Let $S$ be a reduced countable Stein space.  Let $\km\subset\cO(S)$ be a maximal ideal associated with a nonprincipal ultrafilter on $S$. Define $X:=\Spec(\cO(S)/\km)$. Then $X^{\an}=\varnothing$, so $\cO(S)/\km=\cO(X)\to\cO(X^{\an})$ is not an isomorphism. This shows that the finite presentation hypothesis in Corollary \ref{surjO} cannot be dispensed with.
\end{rem}

\section{Algebraization of closed and open subsets}
\label{secH0}

The goals of this section are Theorem \ref{thmclopen} and its consequence Corollary \ref{compH0}.

\subsection{Rings of holomorphic functions of normal Stein spaces}

\begin{lem}
\label{Steinnormal}
If $S$ is a connected normal Stein space,  then the ring $\cO(S)$ is a normal domain.
\end{lem}

\begin{proof}
Since $S$ is normal, it is locally irreducible, and hence irreducible by connectedness. As $S$ is moreover reduced, it follows that $\cO(S)$ is a domain.

Suppose that $a\in\cO(S)$ and $b\in\cO(S)^*$ are such that $h=\frac{a}{b}\in\Frac(\cO(S))$ satisfies an equation of the form $h^d+\sum_{i=0}^{d-1}c_ih^i=0$ with $c_i\in\cO(S)$.  Then $h\in\cM(S)$ is a~section of the normalization sheaf $\widehat{\cO}_S$ of $S$ (see \cite[6, \S 4.1]{GRCoherent}), hence a section of~$\cO_S$ as $S$ is normal.  The ring $\cO(S)$ is therefore integrally closed in $\Frac(\cO(S))$.
\end{proof}

\begin{lem}
\label{lemprodnormal}
Let $B$ be a finitely generated $\Z$-algebra. There exist a normal finitely generated $\Z$-algebra $B'$ and a ring morphism $v:B\to B'$ such that for any normal domain~$A$, any ring morphism~$u:B\to A$ factors through $v$.
\end{lem}

\begin{proof}
Let $X_1,\dots, X_k$ be reduced schemes such that $X_1=\Spec(B)$, such that~$X_{j+1}$ is equal to the singular locus of $X_j$, and with $X_k$ regular. 
Let~$\wX_j$ be the normalization of $X_j$ and let $B'$ be the coordinate ring of the disjoint union of the~$\wX_j$. Define~$v:B\to B'$ to be the natural morphism.

Let $Z\subset X_1$ be the closure of the image of the morphism ${f:\Spec(A)\to\Spec(B)}$ induced by~$u$.  Choose $j$ maximal with $Z\subset X_{j}$.  Let $\oZ\subset\wX_{j}$ be the strict transform of~$Z$ in $\wX_{j}$.  The morphism $f$ lifts to the normalization of $Z$, hence a fortiori to a morphism ~$\of:\Spec(A)\to\oZ$. Consider the composition~${\Spec(A)\xrightarrow{\of}\oZ\subset \wX_{j}\subset\Spec(B')}$. The ring map $w:B'\to A$ it induces satisfies~$u=w\circ v$.
\end{proof}

\begin{lem}
\label{prodnormal}
 Let $(A_i)_{i\in I}$ be normal domains. Then $\prod_{i\in I} A_i$ is a directed colimit of normal $\Z$-algebras of finite type.
\end{lem}

\begin{proof}
Write $\prod_{i\in I} A_i$ as the filtered colimit of all finitely generated $\Z$-algebras $B$ equipped with a ring morphism to $\prod_{i\in I} A_i$ (see \cite[Lemma \href{https://stacks.math.columbia.edu/tag/0BUF}{0BUF}]{SP}).  We only need to prove that the category of those $B\to\prod_{i\in I} A_i$ with $B$ is normal is cofinal (as one can then replace the filtered colimit by a directed colimit, by \cite[Lemma~\href{https://stacks.math.columbia.edu/tag/0032}{0032}]{SP}).

In view of \cite[Lemma \href{https://stacks.math.columbia.edu/tag/0BUC}{0BUC}]{SP},  it suffices to verify that any $B\xrightarrow{u} \prod_{i\in I} A_i$ with~$B$ normal and finitely generated over $\Z$ can be factorized as ${B\xrightarrow{v} B'\xrightarrow{w} \prod_{i\in I} A_i}$ with~$B'$ normal and finitely generated over $\Z$.  
To do so, let $v:B\to B'$ be as in Lemma~\ref{lemprodnormal} and factor through $v$ each of the maps $B\to A_i$ induced by~$u$.
\end{proof}

\begin{rem}
Not all normal rings are directed colimits of normal noetherian rings (one can check that the ring appearing in \cite[\S \href{https://stacks.math.columbia.edu/tag/0568}{0568}]{SP} is a counterexample).
%see MSE1483236 (see also 033O)
\end{rem}

\begin{cor}
\label{cornormal}
Let $S$ be a normal Stein space. Then $\cO(S)$ is a directed colimit of normal $\Z$-algebras of finite type.
\end{cor}

\begin{proof}
Let $(S_i)_{i\in I}$ be the connected components of $S$. Then $\cO(S)=\prod_{i\in I}\cO(S_i)$, so the corollary follows from Lemmas \ref{Steinnormal} and \ref{prodnormal}.
\end{proof}

\subsection{Zerodivisors on \'etale \texorpdfstring{$\cO(S)$}{O(S)}-schemes}

\begin{lem}
\label{rednoeth}
Let $f:X\to Y$ be a morphism of schemes that is \'etale of finite presentation.  
Assume that $Y$ is a directed limit with affine transition maps of normal noetherian schemes.  If $a_1,a_2\in\cO(X)$ are such that $a_1a_2=0$,  one can write $X=V_1\cup V_2$ as the union of two disjoint open subsets with $a_1|_{V_1}=0$ and~$a_2|_{V_2}=0$.
\end{lem}

\begin{proof}
It follows from \cite[Lemmas \href{https://stacks.math.columbia.edu/tag/01ZM}{01ZM} (1), \href{https://stacks.math.columbia.edu/tag/07RP}{07RP} and \href{https://stacks.math.columbia.edu/tag/01YZ}{01YZ}]{SP} that $X$ is itself a directed limit with affine transition maps of schemes which admit an \'etale morphism of finite presentation to a normal noetherian scheme, and hence which are themselves normal noetherian (by \cite[Lemma \href{https://stacks.math.columbia.edu/tag/034F}{034F}]{SP}).

Consequently, in view of \cite[Lemma \href{https://stacks.math.columbia.edu/tag/01YX}{01YX}]{SP}, there exist a normal noetherian scheme $Z$, a morphism $g:X\to Z$, and $b_1,b_2\in\cO(Z)$ with $a_1=g^*b_1$ and $a_2=g^*b_2$ such that~$b_1b_2=0$. Since $Z$ is normal noetherian, it has finitely many connected components, which are integral.  On each of these, at least one of $b_1$ and $b_2$ vanishes.  It follows that one can write $Z=U_1\cup U_2$ as the union of two disjoint open subsets with $b_1|_{U_1}=0$ and $b_2|_{U_2}=0$. We finally set $V_1:=g^{-1}(U_1)$ and $V_2:=g^{-1}(U_2)$.
\end{proof}

\begin{rem}
Lemma \ref{rednoeth} fails if $Y$ is only assumed to be normal (the spectrum of the ring constructed in \cite[\S \href{https://stacks.math.columbia.edu/tag/0568}{0568}]{SP} can be checked to be a counterexample).
%There exist nonintegral connected normal schemes MSE1483236 (see also 033O) so in normal schemes, connected and irr components are distinct in general.  And normal does not imply extremally disconnected in general.
\end{rem}

The next corollary is an immediate consequence of Corollary \ref{cornormal} and Lemma~\ref{rednoeth}.

\begin{cor}
\label{cordisj}
Let $S$ be a normal Stein space. Let $X$ be an \'etale $\cO(S)$-scheme of finite presentation. If $a_1,a_2\in\cO(X)$ are such that $a_1a_2=0$,  one can write $X=V_1\cup V_2$ as the union of two disjoint open subsets with $a_1|_{V_1}=0$ and~$a_2|_{V_2}=0$.
\end{cor}

\subsection{Closed and open subsets of \'etale \texorpdfstring{$\cO(S)$}{O(S)}-schemes}
\label{clopen}

If $E$ is a topological space, we let $\Clop(E)$ denote the set of all closed and open subsets of~$E$.

\begin{thm}
\label{thmclopen}
Let $S$ be a Stein space.  Let $f:X\to\Spec(\cO(S))$ be an \'etale morphism.
The map $\Clop(X)\to \Clop(X^{\an})$ given by $U\mapsto U^{\an}$ is
\begin{enumerate}[(i)]
\item injective if $S$ is reduced;
\item bijective if $S$ is normal.
\end{enumerate}
\end{thm}

We start with a lemma.

%We say that a Stein space $S$ has \textit{bounded index of nilpotence} if there exists $N\geq 0$ such that for all $s\in S$ and all nilpotent $a\in\cO_{S,s}$, one has $a^N=0$. This property is for instance satisfied if $S$ is reduced.

\begin{lem}
\label{lemred}
Let $S$ be a reduced Stein space and let $f:X\to\Spec(\cO(S))$ be an \'etale morphism.
If $X\neq\varnothing$, then $X^{\an}\neq\varnothing$.
\end{lem}

\begin{proof}
As \'etale morphisms are open (see \cite[Lemma~\href{https://stacks.math.columbia.edu/tag/03WT}{03WT}]{SP}), the image of~$f$ is a nonempty open subset of $\Spec(\cO(S))$. This image therefore contains a nonempty affine open subset $U\subset\Spec(\cO(S))$ of the form $\Spec(\cO(S)[\frac{1}{a}])$ for some~$a\in\cO(S)$. As~$U$ is nonempty,  one has~$a\neq 0$.
% one has $a^N\neq 0$ in $\cO(S)$ for all $N\geq 0$.  
Since $S$ 
%has bounded index of nilpotence
is reduced,  there exists~$s\in S$ with~${a(s)\neq 0}$, so the ideal $\km_s\in \Spec(\cO(S))$ of functions vanishing on $s$ is a point of~$
U$. Any point of $X$ above $\km_s$ (there exists one as $U$ is included in the image of~$f$) corresponds to a point of $X^{\an}$ (see Lemma \ref{lemCpoints}), which is therefore nonempty.
\end{proof}

\begin{proof}[Proof of Theorem \ref{thmclopen}]
Assume that $S$ is reduced and let ${U,V\subset X}$ be closed open subsets such that~$U^{\an}=V^{\an}$. Then $(U\setminus(U\cap V))^{\an}=\varnothing$ and it follows from Lemma~\ref{lemred} that $U\setminus(U\cap V)=\varnothing$. Similarly, one proves that~$V\setminus(U\cap V)=\varnothing$. We deduce that~$U=V$. This proves (i).

From now on, we suppose that $S$ is normal and we prove (ii).  Assume first that~$X$ is affine (hence that  $f$ is of finite presentation). By Zariski's Main Theorem \cite[Lemma \href{https://stacks.math.columbia.edu/tag/0F2N}{0F2N}]{SP}, one can write $f=\of\circ j$, where $j:X\hookrightarrow\oX$ is an open immersion and $\of:\oX\to\Spec(\cO(S))$ is finite of finite presentation.  Set $Z:=\oX\setminus X$. The complex space $\oX^{\an}$ is finite over $S$ and hence Stein.  Suppose that $X^{\an}=U_1\cup U_2$ is the union of two disjoint open subsets. Since $Z^{\an}=\oX^{\an}\setminus X^{\an}$ is a closed analytic subset of $\oX^{\an}$,  the closure $\oU_1$ of $U_1$ in $\oX^{\an}$ is an analytic subset of $\oX^{\an}$ (a union of irreducible components of $\oX^{\an}$). 
 Let $\Sigma_1\subset \oX^{\an}\setminus\oU_1$ be a discrete subset containing exactly one point in each irreducible component of $\oX^{\an}\setminus\oU_1$. Let $\Sigma_2\subset U_1$ be a discrete subset containing exactly one point in each irreducible component of $U_1$.

Since $\oX^{\an}$ is Stein, one can find $a_1\in\cO(\oX^{\an})$ with $a_1=0$ on 
$\oU_1$ and $a_1(x)=1$ for all $x\in\Sigma_1$. Let $\cI\subset\cO_X$ be the annihilator of the image of $\cO_{\oX^{\an}}\xrightarrow{a_1}\cO_{\oX^{\an}}$; it is a coherent sheaf by \cite[Annex, \S 4.5]{GRCoherent}.
Since $\oX^{\an}$ is Stein, one can find~$a_2\in\cI(\oX^{\an})$ such that $a_2(x)=1$ for all $x\in\Sigma_2$. Then $a_1a_2=0$ in $\cO(\oX^{\an})$.

Using Corollary \ref{surjO}, we view $a_1$ and $a_2$ as elements of $\cO(\oX)$ such that ${a_1a_2=0}$.
By Corollary \ref{cordisj}, one can write $X=V_1\cup V_2$ as a union of disjoint open subsets such that $a_1|_{V_1}=0$ and $a_2|_{V_2}=0$.  Then~$X^{\an}$ is the union of its two disjoint open subsets $V_1^{\an}$ and $V_2^{\an}$.  Since $a_1|_{V_1^{\an}}=0$ and $a_2|_{V_2^{\an}}=0$, and since $a_1$ (\resp $a_2$) vanishes identically only on those irreducible component of $X^{\an}$ that are that are included (\resp not included) in $U_1$, one must have $V_1^{\an}=U_1$ and $V_2^{\an}=U_2$. This completes the proof of (ii) when $X$ is affine.

In general,  choose a cover $(X_i)_{i\in I}$ of $X$ by affine open subsets.  If $U\subset X^{\an}$ is closed and open, there exists $V_i\subset X_i$ closed and open such that $V_i^{\an}=X_i^{\an}\cap U$ (by~(ii) in the affine case), and the $V_i$ glue by (i) applied on the $X_i\cap X_j$.
\end{proof}

\begin{rems}
\label{remH0}
(i)
Theorem \ref{thmclopen} (i) fails in general if $S$ is not assumed to be reduced.  To see it, take $S:=\bigsqcup_{n\in\N}\Spec(\C[x]/(x^{n+1}))^{\an}$ (so $\cO(S)=\prod_{n\in\N}\C[x]/(x^{n+1})$) and define $X:=\Spec(\cO(S)[\frac{1}{x}])$. Then $X^{\an}=S\setminus\{x=0\}=\varnothing$. But the scheme $X$ is nonempty since $x\in\cO(S)$ is not nilpotent (as $x^n$ does not vanish in the $n$th factor).

(ii)
Theorem \ref{thmclopen} (ii) fails in general if $S$ is reduced but not normal. To see it, set $S_1=S_2=\C$. Let $S$ be the complex space obtained from $S_1\sqcup S_2$ by gluing, for all $i\geq 1$, the points $i\in S_1$ and $i\in S_2$ with a tangency of order $i-1$, so
$$\cO(S)=\{(f_1,f_2)\in\cO(S_1)\times\cO(S_2)\mid f_1^{(j)}(i)=f_2^{(j)}(i)\textrm{ for } i\geq 1\textrm{ and }0\leq j\leq i-1\}.$$
The complex space $S$ is Stein because so is its normalization $S_1\sqcup S_2$ (see \cite[Theorem 1]{Narasimhannormalization}).  Let $f_1\in\cO(S_1)$ be a holomorphic function vanishing only on the integers $i\geq 1$, at order exactly $1$. Let $f_2\in\cO(S_2)$ be the same function viewed on~$S_2$. Let $f\in\cO(S)$ be the function induced by $(f_1,f_2)$. Set $X:=\Spec(\cO(S)[\frac{1}{f}])$. 

Then $X^{\an}=S\setminus\{f=0\}$ has two connected components. 
  To conclude, we show that~$X$ is connected. Take $g\in\cO(X)$ such that $g^2=g$.  Write $g=\frac{h}{f^N}$ for some~$h\in\cO(S)$ and some $N\geq 1$.  Then $h(h-f^N)=0$ in $\cO(S)$.  
%f nonzerodivisor.
It follows that~$h|_{S_i}$ is either equal to $0$ or to $f_i^N$ (for $i\in\{1,2\}$). If $h|_{S_1}=h|_{S_2}=0$, then $h=0$ and~$g=0$. If $h|_{S_1}=f_1^N$ and $h|_{S_2}=f_2^N$, then $h=f^N$ and $g=1$.  Otherwise, 
one has $h|_{S_1}=0$ and $h|_{S_2}=f_2^N$ (or $h|_{S_1}=f_1^N$ and $h|_{S_2}=0$).  Both cases are absurd since $h|_{S_1}$ and $h|_{S_2}$ then do not coincide at order $i-1$ at the integer $i$ if $i\geq N+1$.

(iii)
The map $\Clop(\Spec(\cO(S)))\to\Clop(S)$ is bijective for any Stein space~$S$, as follows from Proposition \ref{eqfinitepres}.  Alternatively, one can argue that $\Clop(\Spec(\cO(S)))$ and $\Clop(S)$ are both in bijection with the set of~${a\in\cO(S)}$ such that $a^2=a$.
This contrasts with the examples of (i) and (ii).

(iv)
It is important to formulate Theorem \ref{thmclopen} (ii) in terms of closed and open subsets, and not in terms of connected components. Indeed, if~$S$ is a reduced countable Stein space and $X=\Spec(\cO(S))$, the map $S=X^{\an}\to X$ does not induce a surjection on sets of connected components (maximal ideals of~$\cO(S)$ associated with nonprincipal ultrafilters on $S$ are not in the image).
\end{rems}

Recall that there is a morphism of sites $\varepsilon :S_{\cl}\to \Spec(\cO(S))_{\et}$ (see \S\ref{paranal}).

\begin{cor}
\label{compH0}
Let $S$ be a Stein space.  Fix $m\geq 1$. The natural morphism $\Z/m\to\varepsilon_*\Z/m$ of \'etale sheaves on $\Spec(\cO(S))$ is
\begin{enumerate}[(i)]
\item injective if $S$ is reduced;
\item an isomorphism if $S$ is normal.
\end{enumerate}
\end{cor}

\begin{proof}
Evaluating the morphism $\Z/m\to\varepsilon_*\Z/m$ on an \'etale $\cO(S)$-scheme $X$ gives rise to the pull-back morphism
\begin{equation}
\label{evH0}
H^0_{\et}(X,\Z/m)\xrightarrow{\varepsilon^*} H^0(X^{\an},\Z/m).
\end{equation}
 The right-hand side of (\ref{evH0}) coincides with the set of ordered partitions of $X^{\an}$ into~$m$ closed and open subsets.  The left-hand side of (\ref{evH0}) coincides with the set of ordered partitions of~$X$ into $m$ closed and open subsets (as the Zariski sheaf~$\Z/m$ is already an \'etale sheaf, by \'etale descent).  Both sides coincide by Theorem \ref{thmclopen}.
\end{proof}

\section{Generic algebraization of topological coverings}
\label{secH1}

Let $S$ be a normal Stein space. Let $X$ an \'etale $\cO(S)$-scheme of finite presentation.
The analytification functor does not induce an equivalence of categories between degree $d$ finite \'etale covers of $X$ and degree $d$ topological coverings of~$X^{\an}$ (see Remark~\ref{remH1} (i)). Together,  Corollary \ref{cormorphcov} and Theorem \ref{thmconstr} show that this however holds in the limit where one inverts more and more non\-zero\-di\-vi\-sors in~$\cO(S)$.

\subsection{Descending holomorphic functions along analytic coverings}

\begin{lem}
\label{lemconductor}
Let $p:T\to S$ be an analytic covering of degree $1$ between reduced Stein spaces. There exists a non\-zero\-di\-vi\-sor $b\in\cO(S)$ such that 
$b\cdot\cO(T)\subset\cO(S)$.
\end{lem}

\begin{proof}
Let $\cI\subset\cO_S$ be the annihilator of the cokernel of the morphism $\cO_S\to p_*\cO_T$ induced by pull-back. It is a coherent ideal sheaf (see \cite[Annex, \S 4.5]{GRCoherent}).  Since~$p$ has degree $1$, the complex subspace of $S$ defined by $\cI$ is nowhere dense in $S$.  It suffices to choose $b$ to be an element of $\cI(S)$ that is not identically zero on any irreducible component of $S$ (such elements exist because $S$ is Stein).
\end{proof}

\begin{lem}
\label{multdescend}
Let $p:T\to S$ be a surjective
%not needed, allows to view O(S) as subring of O(T).
analytic covering of reduced Stein spaces. For all non\-zero\-di\-vi\-sor
%false if not (take T=S\cup S and b sero on one component)
 $b\in\cO(T)$, there exists a non\-zero\-di\-vi\-sor $a\in\cO(T)$ with $ab\in\cO(S)$.
\end{lem}

\begin{proof}
Making use of Lemma \ref{lemconductor} applied to the normalization morphisms of $S$ and~$T$,  one can reduce to the case where both $S$ and $T$ are normal. One can then further assume that $S$ and $T$ are connected.
%for T not immediate: change the a's by something coming from S.
 Let $F$ be the Galois closure of~$\cM(T)$ over $\cM(S)$.  Let ${q:\whT\to T}$ be the surjective analytic covering of normal Stein spaces associated with the field extension~$\cM(T)\subset F$ (see~(\ref{eqext})).  By functoriality of~(\ref{eqext}), the Galois group~$\Gamma:=\Gal(F/\cM(S))$ acts on~$\whT$. Consider the ele\-ments~${c:=\prod_{\gamma\in\Gamma}\gamma^*q^*b}$ and~$a:=\frac{c}{q^*b}$ of $\cO(\whT)$. One has $c\in \cO(\whT)\cap\cM(S)=\cO(S)$ by normality of $S$ and similarly $a\in\cO(\whT)\cap\cM(T)=\cO(T)$ by normality of $T$.
\end{proof}

\subsection{Quasi-finite \texorpdfstring{$\cO(S)$}{O(S)}-schemes}
\label{parqf}

We recall (for use here and later in \S\ref{parbasechange}) that a ring~$F$ is said to be \textit{absolutely flat} if all $F$-modules are flat. 
%on dit aussi von Neumann regular.
We refer to \cite[Proposition~4.41]{Chromatic} for a list of equivalent properties and references.
%ce sont aussi les anneaux réduits satisfaisant les conditions équivalentes de https://stacks.math.columbia.edu/tag/04MG.
Any product of fields is absolutely flat, as one easily checks using the characterization \cite[Proposition~4.41~(4)]{Chromatic}. This applies to the ring $F=\cM(S)=\prod_{i\in I} \cM(S_i)$ of meromorphic functions on a reduced Stein space $S$ with irreducible components $(S_i)_{i\in I}$.

\begin{lem}
\label{finiteF}
Any quasi-separated and quasi-finite scheme $X$ over an absolutely flat ring $F$ is finite.
%quasi-separated required: in the spectrum of a product of fields, double a closed point associated with nonprincipal ultrafilter.
\end{lem}

\begin{proof}
We prove that $X$ is proper over $F$ using the valuative criterion of properness \cite[Lemma \href{https://stacks.math.columbia.edu/tag/0BX5}{0BX5}]{SP}. As all the local rings of $F$ are fields (see \cite[Lemma~\href{https://stacks.math.columbia.edu/tag/092F}{092F}]{SP}), this reduces us to the case where $F$ is a field, which is well-known.
\end{proof}

\begin{lem}
\label{finite}
Let $S$ be a reduced Stein space. Let $X$ be a quasi-finite $\cO(S)$\nobreakdash-scheme of finite presentation. Then there exists a non\-zero\-di\-vi\-sor~${a\in\cO(S)}$ such that the $\cO(S)[\frac{1}{a}]$-scheme $X_{\cO(S)[\frac{1}{a}]}$ is finite.
\end{lem}

\begin{proof}
As $\cM(S)$ is absolutely flat, the $\cM(S)$-scheme $X_{\cM(S)}$ is finite by Lemma~\ref{finiteF} (note that $X_{\cM(S)}$ is quasi-separated because it is of finite presentation). 
%again quasi-separated required: otherwise problem with double exotic point.
The lemma therefore follows from a limit argument (see \cite[Lemma \href{https://stacks.math.columbia.edu/tag/01ZO}{01ZO}]{SP}).
\end{proof}

\subsection{From analytic coverings to finite \'etale covers}

\begin{lem}
\label{lemfiniteflat}
Let $p:T\to S$ be an analytic covering of degree $d$ between reduced Stein spaces.
There exists a morphism $f:X\to\Spec(\cO(S))$ such that:
\begin{enumerate}[(i)]
\item $f$ is of finite presentation and finite flat of degree $d$;
\item $f^{\an}:X^{\an}\to S$ is an analytic covering of reduced Stein spaces;
\item the $\cM(S)$-algebras $\cM(X^{\an})$ and $\cM(T)$ are isomorphic;
\item there is a non\-zero\-di\-vi\-sor $c\in\cO(S)$ such that~$f$ is \'etale over $\Spec(\cO(S)[\frac{1}{c}])$.
\end{enumerate}
\end{lem}

\begin{proof}
Let $(S_i)_{i\in I}$ be the irreducible components of $S$. Let $T_i$ be the disjoint union of the irreducible components of $T$ with image $S_i$. The $\cM(S_i)$-algebra $\cM(T_i)$ is finite \'etale of degree $d$  (see (\ref{eqext})).  One can therefore write $\cM(T_i)=\cM(S_i)[t]/(P_i)$ for some monic degree~$d$ polynomial~$P_i\in\cM(S_i)[t]$.  Set $P:=(P_i)_{i\in I}\in\cM(S)[t]$. Then $\cM(T)=\prod_{i\in I}\cM(T_i)=\prod_{i\in I}\cM(S_i)[t]/(P_i)=\cM(S)[t]/(P)$.
Let $a\in\cO(S)$ be a non\-zero\-di\-vi\-sor with $aP\in\cO(S)[t]$. Replacing $P$ with $t\mapsto a^dP(t/a)$, one can ensure that $P\in \cO(S)[t]$ is monic of degree $d$ such that $\cM(T)=\cM(S)[t]/(P)$.

Set $X:=\Spec(\cO(S)[t]/(P))$. It is an $\cO(S)$-scheme of finite presentation that is finite flat of degree $d$.
As $\cO(X)=\cO(S)[t]/(P)\to\cM(S)[t]/(P)=\cM(T)$ is injective, the scheme $X$ is reduced (and hence so is the complex space $X^{\an}$). Since~$X^{\an}$ is moreover flat over~$S$, the holomorphic map $f^{\an}:X^{\an}\to S$ is an analytic covering. 

As non\-zero\-di\-vi\-sors in~$\cO(S)$ remain non\-zero\-di\-vi\-sors in $\cM(T)$, they are also non\-zero\-di\-vi\-sors in $\cO(X)$. It follows that the total ring of fractions of $\cO(X)$ is $\cO(X_{\cM(S)})=\cM(T)$.
By Corollary \ref{surjO}, one has $\cO(X^{\an})=\cO(X)$. Passing to total rings of fractions, one gets
$\cM(X^{\an})=\cM(T)$ (as $\cM(S)$-algebras).  

Finally, to prove (iv), one can choose $c\in\cO(S)$ to be the discriminant of~$P$.
\end{proof}

\begin{lem}
\label{lemgenfinietaled}
Let $p:T\to S$ be an analytic covering of degree $d$ between reduced Stein spaces. There exists a non\-zero\-di\-vi\-sor $a\in\cO(S)$ such that $\cO(T)[\frac{1}{a}]$ is a finite \'etale $\cO(S)[\frac{1}{a}]$\nobreakdash-algebra of degree $d$.
\end{lem}

\begin{proof}
Let $f:X\to\Spec(\cO(S))$ and $c\in\cO(S)$ be as in Lemma \ref{lemfiniteflat}. Since the $\cM(S)$\nobreakdash-algebras $\cM(X^{\an})$ and $\cM(T)$ are isomorphic, the normalization~$\wT$ of $T$ is also the normalization of $X^{\an}$ (apply (\ref{eqext}) over each irreducible component of $S$).
By Lemma \ref{lemconductor} applied to the normalization morphisms~$\wT\to T$ and~${\wT\to X^{\an}}$, there exist non\-zero\-di\-vi\-sors ${b\in\cO(X^{\an})}$ and $b'\in\cO(T)$ such that ${\cO(X^{\an})[\frac{1}{b}]=\cO(\wT)[\frac{1}{b}]}$ and ${\cO(T)[\frac{1}{b'}]=\cO(\wT)[\frac{1}{b'}]}$. By Lemma \ref{multdescend}, we may assume, after multiplying them by appropriate non\-zero\-di\-vi\-sors, that $b$ and $b'$ belong to~$\cO(S)$.
Set $a:=b\, b'c\in\cO(S)$. One has $\cO(T)[\frac{1}{a}]=\cO(\wT)[\frac{1}{a}]=\cO(X^{\an})[\frac{1}{a}]=\cO(X)[\frac{1}{a}]$, where the last equality is Corollary \ref{surjO}. This $\cO(S)[\frac{1}{a}]$-algebra is finite \'etale of degree~$d$ by our choice of $c$.
\end{proof}

\begin{lem}
\label{lemgenfinietale}
Let $p:T\to S$ be an analytic covering of bounded degree between reduced Stein spaces. There exists a non\-zero\-di\-vi\-sor $a\in\cO(S)$ such that $\cO(T)[\frac{1}{a}]$ is a finite \'etale $\cO(S)[\frac{1}{a}]$\nobreakdash-algebra.
\end{lem}

\begin{proof}
Let $\tp:\wT\to\wS$ be the induced analytic covering between the normalizations~$\wT$ and $\wS$ of $T$ and $S$.  Let $\tp_d:\wS_d\to\wT_d$ be the restriction of $\tp$ over the locus where it has degree $d$. Applying Lemma \ref{lemgenfinietaled} to the finitely many ${\tp_d:\wS_d\to\wT_d}$ with~$\wT_d\neq\varnothing$ shows the existence of a non\-zero\-di\-vi\-sor $\ta\in\cO(\wS)$ such that~
$\cO(\wT)[\frac{1}{\ta}]$ is a finite \'etale $\cO(\wS)[\frac{1}{\ta}]$\nobreakdash-algebra.
Apply Lemma~\ref{lemconductor} to get non\-zero\-di\-vi\-sors $b\in\cO(S)$ and~$b'\in\cO(T)$ with 
${\cO(S)[\frac{1}{b}]=\cO(\wS)[\frac{1}{b}]}$ and ${\cO(T)[\frac{1}{b'}]=\cO(\wT)[\frac{1}{b'}]}$.
By Lemma~\ref{multdescend}, we may assume, after multiplying them by non\-zero\-di\-vi\-sors, that $\ta$ and $b'$ belong to $\cO(S)$.
% and~${b'\in\cO(S)}$. 
Set $a:=\ta\, b\, b'\in\cO(S)$.  Then $\cO(T)[\frac{1}{a}]=\cO(\wT)[\frac{1}{a}]$ is a finite \'etale 
$\cO(S)[\frac{1}{a}]=\cO(\wS)[\frac{1}{a}]$-algebra by our choice of~$\ta$.
\end{proof}

\subsection{Morphisms of coverings}

\begin{prop}
\label{morphcov}
Let $S$ be a normal Stein space. Let $X$ be an \'etale $\cO(S)$\nobreakdash-scheme.  Let $Y\to X$ and $Z\to X$ be two finite \'etale morphisms. Then the analytifi\-ca\-tion functor induces a bijection
$$\Hom_X(Y,Z)\to \Hom_{X^{\an}}(Y^{\an},Z^{\an}).$$
\end{prop}

\begin{proof}
The set $\Hom_X(Y,Z)$ identifies with the set of sections of the finite \'etale morphism $Y\times _X Z\to Y$ and hence with the set of closed and open subsets of~$Y\times _X Z$ that project isomorphically to $Y$.  Similarly,  the set $\Hom_{X^{\an}}(Y^{\an},Z^{\an})$ identifies with the set of sections of the topological covering $Y^{\an}\times_{X^{\an}}Z^{\an}\to Y^{\an}$ and hence with the set of closed and open subsets of $Y^{\an}\times_{X^{\an}}Z^{\an}$ that project isomorphically to $Y^{\an}$.  The result therefore follows from Theorem \ref{thmclopen}.
\end{proof}

\begin{cor}
\label{cormorphcov}
Let $S$ be a reduced Stein space. There is a non\-zero\-di\-vi\-sor~${a\in\cO(S)}$ such that for any \'etale $\cO(S)$\nobreakdash-scheme $X$ and any finite \'etale morphisms~$Y\to X$ and~$Z\to X$, the analytification functor induces a bijection
$$\Hom_{X_{\cO(S)[\frac{1}{a}]}}(Y_{\cO(S)[\frac{1}{a}]},Z_{\cO(S)[\frac{1}{a}]})\to \Hom_{X^{\an}\setminus\{a=0\}}(Y^{\an}\setminus\{a=0\},Z^{\an}\setminus\{a=0\}).$$
\end{cor}

\begin{proof}
Let $p:\wS\to S$ be the normalization morphism. By Lemma \ref{lemgenfinietaled}, there exists a non\-zero\-di\-vi\-sor $a\in\cO(S)$ such that $\cO(S)[\frac{1}{a}]=\cO(\wS)[\frac{1}{a}]$.  The corollary follows from Proposition \ref{morphcov} after replacing $S$ with $\wS$ and $X$ with $X_{\cO(S)[\frac{1}{a}]}$.
\end{proof}

\subsection{Construction of coverings}

\begin{lem}
\label{lemproduct}
Let $(A_i)_{i\in I}$ be rings. Set $A:=\prod_{i\in I}A_i$.  Let $M$ be a finitely presented $A$-module. The natural map $M\to\prod_{i\in I}M\otimes_A A_i$ is an isomorphism of $A$\nobreakdash-modules.
\end{lem}

\begin{proof}
This follows from \cite[Proposition \href{https://stacks.math.columbia.edu/tag/059K}{059K} (1)$\Rightarrow$(2)]{SP} applied with $R:=A$.
\end{proof}

\begin{lem}
\label{etalegeneric}
Let $S$ be a reduced Stein space and let $X$ be an
\'etale $\cO(S)$\nobreakdash-scheme of finite presentation.
%quasi-separated important (do not want doubled exotic closed point).
 Then there exist a non\-zero\-di\-vi\-sor $c\in\cO(S)$ and an analytic covering 
$q:\whS\to S$ of reduced Stein spaces such that
$X_{\cO(S)[\frac{1}{c}]}$ and $\Spec(\cO(\whS)[\frac{1}{c}])$
are isomorphic finite \'etale $\cO(S)[\frac{1}{c}]$-schemes.
\end{lem}

\begin{proof}
By Lemma \ref{finite}, we may assume, after possibly replacing $X$ with $X_{\cO(S)[\frac{1}{a}]}$ for some non\-zero\-di\-vi\-sor $a\in\cO(S)$, that $X$ is affine.

Let $(S_i)_{i\in I}$ be the irreducible components of $S$. Then $\cM(S)=\prod_{i\in I}\cM(S_i)$. 
Let~$d_i$ be the degree of the \'etale $\cM(S_i)$-algebra $\cO(X_{\cM(S_i)})$. 
Let $q_i:\whS_i\to S_i$ be the analytic covering of degree $d_i$ with $\whS_i$ normal associated with it (see~(\ref{eqext})).  The~$d_i$ are bounded by \cite[Lemma \href{https://stacks.math.columbia.edu/tag/03JA}{03JA}]{SP}. Let $\whS$ be the disjoint union of the~$\whS_i$. Let ${q:\whS\to S}$ be the induced analytic covering of bounded degree. By Lemma~\ref{lemgenfinietale}, there is a non\-zero\-di\-vi\-sor $a\in\cO(S)$ such that $\cO(\whS)[\frac{1}{a}]$ is a finite \'etale~$\cO(S)[\frac{1}{a}]$\nobreakdash-al\-ge\-bra.
There are isomorphisms of $\cM(S)$-algebras
$$\cO(X_{\cM(S)})=\prod_{i\in I} \cO(X_{\cM(S_i)})=\prod_{i\in I}\cM(\whS_i)=\cM(\whS)=\cO(\whS)\otimes_{\cO(S)}\cM(S),$$
where the first equality results from Lemma \ref{lemproduct}
$$\cO(X)[\frac{1}{a}]\otimes_{\cO(S)[\frac{1}{a}]}\cM(S)=\cO(\whS)[\frac{1}{a}]\otimes_{\cO(S)[\frac{1}{a}]}\cM(S).$$
A limit argument (see \cite[Lemmas \href{https://stacks.math.columbia.edu/tag/01ZM}{01ZM} (2) and \href{https://stacks.math.columbia.edu/tag/081E}{081E}]{SP}) shows the existence of a non\-zero\-di\-vi\-sor $b\in\cO(S)$ such that $\cO(X)[\frac{1}{ab}]=\cO(\whS)[\frac{1}{ab}]$. Take~$c:=ab$.
\end{proof}

\begin{thm}
\label{thmconstr}
Let $S$ be a reduced Stein space.  Let $X$ be an \'etale $\cO(S)$-scheme of finite presentation.
%important otherwise might need to invert infinitely many a's.
 Let $p:T\to X^{\an}$ be a topological covering of degree $d$. There is a non\-zero\-di\-vi\-sor $a\in\cO(S)$ and a finite \'etale covering $f:Y\to X_{\cO(S)[\frac{1}{a}]}$ of degree $d$
such that $f^{\an}:Y^{\an}\to X^{\an}\setminus\{a=0\}$ and ${p|_{T\setminus\{a=0\}}:T\setminus\{a=0\}\to X^{\an}\setminus\{a=0\}}$ are isomorphic topological coverings.
\end{thm}
 
\begin{proof}
We may assume that $X$ is affine (cover $X$ by finitely many affine open subsets, apply the theorem to each of them,  and glue the resulting finite \'etale coverings thanks to Corollary \ref{cormorphcov}).

Let $c\in\cO(S)$ and $q:\whS\to S$ be as in Lemma \ref{etalegeneric}. After replacing $S$ with~$\whS$ (which is legitimate by Lemma \ref{lemBing}), the scheme $X$ with $X_{\cO(S)[\frac{1}{c}]}=\Spec(\cO(\whS)[\frac{1}{c}])$ 
and $T$ with $T\setminus\{c=0\}$, we may assume that $X=\Spec(\cO(S)[\frac{1}{c}])$ for some non\-zero\-di\-vi\-sor $c\in\cO(S)$. After further multiplying $c$ by a non\-zero\-di\-vi\-sor in $\cO(S)$,  we may assume that $X^{\an}=S\setminus\{c=0\}$ is normal.
From now on, we assume that we are in this situation (\ie that $X=\Spec(\cO(S)[\frac{1}{c}])$ and $X^{\an}$ is normal).

The topological covering $p:T\to X^{\an}$ extends to an analytic covering~${\op:\oT\to S}$ of degree $d$ with $\oT$ normal (apply the Grauert--Remmert theorem \cite[XII, Th\'eor\`eme 5.4]{SGA1} on the normalization of $S$).
By Lemma \ref{lemgenfinietaled},  there is a non\-zero\-di\-vi\-sor~$a\in\cO(S)$ such that $\cO(\oT)[\frac{1}{a}]$ is a degree $d$ finite \'etale $\cO(S)[\frac{1}{a}]$\nobreakdash-algebra.

One can then set $Y:=\Spec(\cO(\oT)[\frac{1}{ac}])$ to conclude (indeed, $Y^{\an}=\oT\setminus\{ac=0\}$ by Lemma \ref{lemBing}, so $Y^{\an}=T\setminus\{a=0\}$).
\end{proof}

\begin{rems}
\label{remH1}
(i)
One cannot always take $a=1$ in Theorem \ref{thmconstr}, even for $S$ normal and connected, as we now show.
For $i\geq 1$,  let $p_i:T_i\to S_i$ be the degree~$2$ analytic covering with ${S_i:=\{(w,x,y,z)\in\C^4\mid xy-z^2=w^i-x-y=0\}}$ and ${T_i:=\{(u,v,w)\in\C^3\mid w^i=u^2+v^2\}}$ given by $p_i(u,v,w)=(w,u^2,v^2,uv)$.
The map $p_i$ is unramified over~$S_i\setminus\{s_i\}$, where $s_i\in S_i$ is the origin.
One computes that the Fitting ideal $I_i:=\Fitt_3((p_{i*}\cO_{T_i})_{s_i})\subset \cO_{S_i,s_i}$ (see \cite[Definition 20.4]{Eisenbud}) is~$\langle x,y,z \rangle$.
%O(S)^2->O(S)^3->O(T)->0 où la seconde flèche est (1,u,v) et la première les relations xv=zu et yu=zv.  On regarde la matrice des mineurs de taille 1, qui sont x, y, et z.
It follows that~$w\in \cO(S_i)$ is nilpotent of order exactly $i$ in~$\cO_{S_i,s_i}/I_i$.

Gluing open neighborhoods of $s_i$ in $S_i$ appropriately yields a degree $2$ analytic covering $p:T\to S$ with $S$ and $T$ Stein, normal and connected, which is unramified over~$S\setminus\{s_1,s_2,\dots\}$.  
(One way to ensure both connectedness and Steinness of $S$, hence of $T$, is to rather glue the real-analytic spaces defined by the same equations as the $S_i$, and to let $S$ be a Stein complexification of the resulting real-analytic space; see \cite[III, Theorem 3.6]{GMT}.)
The formula we gave for~$p_i$ implies that the fibers of $p_*\cO_T$ have dimension $\leq 3$. It therefore follows from \cite[Theorem 1]{Kripke} that~$\cO(T)$ is a finite $\cO(S)$-module.  Consider the Fitting ideal ${I:=\Fitt_3(\cO(T))\subset \cO(S)}$.
By \cite[Corollary 20.5]{Eisenbud}, one has $I_i=I\cdot\cO_{S,s_i}$.  Consequently, our choices imply that there exists $b\in\cO(S)$ whose image in~$\cO_{S,s_i}/I_i$ is nilpotent of order exactly $i$,  and hence whose image in $\cO(S)/I$ is not nilpotent.

Set $X:=\Spec(\cO(S)[\frac{1}{b}])$.  We claim that there does not exist a finite \'etale covering $f:Y\to X$ such that $f^{\an}$ is isomorphic to $p|_{T\setminus\{b=0\}}$.  Assume otherwise. 
By Lemma \ref{Steinnormal},  the rings $\cO(S)$ and $\cO(T)$ are integrally closed in $\cM(S)$ and $\cM(T)$ respectively, so the integral closure of $\cO(S)$ in $\cM(T)$ is $\cO(T)$, 
%we proved above that O(T) is a finite O(S)-module
and hence the integral closure of $\cO(S)[\frac{1}{b}]$ in $\cM(T)$ is $\cO(T)[\frac{1}{b}]$.
By \cite[Lemma \href{https://stacks.math.columbia.edu/tag/03GE}{03GE}]{SP}, the integral closure of $\cO(S)[\frac{1}{b}]$ in $\Frac(\cO(Y))$ is $\cO(Y)$.
As $\Frac(\cO(Y))=\cM(T)$ (use Lemma~\ref{lemgenfinietale} to choose a non\-zero\-di\-vi\-sor $a\in\cO(S)$ such that $\cO(T)[\frac{1}{a}]$ is a finite \'etale~$\cO(S)[\frac{1}{a}]$\nobreakdash-algebra and apply Corollary \ref{cormorphcov}), we deduce that~$\cO(Y)=\cO(T)[\frac{1}{b}]$. 

As $f$ is \'etale, the $\cO(S)[\frac{1}{b}]$-module $\cO(Y)$ is flat of finite presentation, hence locally free (see \cite[Lemma \href{https://stacks.math.columbia.edu/tag/00NX}{00NX}]{SP}).
However, the finite $\cO(S)[\frac{1}{b}]$-module $\cO(T)[\frac{1}{b}]$ is not locally free (because~$b$ is not nilpotent in $\cO(S)/I$).  This is a contradiction.

(ii)
In constrast with (i), it is always possible to take $a=1$ in Theorem \ref{thmconstr} when $X=\Spec(\cO(S))$ and $S$ is finite-dimensional. To see it, let $p:T\to S$ be a topological covering of degree $d$. Endow $T$ with a structure of complex space, so that $q$ is a local biholomorphism.  Then $p_*\cO_T$ is locally free of rank $d$, hence finitely presented by Remark \ref{remseqcat} (ii). The finite \'etale $\cO(S)$-scheme $Y$ associated with~${p:T\to S}$ by the equivalence of Proposition \ref{eqfinitepres} then satisfies $Y^{\an}=T$.

(iii)
The hypothesis that $S$ is finite-dimensional in (ii) cannot be removed. To see it, let $\cL_i$ be a $2$-torsion holomorphic line bundle on a Stein space $S_i$ such that~$\cL_i$ cannot be generated by $\leq i$ global sections (\eg take $S_i$ to be a Grauert tube of the real-analytic variety $\P^{2i}(\R)$, choose $\cL_i$ so that its first Chern class $c_1(\cL_i)$ is the generator of $H^2(\P^{2i}(\R),\Z)=\Z/2$, and note that $c_1(\cL_i)^i\neq 0$). 
Let $S$ be the disjoint union of the $(S_i)_{i\geq 1}$ and let $\cL$ be the holomorphic line bundle on $S$ induced by the~$(\cL_i)_{i\geq 1}$.  Set $T:=\Specan(\cO_S\oplus\cL)$, where the algebra structure on~$\cO_S\oplus\cL$ is induced by an isomorphism $\cL^{\otimes 2}\isoto\cO_S$ (see \cite[Theorem~1.15~b)]{Fischer} for the analytic spectrum construction). The structural morphism $p:T\to S$ is a topological double cover. As $p_*\cO_T=\cO_S\oplus\cL$ is not generated by finitely many global sections, it is not finitely presented. It follows from Proposition \ref{eqfinitepres} that $T$ is not of the form $Y^{\an}$ for any finite \'etale $\cO(S)$-scheme $Y$.
\end{rems}

\section{Degree \texorpdfstring{$2$}{2} cohomology classes on Stein surfaces}
\label{secH2}

In this section, we study degree $2$ cohomology classes on Zariski-open subsets of Stein surfaces. We develop techniques to kill them on appropriate alterations.   Our main goals are Proposition \ref{killram} (as well as its variants Propositions~\ref{killram2} and~\ref{killramab}) and Proposition \ref{GLefschetz11}.
To be able to apply our results in the proof of Theorem \ref{thpi}, we need to work $G$-equivariantly throughout (\S\S\ref{parGeq}--\ref{parGThom} are devoted to generalities concerning $G$-equivariant cohomology). 
Easier non-$G$-equivariant analogues of our results (which we sometimes point out explicitly, see \eg Corollaries \ref{corkillram} and \ref{corGLefschetz11}) are obtained by applying them formally to $G$-equivariant complex spaces of the form $S\sqcup S^{\sigma}$ (often making use of (\ref{cohoRC})) or by disregarding $G$-actions in the proofs.

\subsection{Generalities on \texorpdfstring{$G$}{G}-equivariant cohomology}
\label{parGeq}

Let $E$ be a topological space endowed with a continuous action of $G$. Let $\F$
 be a $G$-equivariant sheaf on $E$. We let~$H^k_G(E,\F)$ denote the $G$-equivariant cohomology groups of $\F$ (the derived functors of $\F\mapsto H^0(E,\F)^G$). 
We denote by $H^k_G(E,\F)_0\subset H^k_G(E,\F)$ the subset of those classes $\alpha\in H^k_G(E,\F)$ such that $\alpha|_x=0$ for all $x\in E^G$.
The Hochschild--Serre spectral sequence (the second spectral sequence of \cite[Th\'eor\`eme 5.2.1]{Tohoku}) reads
\begin{equation}
\label{HSss}
H^p(G, H^q(E,\F))\implies H^{p+q}_G(E,\F).
\end{equation}
%Applying (\ref{HSss}) to $\F[G]:=\F\otimes_{\Z}\Z[G]$ (as in \cite[(1.5)]{BW1}) yields isomorphisms
%\begin{equation}
%\label{iso[G]}
%H^k_G(E,\F[G])\isoto H^k(E,\F).
%\end{equation}

If $E'\subset E$ is a closed $G$-invariant subset, we also consider the $G$\nobreakdash-equi\-vari\-ant cohomology groups with support $H^k_{E',G}(E,\F)$ (the derived functors of $\F\mapsto H^0_{E'}(E,\F)^G$) and the $G$-equivariant relative cohomology groups $H^k_G(E,E',\F):=H^k_G(E,j_!j^*\F)$ (where $j:E\setminus E'\hookrightarrow E$ is the inclusion and $j_!$ denotes the extension by zero).

Let $\pi:E\to E/G$ be the quotient map.  Let $j:E\setminus E^G\hookrightarrow E$ be the inclusion. Consider the sheaf $\G:=(\pi_*\F)^G$ on $E/G$ induced by $\F$.  Since the stalks of $j_!j^*\F$ along $E^G$ vanish, the first spectral sequence of \cite[Th\'eor\`eme 5.2.1]{Tohoku} applied to~$j_!j^*\F$ degenerates and yields isomorphisms
\begin{equation}
\label{restrfp}
H^k_G(E,E^G,\F)\isoto H^k(E/G,E^G,\G).
\end{equation}

We still denote by $A$ the constant $G$-equivariant sheaf associated with a $G$\nobreakdash-mod\-ule~$A$.
 Let $\Z(j)$ be the $G$-module which is isomorphic to $\Z$ as an abelian group, and on which $\sigma\in G$ acts by multiplication by $(-1)^j$.  
We set~$\F(j):=\F\otimes_{\Z}\Z(j)$. 

If~$S$ is a complex space, we highlight the following particular case of (\ref{restrfp}):
\begin{equation}
\label{cohoRC}
H^k_G(S\sqcup S^{\sigma},A)\isoto H^k(S,A).
\end{equation}

\subsection{\texorpdfstring{$G$}{G}-equivariant Thom isomorphisms}
\label{parGThom}

The construction of (possibly $G$-equi\-vari\-ant) Thom isomorphisms is explained in \cite[\S\S 1.1.4-1.1.5]{BW1} in a slightly different context. Here is how this works in our setting.  We consider simultaneously the non-$G$-equivariant and the $G$-equivariant cases.

Let $S$ and $S'$ be ($G$-equivariant) complex manifolds of dimensions $n$ and $n'$. Let $f:S'\to S$ be a ($G$-equivariant) embedding. As~$\oor_S=\Z(n)$ and~${\oor_{S'}=\Z(n')}$ ($G$\nobreakdash-equivariantly because complex conjugation acts antiholomorphically),  one has ${\oor_{S'/S}:=\Homrond(f^*\oor_S,\oor_{S'})=\Z(-c)}$.
Let~$\F$ be a ($G$-equivariant) locally constant sheaf on~$S$.   
Applying \cite[Re\-mark~3.3.5, Pro\-po\-si\-tion~3.1.11]{KS} yields a morphism
\begin{equation}
\label{duality}
f^*\F(-c)[-2c]=f^*\F\otimes_{\Z}\oor_{S'/S}[-2c]=f^*\F\otimes_{\Z}f^!\Z \to f^!\F
\end{equation}
which is canonical (and hence $G$-equivariant). We claim that (\ref{duality}) is an isomorphism. 
Working locally, we may assume that~$\F$ is the constant sheaf associated with the abelian group $A$. Writing $A$ as a filtered direct limit of finitely generated groups, we reduce to the case where $A$ is finitely generated, so we may assume that~$A=\Z$ or $A=\Z/m$.  We further reduce to the case~$A=\Z$ by the five lemma. In this case, one can apply \cite[Remark 3.3.5]{KS}.

Taking the ($G$\nobreakdash-equi\-vari\-ant) cohomology groups of degree $k$ of (\ref{duality}) and applying \cite[Proposition~3.1.12]{KS} yields the ($G$-equivariant) Thom isomorphisms
\begin{alignat}{4}
\label{Thom}
H^{k-2c}(S',f^*\F(-c))&\isoto H_{S'}^k(S,\F),\\
\label{eqThom}
H_G^{k-2c}(S',f^*\F(-c))&\isoto H_{S',G}^k(S,\F).
\end{alignat}
The image $u\in H^{2c}_{S'}(S,\Z(c))$ of $1\in H^{0}(S',\Z)$ by the isomorphism (\ref{Thom}) applied with~$\F=\Z(c)$ is the Thom class of the complex vector bundle $N_{S'/S}$. 

Suppose now that we are also given a ($G$-equivariant) holomorphic map ${p:T\to S}$ of ($G$\nobreakdash-equi\-vari\-ant) complex manifolds, and that the subset $T':=p^{-1}(S')$ of $T$ is a complex submanifold of codimension $c$.  Let $f':T'\to T$ and $q:T'\to S'$ be the induced ($G$\nobreakdash-equivariant) maps. Let $v\in H^{2c}_{T'}(T,\Z(c))$ be the Thom class of $N_{T'/T}$. Let~$\lambda\in H^0(T',\Z)$ be such that $p^*u=\lambda\cdot v$, \ie such that the diagram
\begin{equation}
\label{pullbackThom}
\begin{aligned}
\xymatrix@C=1.5em@R=3ex{
q^*f^*\F(-c)[-2c]\ar^{\hspace{2em}\sim}[r]\ar_{\lambda}[d]&q^* f^!\F\ar^{}[d] \\
(f')^*p^*\F(-c)[-2c]\ar^{\hspace{1.4em}\sim}[r]&(f')^!(p^*\F),
}
\end{aligned}
\end{equation}
whose horizontal arrows are (\ref{duality}) and whose right vertical arrow is \cite[Proposition 3.1.9 (iii)]{KS},
commutes for $\F=\Z(c)$ (hence for any $\F$ as one sees by arguing as above).  Taking ($G$-equivariant) cohomology in (\ref{pullbackThom}) yields commutative diagrams
\begin{gather}
\begin{aligned}
\label{pullbacksupportnonG}
\xymatrix@C=1.5em@R=3ex{
H^{k-2c}(S',f^*\F(-c))\ar^{\hspace{1.9em}\sim}[r]\ar_{\lambda\cdot q^*}[d]&H_{S'}^k(S,\F)\ar^{p^*}[d] \\
H^{k-2c}(T',f^*\F(-c))\ar^{\hspace{1.4em}\sim}[r]&H_{T'}^k(T,p^*\F),
}
\end{aligned}
\\
\begin{aligned}
\label{pullbacksupport}
\xymatrix@C=1.5em@R=3ex{
H_G^{k-2c}(S',f^*\F(-c))\ar^{\hspace{1.9em}\sim}[r]\ar_{\lambda\cdot q^*}[d]&H_{S',G}^k(S,\F)\ar^{p^*}[d] \\
H_G^{k-2c}(T',f^*\F(-c))\ar^{\hspace{1.4em}\sim}[r]&H_{T',G}^k(T,p^*\F).
}
\end{aligned}
\end{gather}
If $p$ is transverse to~$f$, then $N_{T'/T}=q^*N_{S'/S}$ and hence $\lambda=1$.
If~$c=1$,  a local Thom class computation 
%choose a transverse slice upstairs.
%check very carefully ?!
shows that the integer $\lambda(x)$ is such that a local equation of $S'$ at~$p(x)$ pulls back to the $\lambda(x)$-th power of a local equation of~$T'$ at~$x$.

\subsection{Extending cohomology classes}

\begin{lem}
\label{ext0}
Let $S$ be a $G$-equivariant complex manifold. Let $S'\subset S$ be a nowhere dense $G$-invariant analytic subset.  Let $A$ be a $G$-module and fix $k\geq 1$. If the image of $\alpha\in H^k_G(S,A)$ in $H^k_G(S\setminus S',A)$ belongs to $H^k_G(S\setminus S',A)_0$, then $\alpha\in H^k_G(S,A)_0$.
\end{lem}

\begin{proof}
Fix $x\in S^G$.  There exists a path in $S^G$ connecting $x$ to a point $y\in (S\setminus S')^G$ (use Lemma \ref{lemrealpoints}). 
By homotopy invariance of cohomology, that $\alpha|_y=0$ implies that $\alpha|_x=0$.
\end{proof}

\begin{lem}
\label{extpetit}
Let $S$ be a $G$-equivariant complex manifold.  Let~$S'\subset S$ be a $G$\nobreakdash-in\-vari\-ant analytic subset of codimension $c\geq 1$. Let $A$ be a $G$-module.  If $k\leq 2c-2$, the restriction maps $H^k_G(S,A)\to H^k_G(S\setminus S',A)$ and $H^k_G(S,A)_0\to H^k_G(S\setminus S',A)_0$ are isomorphisms.
\end{lem}

\begin{proof}
Stratifying $S'$ by its singular locus, the singular locus of its singular locus, etc., we reduce to the case where $S'$ is a $G$-invariant submanifold of codimension $c$ of $S$.
The $G$\nobreakdash-equivariant Thom isomorphism (\ref{eqThom}) then yields an exact sequence
\begin{equation}
\label{sessupport}
H^{k-2c}_G(S',A(-c))\to H^k_G(S,A)\to H^k_G(S\setminus S',A)\to H_G^{k-2c+1}(S',A(-c)).
\end{equation}
If $k\leq 2c-2$, the two extreme terms of (\ref{sessupport}) vanish for degree reasons,  proving the first assertion. The second assertion now follows from Lemma \ref{ext0}.
\end{proof}

\begin{lem}
\label{extdiv}
Let $S$ be a $G$-equivariant complex manifold.
Let~${S'\subset S}$ be a $G$\nobreakdash-in\-variant complex submanifold of codimension~$1$. 
If ${H^1_G(S',A(-1))_0=0}$ for some $G$-module $A$, then the restriction map $H^{2}_G(S,A)_0\to H^{2}_G(S\setminus S',A)_0$ is~surjective.
\end{lem}

\begin{proof}
Fix $\alpha\in H^{2}_G(S\setminus S',A)_0$.  We claim that the image 
$\beta\in H^{1}_G(S',A(-1))$ of $\alpha$ by the right-hand side arrow of (\ref{sessupport}) (applied with $k=2$ and $c=1$) vanishes. The claim shows that $\alpha$ lifts to a class in $H^{2}_G(S,A)$ (hence in $H^{2}_G(S,A)_0$ by Lemma~\ref{ext0}).

We now prove the claim. As ${H^1_G(S',A(-1))_0=0}$, it suffices to show that~$\beta|_x=0$ for all $x\in (S')^G$.
Replacing $S$ with a small $G$-invariant complex ball~${T\subset S}$ that is transverse to $S'$ at $x$, and $S'$ with $x$ (to see that this is legitimate, use~(\ref{pullbacksupport}) noting that $\lambda=1$), we may assume that $S$ is the unit ball in~$\C$ and that $S'$ is the origin.  In this case,  $S\setminus S'$ retracts $G$-equivariantly  to a sphere~$\bS^1$ with~${(\bS^1)^G=\bS^0}$.
Since~$H^k_G(\bS^1,(\bS^1)^G,A)=0$ for all $k\geq 2$ (use (\ref{restrfp})), the restriction morphism ${H^2_G(\bS^1,A)\to H^2_G(\bS^0,A)}$ is an isomorphism, so ${H^2_G(S\setminus S',A)_0=H^2_G(\bS^1,A)_0=0}$. 
It follows that $\alpha$ vanishes and hence that so does $\beta$.
\end{proof}

\begin{rem}
\label{remH1van}
The assertion that $H^1_G(S',A(-1))_0=0$ in Lemma \ref{extdiv} holds if 
\begin{enumerate}[(i)]
\item either $S'$ is connected, $(S')^G\neq \varnothing$ and $H^1(S',A)=0$ (to see it, use (\ref{HSss}));
\item or $S'=T\sqcup T^{\sigma}$ for some complex manifold $T$ with $H^1(T,A)=0$ (use (\ref{cohoRC})).
\end{enumerate}
\end{rem}

We record the following non-$G$-equivariant analogue of Lemma \ref{extdiv}, obtained by applying it to $S\sqcup S^{\sigma}$ (or using directly the non-$G$-equivariant Thom isomorphism).

\begin{cor}
\label{corextdiv}
Let $S$ be a complex manifold.
Let~${S'\subset S}$ be a complex submanifold of codimension~$1$.  Let $A$ be an abelian group.
If ${H^1(S',A)=0}$, then the restriction map $H^{2}(S,A)\to H^{2}(S\setminus S',A)$ is~surjective.
\end{cor}

\subsection{A few \texorpdfstring{$G$}{G}-equivariant complex geometry lemmas}

Let $D$ be a divisor in a complex manifold~$S$.  Let $(D_i)_{i\in I}$ be the irreducible components of $D$. 
We say that $D$ is \textit{strict normal crossings} (or \textit{snc} for short) if, for all finite $J\subset I$,  the subset $D_J=\cap_{i\in J}D_i$ of $S$ is a (possibly empty) complex submanifold of codimension~$|J|$.
If $S$ is $G$\nobreakdash-equivariant and~$D$ is $G$-invariant, we say that $D$ is $G$\textit{-snc} if moreover, for all $i\in I$, either~$\sigma(D_i)=D_i$ or $\sigma(D_i)\cap D_i=\varnothing$.
The next proposition is standard. 

\begin{prop}
\label{ressing}
Let $S$ be a $G$-equivariant complex space. Let~$(S'_j)_{1\leq j\leq m}$ be finitely many nowhere dense $G$-invariant closed analytic subspaces of $S$. There exists a $G$-equivariant resolution of singularities $\nu:\wS\to S$ such that the support of~$\nu^{-1}(S_j')$ is a $G$-snc divisor in $\wS$ for all $1\leq j\leq m$.
\end{prop}

\begin{proof}
Blowing up the $S'_j$, we may assume that they are Cartier divisors.  Replacing them with their sum, we may suppose that there is only one of them (denoted by~$S'$).

Complex spaces admit resolutions of singularities (see \cite[The\-o\-rem~3.45]{Kollarsing}). As the resolutions constructed in \loccit are functorial for local biholomorphisms, and since a $G$-equivariant complex space can be thought of as a complex space $S$ together with a biholomorphism $\alpha:S^{\sigma}\to S$ such that~${\alpha\circ\alpha^{\sigma}=\Id_S}$, we see that $G$-equivariant complex spaces admit $G$-equivariant resolutions of singularities.

The line of reasoning of \cite[\S 3.44]{Kollarsing} shows that the principalization theorem \cite[Theorem 3.35]{Kollarsing} extends to the setting of complex spaces, hence,  arguing as above, also to the setting of $G$-equivariant complex spaces.  
Combining resolution of singularities and principalization proves the proposition with snc instead of $G$-snc.

Let $\nu_1:\wS_1\to S$ be the $G$-equivariant resolution obtained in this way.  Let $D$ be the $G$-invariant snc divisor $\nu_1^{-1}(S')$ and let $(D_i)_{i\in I}$ be its irreducible components. Set $D_J:=\cap_{i\in J}D_I$ for $J\subset I$ finite. To turn $D$ into a $G$-snc divisor after further blow-ups, first blow-up the union of the $(D_J)_{J\subset I, |J|=n}$ (where~$n$ is the dimension of~$\wS_1$), then the union of the strict transforms of the $(D_J)_{J\subset I, |J|=n-1}$, etc. 
Let~${\mu:\wS\to\wS_1}$ be the resulting modification, let $\wD\subset\wS$ be the inverse image of~$D$, and set $\nu:=\nu_1\circ\mu$.
As the blown-up loci are smooth (because the previous blow-ups separated their irreducible components),  the $G$-invariant snc divisor $\wD$ can be written as a union of smooth $G$-invariant divisors, and hence is a $G$-snc divisor.
\end{proof}

A $G$-equivariant resolution of singularities with $G$-snc exceptional locus is called a $G$\textit{-resolution of singularities}. They always exist by Proposition \ref{ressing}.

Let $f:X\to S$ be a proper holomorphic map of complex spaces. We refer to \cite[p.\,141]{BS} for what it means for a holomorphic line bundle on $X$ to be \textit{ample with respect to} $S$.
Proposition \ref{propample} below is standard in algebraic geometry.
% (see \eg \cite[Lemmas 0D2M and 0D2N]{SP}). 
We could not find a reference for it in complex-analytic geometry and provide a short proof following Conrad's arguments in rigid geometry \cite{Conrad} (it is already indicated in \loccit that the proof presented there extends to the complex setting).

\begin{prop}
\label{propample}
Let $f:X\to S$ be a proper holomorphic map between complex spaces. Let $\cL$ be a holomorphic line bundle on $X$.
\begin{enumerate}[(i)]
\item Let $\cF$ be a coherent sheaf on $X$.  If $\cL|_{X_s}$ is ample for some $s\in S$, then $\RR^kf_*(\cF\otimes\cL^{\otimes l})_s=0$ for all $k>0$ and all $l\gg0$.
\item If $\cL|_{X_s}$ is ample for all $s\in S$, then $\cL$ is ample with respect to $S$.
\end{enumerate}
\end{prop}
%I could state and prove openness of ampleness but I don't need it.  [Conrad, p.1094] states that Zariski-openness is not known (sounds weird... can't we look for all l at the meromorphic map to projective bundle induced by f_*L^n??)

\begin{proof}
Let $\km_s\subset\cO_S$ be the ideal sheaf of $s$ in $S$.  For $n\geq 1$, let $S_n\subset S$ be the complex subspace defined by $\km_s^n\subset \cO_S$. Set $X_n:=X\times_{S_n}S$. As $\cL|_{X_s}$ is ample,  so is $\cL|_{X_n^{\red}}$ for all $n\geq 1$.  We deduce from the cohomological criterion of ampleness that~$\cL_n:=\cL|_{X_n}$ is ample for all $n\geq 1$.
%stated in [Höring, Peternell, Minimal models for Kähler threefolds] with no reference.
%Unfortunately [Grauert,  Uber Modifikationen...] assumes complex spaces have no nilpotents...
By GAGA, the compact complex space~$X_n$ is a projective algebraic variety over $\C$, endowed with the ample algebraic line bundle~$\cL_n$ and with the algebraic coherent sheaf $\cF_n:=\cF|_{X_n}$. The schemes $(X_n)_{n\geq 1}$ therefore form a formal scheme over $\Spf(\widehat{\cO_{S,s}})$, which is the formal completion of a projective scheme $\kX\to\Spec(\widehat{\cO_{S,s}})$ endowed with an ample line bundle $\mathfrak{L}=(\cL_n)_{n\geq 1}$ and a coherent sheaf $\kF:=(\cF_n)_{n\geq 1}$, by formal GAGA.  
One then computes
\begin{equation}
\label{GAGAs}
\RR^kf_*(\cF\otimes\cL^{\otimes l})_s\otimes_{\cO_{S,s}}\widehat{\cO_{S,s}}=\varprojlim_n H^k(X_n,\cF_n\otimes\cL_n^{\otimes l})=H^k(\kX,\kF\otimes\kL^{\otimes l}),
\end{equation}
by Grauert's and Grothendieck's theorems on formal functions. 
%also uses GAGA to compare algebraic and analytic cohomology of coherent sheaves.
As $\kL$ is ample, the right-hand side of (\ref{GAGAs}) vanishes for $k>0$ and $l\gg 0$, and assertion (i) is proven.

Assertion (ii) follows from (i) and from \cite[IV, Theorem 4.1]{BS} (in the statement of which one should replace very ample by ample).
\end{proof}

\begin{lem}
\label{lemglobgen}
Let $\nu:\wS\to S$ be a $G$\nobreakdash-resolution of singularities of a 
$G$\nobreakdash-equi\-vari\-ant Stein surface $S$.  Let $(\cF_i)_{i\in I}$ be a finite collection of $G$-equivariant coherent sheaves on~$\wS$. Then there exists a $G$-equivariant holomorphic line bundle~$\cL$ on $\wS$ such that~$\cF_i\otimes\cL^{\otimes l}$ is globally generated for all $l\geq 1$ and all $i\in I$.
\end{lem}

\begin{proof}
 It suffices to deal with the single $G$-equivariant coherent sheaf $\cF:=\bigoplus_{i\in I}\cF_i$.

After replacing it with its normalization, we may assume that $S$ is normal. 
 Since~$S$ is a normal surface, the subset $\Sigma\subset S$ above which $\nu$ is not an isomorphism is discrete.
Fix $s\in\Sigma$.
As the intersection matrix of the components of~$\wS_s:=\nu^{-1}(s)$ is negative definite (see \eg \cite[Proposition~2.1.12]{Nemethi}),  there is a (not necessarily effective)
%Without requiring effective as we do here, this is trivial.
divisor $D_s$ on~$\wS$ supported on $\wS_s$ such that~${\cL_s:=\cO_{\wS}(-D_s)}$ has positive degree on all the irreducible components of $\wS_s$, and hence such that~$\cL_s|_{\wS_s}$ is ample.  After replacing~$D_s$  with~${D_s+D_{\sigma(s)}}$, we may assume that $D_{\sigma(s)}=D_s$ for all $s\in\Sigma$.

For $s\in\Sigma$,  Proposition \ref{propample} shows that $\cL_s$ is ample with respect to~$S$ above some neighborhood of $s$.  Consequently, by \cite[IV, Theorem 2.1]{BS}, there exists $m_s\geq 0$ such that the evaluation morphism
$\nu^*\nu_*(\cF\otimes\cL_s^{\otimes l})\to \cF\otimes\cL_s^{\otimes l}$ is surjective in a neighborhood of $\wS_s$ for all ${l\geq m_s}$. One can of course ensure that $m_s=m_{\sigma(s)}$.

Set $\cL:=\cO_{\wS}(-\sum_{s\in \Sigma}m_sD_s)$. Then $\nu^*\nu_*(\cF\otimes\cL^{\otimes l})\to \cF\otimes\cL^{\otimes l}$ is surjective for~${l\geq 1}$.  As $S$ is Stein, $\nu_*(\cF\otimes\cL^{\otimes l})$ is globally generated, hence so is~${\cF\otimes\cL^{\otimes l}}$.
\end{proof}

\begin{lem}
\label{BertiniGsnc}
Let $D$ be a reduced $G$-snc divisor in a $G$-equivariant complex manifold~$S$. Let $\cL$ be a $G$-equivariant holomorphic line bundle on $S$. Let $V\subset H^0(S,\cL)$ be a finite-dimensional $G$-invariant subspace generating $\cL$. Then there exists a countable intersection of dense open subsets $W\subset V^G$ such that $D\cup\{\sigma=0\}$ is a reduced $G$-snc divisor for all $\sigma\in W$.
\end{lem}

\begin{proof}
Let $(D_i)_{i\in I}$ be the irreducible components of $D$. For each finite subset $J\subset I$, set $D_J:=\cap_{i\in J}D_i$. As $D$ is snc, the subset $D_J\subset S$ is a complex submanifold of codimension $|J|$.  By \cite[Theo\-rem II.5]{Manaresi} applied on $D_J$, there exists a countable intersection of dense open subsets $W_J\subset V^G$ such that $D_J\cap\{\sigma=0\}$ is nonsingular of codimension $1$ in $D_J$ for all $\sigma\in W_J$ (in \cite[Step II of the proof of Theorem~II.5]{Manaresi}, choose the $F_i$ and $g$ to be $G$-invariant, and note that $c$ can be taken to be real).  It remains to set $W:=\cap_{J\subset I\textrm{ finite}}W_J$. 
\end{proof}

\subsection{Snc divisors on resolutions of singularities of Stein surfaces}

\begin{lem}
\label{lemsnc}
Let $\nu:\wS\to S$ be a $G$\nobreakdash-resolution of singularities of a 
$G$\nobreakdash-equi\-vari\-ant Stein surface $S$. Let $D\subset \wS$ be a reduced $G$-snc divisor. Fix~$m\geq 1$.  There exist a globally generated $G$-equivariant holomorphic line bundle~$\cL$ on~$\wS$ and a section $\sigma\in H^0(\wS,\cL^{\otimes m})^G$ such that $\{\sigma=0\}$ is a reduced $G$-snc divisor containing~$D$.
\end{lem}

\begin{proof}
Construct $\cL$ by applying Lemma \ref{lemglobgen} with $\cF_1=\cO_{\wS}(-D)$ and $\cF_2=\cO_{\wS}$. As~$\cL^{\otimes m}(-D)$ is globally generated, an application of Baire's theorem in the Fr\'echet space~$H^0(\wS,\cL^{\otimes m}(-D))$ shows the existence of sections $\tau_r\in H^0(\wS,\cL^{\otimes m}(-D))$ that are not identically zero on any irreducible component of ${\{\tau_1=\dots=\tau_{r-1}=0\}}$ (for~$1\leq r\leq 3$).  The smallest $G$-invariant subspace $V\subset H^0(\wS,\cL^{\otimes m}(-D))$ containing $(\tau_r)_{1\leq r\leq 3}$ is finite-dimensional and generates~$\cL^{\otimes m}(-D)$.

By Lemma \ref{BertiniGsnc} and Baire's theorem, one can find $\tau\in V^G$ such that $D\cup\{\tau=0\}$ is a reduced $G$-snc divisor. To conclude, set $\sigma:=\tau\cdot\tau_{D}$, where $\tau_{D}\in H^0(\wS,\cO_{\wS}(D))$ is the equation of $D$.
\end{proof}

The next lemma will only be used in the proof of Proposition \ref{killram2}.

\begin{lem}
\label{reallocusram}
Let $\nu:\wS\to S$ be a $G$-resolution of singularities of a $G$\nobreakdash-equi\-variant Stein surface $S$.  Let $D\subset\wS$ be a reduced $G$-snc divisor.  There exist a $G$\nobreakdash-equi\-variant line bundle $\cN$ on $\wS$ and a section $\tau\in H^0(\wS,\cN^{\otimes 2})^G$ such that $D':=D\cup \{\tau=0\}$ is a reduced $G$-snc divisor, and each connected component~$\Omega$ of~$\wS^G$ contains exactly one connected component of $(D')^G$, which is noncompact if so is $\Omega$.
\end{lem}
%Could hope for stronger statement approximating any 1-dim submanifold of $\wS^G$ by real loci of nonsingular curves (combining somehow Stone-Weierstrass and Carleman).

\begin{proof}
After replacing $S$ with its normalization, we may assume that $S$ is normal.  Let $\Sigma\subset S$ be the discrete subset over which $\nu$ is not an isomorphism.  For~${s\in\Sigma^G}$,  set $\wS_s:=\nu^{-1}(s)$.  Choose a finite subset $\Theta_s\subset(\wS_s)^G$ containing one point (in generic position) on each connected component of $E^G$ for each irreducible component~$E$ of~$\wS_s$.
%not on other connected components, not on D if can afford that.
Define~$\Theta:=\cup_{s\in\Sigma}\Theta_s$.
Our proof has two steps.

\begin{Step}
\label{step1real}
We reduce to the case where for all $s\in\Sigma^G$ and all $x\in\Theta_s$,  there is an ir\-re\-ducible component $\Delta$ of $D$ such that $\Delta^G$ intersects $(\wS_s)^G$ transversally at $x$.
\end{Step}

Let $\cI\subset\cO_{\wS}$ be the ideal sheaf of $\Theta$. By Lemma \ref{lemglobgen},  there exists a $G$-equivariant holomorphic line bundle $\cL$ on $\wS$ such that $\cI\otimes\cL^{\otimes 2}$ is globally generated.
As~$\wS$ has finite dimension and the fibers of $\cI\otimes \cL^{\otimes 2}$ have bounded dimension,  the $G$\nobreakdash-equivariant coherent sheaf $\cI\otimes\cL^{\otimes 2}$ is in fact generated by a finite-dimensional vector subspace~$V\subset H^0(\wS,\cI\otimes\cL^{\otimes 2})$, which we may choose to be $G$-invariant.

Fix $s\in\Sigma^G$ and $x\in\Theta_s$ lying on some irreducible component $E$ of $\wS_s$.
As~$V$ generates $\cI\otimes\cL^{\otimes 2}$ at $x$,  there exists a dense open subset $W_x\subset  V^G$ such that~$\{\sigma=0\}$ is smooth at $x$ and intersects $E$ transversally at $x$ for all $\sigma\in W_x$. 
By Lemma \ref{BertiniGsnc}, there exists a countable intersection of dense open subsets $W\subset V^G$ such  that the divisor $(D\cup\{\sigma=0\})\setminus \Theta$ is $G$-snc in $\wS\setminus\Theta$ for all~$\sigma\in W$.
By Baire's theorem, we may choose $\sigma\in W\cap\bigcap_{x\in \Theta}W_x$.  It now suffices to replace $D$ with $D\cup\{\sigma=0\}$.
\begin{Step}
We prove the lemma under the additional hypothesis stated in Step~\ref{step1real}.
\end{Step}

The hypothesis that the normal crossings curve $D^G$ in the topological surface~$\wS^G$ intersects all the components of $\wS_s^G$ transversally at some point (for all $s\in\Sigma^G$) implies that one can find a closed subset $M\subset S^G$ with $M\cap\Sigma^G=\varnothing$ such that 
\begin{enumerate}[(i)]
\item 
\label{iM}
$M$ is a $\ci$ submanifold of dimension $1$ of $S^G\setminus\Sigma^G$;
\item 
\label{iiM}
the cohomology class $[M]\in H^1(S^G,\Z/2)$ is trivial; equivalently, one can write $M=\{\phi=0\}$ for some $\ci$ map $\phi:S^G\to\R$ of which $0$ is a regular value;
\item 
\label{iiiM}
the curves $M$ and $D^G$ intersect transversally along nonsingular points;
\item
\label{ivM}
 For each connected component $\Omega$ of $\wS^G$,  the set $(D^G\cup M)\cap\Omega$ is connected, and noncompact if so is $\Omega$.
\end{enumerate}
More precisely, it is easy to construct $M$ such that (\ref{iiiM}) and (\ref{ivM}) are satisfied, but such that $M$ is a disjoint union of embedded segments and closed half-lines. To ensure that~(\ref{iM}) also holds, double these segments and half-lines to turn them into (very stretched out) circles and lines respectively.
This procedure (smoothing doubled segments and half-lines) guarantees that $M$ is included in a disjoint union of contractible open subsets of $S^G\setminus\Sigma^G$. It follows that $[M]=0$, hence that~(\ref{iiM}) holds. 
Making use of a tubular neighborhood of $M$ in $S^G\setminus \Sigma^G$, it is possible to modify the~$\ci$ equation $\phi:S^G\to\R$ of~$M$ (see (\ref{iiM})) so that the following holds:
\begin{enumerate}[(i)]
  \setcounter{enumi}{4}
\item 
\label{vM}
for all $s\in\Sigma^G$, the function $\phi$ is constant in some neighborhood of $s$;
\item 
\label{viM}
there is $\varepsilon>0$ such that, for $|t|<\varepsilon$, the subset $M_{t}:=\phi^{-1}(t)$ satisfies (\ref{iM})-(\ref{ivM}).
\end{enumerate}

Now choose a $G$-equivariant proper injective holomorphic map $i:S\to \C^N$ that is immersive at the nonsingular points of $S$, hence away from $\Sigma$ (see \cite[V, Theorem~3.7]{GMT}).  Extend $\phi$ to a $\ci$ map $\tphi:\R^N\to \R$ (do it locally and globalize using partitions of unity). By a higher-dimensional version of Carleman's approximation theorem with a control on derivatives (\eg \cite[Theorem 1.1]{MWO} whose bounded $E$-hulls hypothesis is satisfied by \cite[Theorem 2]{SC}), 
% \cite[Theorem]{Scheinberg} for higher dim with no derivatives
%[Hoischen, Eine Verschfärung...] pour un énoncé C^infty très fort en une variable.
there exists a holomorphic map $\tpsi:\C^N\to\C$ such that $\tpsi|_{\R^N}$ is close to $\tphi$ in the strong $\mathcal{C}^1$ topology.  Replacing~$\tpsi$ with $z\mapsto (\tpsi(z)+\overline{\tpsi(\bar{z})})/2$, we may assume that $\tpsi$ is $G$-equivariant. Set~${\psi:=\tpsi\circ i:S\to \C}$.

 If $\tpsi|_{\R^N}$ approximates $\tphi$ well enough, then~$\psi|_{S^G}$ is close to~$\phi$ in the strong $\mathcal{C}^1$ topology (for which see \eg \cite[Chap.\,2,  \S 1]{Hirsch}). It follows that~(\ref{viM}) holds with~$\phi$ replaced with~$\psi|_{S^G}$. To verify this, fix a small tubular neighborhood $U$ of $M$ in~$S^G\setminus \Sigma^G$ with retraction $\pi:U\to M$, and note that if~$\psi|_{S^G}$ is sufficiently close to~$\phi$ in the strong $\mathcal{C}^1$ topology and $t\in\R$ is small enough, then~$M'_t:=(\psi|_{S^G})^{-1}(t)$ is a $\ci$ submanifold of $U$ such that $\pi|_{M'_t}:M'_t\to M$ is a diffeomorphism whose inverse is close  to the inclusion~$M\subset S^G\setminus\Sigma^G$ in the strong $\mathcal{C}^1$ topology.

Choose $\cN:=\cO_{\wS}$ and $\tau:=\nu^*\psi-t$ for $t\in\R$.
By Lemma \ref{BertiniGsnc},  the divisor $D\cup\{\sigma=0\}$ is reduced and $G$-snc if $t$ is chosen generic in Baire's sense. The other assertions hold by (\ref{viM}) (with $\phi$ replaced with~$\psi|_{S^G}$) if $|t|$ is small enough.
\end{proof}

\subsection{Killing ramification in alterations}

\begin{prop}
\label{killram}
Let $S$ be a reduced $G$-equivariant Stein space of dimension~$\leq~2$.  Let~${S'}$ be a nowhere dense $G$-invariant closed analytic subset of $S$. 
Let~$A$ be an $m$\nobreakdash-torsion $G$\nobreakdash-module for some $m\geq 1$. Choose ${\alpha\in H^2_G(S\setminus S',A)_0}$. There exists a $G$\nobreakdash-equi\-vari\-ant alteration $p:T\to S$ of degree $m$ with $T$ nonsingular and a class~${\beta\in H^2_G(T,A)_0}$ such that ${p^*\alpha=\beta|_{T\setminus p^{-1}(S')}}$ in $H^2_G(T\setminus p^{-1}(S'),A)$.
\end{prop}

\begin{proof}
Let $\nu:\wS\to S$ be a $G$-resolution of singularities such that the support~$\wS'$ of~$\nu^{-1}(S')$ is a $G$-snc divisor in $\wS$ (see Proposition \ref{ressing}).
By Lemma \ref{lemsnc},  
one can find a $G$\nobreakdash-equi\-vari\-ant holomorphic line bundle $\cL$ on $\wS$  and ${\sigma\in H^0(\wS, \cL^{\otimes m})^G}$ such that $\{\sigma=0\}$ is a reduced $G$\nobreakdash-snc divisor containing $\wS'$.  Consider the degree~$m$ analytic covering ${\tp:\wT\to\wS}$ with equation $\{z^m=\sigma\}$. Let $\mu:T\to\wT$ be the minimal resolution of singularities of $\wT$. Let $p:T\to S$ be the induced $G$-equivariant map.

Let $\Sigma\subset \wS$ be the singular locus of $\wS'$.  Applying (\ref{pullbacksupport}) to $\tp:\wT\setminus\tp^{-1}(\Sigma)\to\wS\setminus\Sigma$,  and noting that $\lambda=m$, one sees that the pullback morphism
$$H^3_{\tilde{S}'\setminus\Sigma,G}(\wS\setminus\Sigma,A)\to H^3_{\tp^{-1}(\tilde{S}')\setminus\tp^{-1}(\Sigma),G}(\wT\setminus\tp^{-1}(\Sigma),A)$$
vanishes identically. It follows that $\tp^*\nu^*\alpha\in H^2_G(\wT\setminus \tp^{-1}(\wS),A)_0$ lifts to a class $\tilde{\alpha}\in H^2_G(\wT\setminus\tp^{-1}(\Sigma),A)$.  By Lemma \ref{ext0}, one has $\tilde{\alpha}\in  H^2_G(\wT\setminus\tp^{-1}(\Sigma),A)_0$. 

As $\tp^{-1}(\Sigma)$ does not meet the irreducible components of $\wT$ of dimension $\leq 1$, we may henceforth assume that $\wT$ is a surface. The singularities of $\wT$ at points of~$\tp^{-1}(\Sigma)$ are then $A_{m-1}$ singularities (locally isomorphic to $\{z^m=xy\}$), so the connected components of the exceptional locus of $\mu:T\to\wT$ are chains of 
$m-1$ 
rational curves. Several applications of Lemmas \ref{extpetit} and \ref{extdiv} (taking Remark~\ref{remH1van} into account) imply that $\tilde{\alpha}$ extends to a class $\beta\in H^2_G(T,A)_0$,  as desired.
\end{proof}

Applied to a $G$-equivariant Stein space of the form $S\sqcup S^{\sigma}$, Proposition \ref{killram}
yields the next non-$G$-equivariant corollary (use (\ref{cohoRC})).

\begin{cor}
\label{corkillram}
Let $S$ be a reduced Stein space of dimension $\leq2$.  Let~${S'}\subset S$ be a nowhere dense closed analytic subset.  Fix $m\geq 1$, an $m$\nobreakdash-torsion abelian group $A$, and ${\alpha\in H^2(S\setminus S',A)}$. There exist an alteration $p:T\to S$ of degree $m$ with $T$ nonsingular and a class~${\beta\in H^2(T,A)}$ with ${p^*\alpha=\beta|_{T\setminus p^{-1}(S')}}$ in $H^2(T\setminus p^{-1}(S'),A)$.
\end{cor}

The following variant of Proposition \ref{killram} will be used in the proof of Theorem~\ref{thpi}. 
Its proof is inspired by \cite[\S 4.2]{pi}.
We let $\partial$ denote the boundary maps associated with the short exact sequence of $G$-modules $0\to\Z(1)\xrightarrow{2}\Z(1)\to\Z/2\to 0$.

\begin{prop}
\label{killram2}
Keep the notation of Proposition \ref{killram}. If $m=2$ and~$A=\Z/2$, one can choose $p:T\to S$ and $\beta$ so that $\partial(\beta)\in H^3_G(T,\Z(1))$ vanishes.
\end{prop}

\begin{proof}
Keep the notation of the proof of Proposition \ref{killram} and set ${q:=\tp\circ\nu:T\to\wS}$.

We claim that one can arrange that $q_*:H^2(T^G,\Z/2)\to H^2(\wS^G,\Z/2)$ is injective. To do so, we may assume that $\wS$ has dimension $2$.  Just after defining~$\cL$ and~$\sigma$ in the proof of Proposition~\ref{killram},  choose $\cN$ and $\tau$ as in Lemma \ref{reallocusram} applied with~${D:=\{\sigma=0\}}$, and replace $\cL$ with~$\cL\otimes\cN$ and $\sigma$ with $\sigma\cdot\tau$. This ensures that each connected component $\Omega$ of~$\wS^G$ contains exactly one connected component of~$D^G$, which is noncompact if so is $\Omega$. 
This implies that $(q|_{T^G})^{-1}(\Omega)$ is connected, and noncompact if so is $\Omega$. It follows that 
$q_*:H^2(T^G,\Z/2)\to H^2(\wS^G,\Z/2)$ is injective, as wanted.

We may now prove that $\partial(\beta)=0$.
Consider the commutative diagram
\begin{equation}
\label{pushBrauer}
\begin{aligned}
\xymatrix@C=1.5em@R=3ex{
H^2_G(T,\Z(1))\ar^{}[r]\ar_{q_*}[d]&H^2_G(T,\Z/2)\ar^{q_*}[d]\ar^{\partial}[r]&H^3_G(T,\Z(1))\ar^{q_*}[d]
\\
H^2_G(\wS,\Z(1))\ar^{}[r]&H^2_G(\wS,\Z/2)\ar^{\partial}[r]&H^3_G(\wS,\Z(1)),
}
\end{aligned}
\end{equation}
whose rows are exact sequences of $G$-equivariant cohomology associated with the short exact sequence $0\to\Z(1)\xrightarrow{2}\Z(1)\to\Z/2\to 0$, and whose vertical arrows are pushforward maps.
One computes that $(q_*\beta)|_{\tilde{S}\setminus \tilde{S}'}=q_*q^*\nu^*\alpha=2\nu^*\alpha=0$ (where we used the projection formula). 
Using (\ref{eqThom}), one sees that the reduction modulo $2$ morphism $H^2_{\tilde{S}'\setminus \Sigma,G}(\wS\setminus \Sigma,\Z(1))\to H^2_{\tilde{S}'\setminus \Sigma,G}(\wS\setminus\Sigma,\Z/2)$ is onto. Consequently,  the class~$(q_*\beta)|_{\tilde{S}\setminus\Sigma}$, which lifts to 
$H^2_{\tilde{S}'\setminus \Sigma,G}(\wS\setminus \Sigma,\Z/2)$ because $(q_*\beta)|_{\tilde{S}\setminus \tilde{S}'}=0$, further lifts to $H^2_{\tilde{S}'\setminus \Sigma,G}(\wS\setminus \Sigma,\Z(1))$. 
We deduce that $(q_*\beta)|_{\tilde{S}\setminus\Sigma}$ lifts to $H^2_G(\wS\setminus\Sigma,\Z(1))$. It follows from Lemma \ref{ext0} that $q_*\beta$ lifts to $H^2_G(\wS,\Z(1))$, and hence from the exactness and commutativity of (\ref{pushBrauer}) that $q_*\partial(\beta)=0$ in $H^3_G(\wS,\Z(1))$.

We now consider the diagram
\begin{equation}
\label{decompoRpoints}
\begin{aligned}
\xymatrix@C=1.5em@R=3ex{
H^3_G(T,\Z(1))\ar^{\sim\hspace{.5em}}[r]\ar_{q_*}[d]&H^3_G(T^G,\Z(1))\ar@{=}^{}[r]&H^0(T^G,\Z/2)\oplus H^2(T^G,\Z/2)\ar^{q_*}[d]
\\
H^3_G(\wS,\Z(1))\ar^{\sim\hspace{.5em}}[r]&H^3_G(\wS^G,\Z(1))\ar@{=}^{}[r]&H^0(\wS^G,\Z/2)\oplus H^2(\wS^G,\Z/2),
}
\end{aligned}
\end{equation}
whose vertical arrows are pushforward maps,  whose horizontal arrows are restriction maps (which are isomorphisms by \cite[Proposition 2.8]{Artinvanishing}),  where the equalities are the canonical decompositions of \cite[(1.30)]{BW1}, and which is commutative as a consequence of \cite[(1.33) and Proposition 1.22]{BW1}.
(We applied to the $G$\nobreakdash-equi\-vari\-ant cohomology of $G$-equivariant complex spaces a few facts that are proven in \cite{BW1} for the $G$-equivariant cohomology of algebraic varieties over real closed fields, \eg over $\R$.  These facts remain true in our setting,  with the same proofs.)

Since $\beta\in H^2_G(T,\Z/2)_0$,  one has~$\partial(\beta)\in H^3_G(T,\Z(1))_0$.
The image of~$\partial(\beta)|_{T^G}$ in~$H^0(T^G,\Z/2)$, which is given by restriction to real points, therefore vanishes.  As~${q_*\partial(\beta)=0}$ and ${q_*:H^2(T^G,\Z/2)\to H^2(\wS^G,\Z/2)}$ is injective,  the commutativity of (\ref{decompoRpoints}) implies that~$\partial(\beta)|_{T^G}=0$. It follows that~$\partial(\beta)=0$.
\end{proof}

The next variant of Corollary \ref{corkillram} will be used in the proof of Theorem \ref{thab}. Its proof is inspired by \cite[Proof of Theorem 2.2]{CTOP}.

\begin{prop}
\label{killramab}
Let $\pi:\whS\to S$ be a an analytic covering of connected normal Stein surfaces such that $\cM(S)\subset\cM(\whS)$ is Galois of group $\Gamma$. Let~${{\whS'}\subset \whS}$ be a $\Gamma$\nobreakdash-in\-va\-riant nowhere dense closed analytic subset.  Fix $m\geq 1$
and ${\alpha\in H^2(\whS\setminus \whS',\Z/m)}$. Set $N:=m|\Gamma|$.
There exist $f\in\cM(S)^*$,  an alteration $p:T\to \whS$ with $T$ nonsingular and $\cM(T)=\cM(\whS)[f^{\frac{1}{N}}]$, and a class~${\beta\in H^2(T,\Z/m)}$ with ${p^*\alpha=\beta|_{T\setminus p^{-1}(\hat{S}')}}$.
\end{prop}

\begin{proof}
Let $\nu:\wS\to\whS$ be a $\Gamma$-equivariant resolution of singularities such that ${\wS':=\nu^{-1}(\whS')}$ is an snc divisor. 
By a non-$G$-equivariant version of Lemma~\ref{lemsnc} (obtained by discarding $G$-actions in the proof, or by applying it formally to~${S\sqcup S^{\sigma}}$), there is a globally generated holomorphic line bundle $\cL$ on $\wS$ and $\sigma\in H^0(\wS,\cL^{\otimes N})$ such that~${\{\sigma=0\}}$ is a reduced snc divisor containing $\wS'$.
Write $\cL_1:=\bigotimes_{\gamma\in\Gamma}\gamma^*\cL$ and $\sigma_1:=\prod_{\gamma\in\Gamma}\gamma^*\sigma\in H^0(\wS,\cL_1^{\otimes N})$. Let $\nu_1:\wS_1\to\wS$ be a $\Gamma$-equivariant modification  with $\wS_1$ nonsingular such that $\{\nu_1^*\sigma_1=0\}$ is an snc divisor. 
Let $\tp_1:\wT_1\to\wS_1$ be the normalization of the degree~$N$ analytic covering with equation $\{z^N=\nu_1^*\sigma_1\}$. Let~${\mu:T\to\wT_1}$ be the minimal resolution of singularities. Let $p:T\to \whS$ be the induced $\Gamma$\nobreakdash-equivariant~map.

We first verify that $\cM(T)$ has the required form. Let $\tau\in H^0(\wS,\cL)$ be any nonzero section. Define $\tau_1:=\prod_{\gamma\in\Gamma}\gamma^*\tau\in H^0(\wS,\cL_1)$. Consider the rational function $g:=\frac{\sigma}{\tau^N}\in\cM(\wS)=\cM(\whS)$. Then $f:=\frac{\sigma_1}{\tau_1^N}=\prod_{\gamma\in\Gamma}\gamma^*g\in\cM(\whS)^{\Gamma}=\cM(S)$. By construction, one has $\cM(\whS)[f^{\frac{1}{N}}]\subset\cM(T)$. This inclusion is an equality as both $\cM(T)$ and $\cM(\whS)[f^{\frac{1}{N}}]$ have degree $N$ over $\cM(\whS)$ (apply (\ref{eqext}) to the Stein factorization of $p$).

We finally construct a class $\beta$ with the desired property. Let $\wS'_1\subset \wS_1$ be the strict transform of $\wS'$. Since the connected components of the exceptional locus of~$\nu_1$ are trees of rational curves meeting $\wS_1'$ at a single point, applying Corollary~\ref{corextdiv} repeatedly shows that the class $\nu_1^*\nu^*\alpha\in H^2(\wS_1\setminus\nu_1^{-1}(\wS'),\Z/m)$ extends to a class~$\alpha_1\in H^2(\wS_1\setminus \wS'_1,\Z/m)$.  Let $\Sigma\subset \wS_1$ be the singular locus of $\wS'_1$.

As the multiplicities of the components of $\wS_1'$ in $\{\nu_1^*\sigma_1=0\}$ divide $|\Gamma|$, the indices of ramification of $\tp_1$ along divisors of $\wT_1$ dominating these components are multiples of~$m$. It follows from (\ref{pullbacksupportnonG}) that the pull-back morphism
$$H^3_{\tilde{S}_1'\setminus\Sigma}(\wS_1\setminus\Sigma,\Z/m)\to H^3_{\tp_1^{-1}(\tilde{S}_1')\setminus\tp_1^{-1}(\Sigma)}(\wT_1\setminus\tp_1^{-1}(\Sigma),\Z/m)$$
vanishes identically. We deduce that $\tp_1^*\alpha_1\in H^2(\wT_1\setminus\tp_1^{-1}(\wS_1'),\Z/m)$ lifts to a class $\talpha_1\in H^2(\wT_1\setminus\tp_1^{-1}(\Sigma),\Z/m)$.

Local computations show that $\wT_1$ has $A_k$ singularities (for varying $k$) at points of $\tp_1^{-1}(\Sigma)$, 
and hence that the connected components of the exceptional locus of $\mu:T\to\wT_1$ are chains of rational curves. Several applications of Corollary \ref{corextdiv} therefore imply that $\talpha_1$ extends to a class $\beta\in H^2(T,\Z/m)$,  as wanted.
\end{proof}

\subsection{Killing global cohomology classes on Zariski-open sets}

\begin{prop}
\label{GLefschetz11}
Let $S$ be a $G$-equivariant reduced Stein space of dimension $\leq 2$. Let~$p:T\to S$ be a $G$-equivariant alteration.  For all $\beta\in H^2_G(T,\Z(1))$, there exists a $G$-invariant nowhere dense closed analytic subset $T'\subset T$ such that $\beta|_{T\setminus T'}=0$.
\end{prop}

\begin{proof}
After replacing $S$ with $\oT$, where $T\to\oT\to S$ is the Stein factorization of~$p$, we may assume that $p$ is a modification. Replacing $T$ with a $G$-equivariant resolution of singularities (see Proposition \ref{ressing}), we may assume that $T$ is nonsingular.

The sheaf $\RR^1p_*\cO_T$ is supported on a discrete set and the sheaves $\RR^sp_*\cO_T$ vanish for $s\geq 2$ (as a consequence of Grauert's comparison theorem \cite[Haupsatz~II~a]{Grauert}). 
%+vanishing of coherent cohomology above dimension, which seems to be due to [Andreotti--Grauert, Théorèmes de finitude....] but that I will not quote here (in the dimension <=1 case...).
Since the higher cohomology groups of~$p_*\cO_T$ vanish as $S$ is Stein, we deduce from the Leray spectral sequence of~$p$ that $H^2(T,\cO_T)=0$. 
It follows that the group $H^2_G(T,\cO_T)$ vanishes (see \cite[\S A.3]{Tight}).

The long exact sequence of $G$-equivariant cohomology associated with the exponential short exact sequence 
$0\to \Z(1)\xrightarrow{2\pi\sqrt{-1}}\cO_T\xrightarrow{\exp}\cO_T^*\to 0$ therefore shows that the cohomology  class $\beta$ is the image by the $G$-equivariant cycle class map $\cl:\Pic_G(T)=H^1_G(T,\cO_T^*)\to H^2_G(T,\Z(1))$ (see \cite[\S A.4]{Tight}), of some $G$\nobreakdash-equivariant holomorphic line bundle $\cN\in\Pic_G(T)$ on $T$.

By Lemma \ref{lemglobgen}, there exists a $G$-equivariant holomorphic line bundle $\cL$ on $T$ such that both $\cL$ and $\cN\otimes \cL$ are globally generated. It follows that there exist divisors~$D_1\subset T$ and $D_2\subset T$ with~$\cO(D_1)=\cL$ and $\cO(D_2)=\cN\otimes \cL$. One therefore has $\cN=\cO(D_2-D_1)$. It now suffices to take $T':=D_1\cup D_2$.
\end{proof}

\begin{cor}
\label{corGLefschetz11}
Let $S$ be a reduced Stein space of dimension $\leq 2$. 
Let $p:T\to S$ be an alteration.  Let $A$ be a finitely generated abelian group.
%Actually not necessary.  
For all~$\beta\in H^2(T,A)$,  there exists a nowhere dense closed analytic subset $T'\subset T$ such that $\beta|_{T\setminus T'}=0$.
\end{cor}

\begin{proof}
We may assume that $A=\Z$ or $A=\Z/m$ for some $m\geq 1$. The Leray spectral sequence for $p$ and the Artin vanishing theorem (see \cite[Proposition~2.2]{Artinvanishing}) imply that $H^3(T,\Z)=0$.  It follows that the reduction modulo $m$ morphism $H^2(T,\Z)\to H^2(T,\Z/m)$ is surjective. The case $A=\Z/m$ therefore reduces to the case $A=\Z$.  In this case, it suffices to apply Proposition \ref{GLefschetz11} to the $G$\nobreakdash-equivariant space $S\sqcup S^{\sigma}$ (making use of~(\ref{cohoRC})). 
\end{proof}

\section{A generic comparison theorem over Stein surfaces}
\label{seccomp}

In this section, we prove the comparison theorem that is the main technical result of this article (Theorem \ref{compC}), as well as its $G$-equivariant companion Theorem \ref{compG}.

\subsection{Generic properties of constructible sheaves}
\label{parbasechange}

Recall that the notion of absolutely flat ring was introduced in \S\ref{parqf}.

\begin{lem}
\label{lemloclibre}
Let $A$ be a ring whose total ring of fractions $F$ is absolutely flat. Let~$\LL$ be a constructible \'etale sheaf on $\Spec(A)$.  Then there exists a non\-zero\-di\-vi\-sor $a\in A$ such that $\LL|_{\Spec(A[\frac{1}{a}])}$ is finite locally constant.
\end{lem}

\begin{proof}
All quasi-compact open subsets $U\subset \Spec(F)$ are closed by  \cite[Lemma~\href{https://stacks.math.columbia.edu/tag/092F}{092F} (2)$\Rightarrow$(3) and Lemma \href{https://stacks.math.columbia.edu/tag/04MG}{04MG} (7)$\Rightarrow$(4)]{SP}. Consequently, all constructible subsets of~$\Spec(F)$ are open. It follows that $\LL|_{\Spec(F)}$ is finite locally constant.
Viewing $F$ as the colimit of the rings $A[\frac{1}{a}]$, where $a$ runs over all non\-zero\-di\-vi\-sors of $A$, and applying \cite[Lemmas \href{https://stacks.math.columbia.edu/tag/09YU}{09YU} and \href{https://stacks.math.columbia.edu/tag/0GL2}{0GL2}]{SP} concludes.
\end{proof}

\begin{lem}
\label{lemconstr}
Let $A$ be a ring whose total ring of fractions $F$ is absolutely flat.
Let $f:X\to\Spec(A)$ be a morphism of finite presentation and let $\LL$ be a constructible \'etale sheaf on $X$. Then there exists a non\-zero\-di\-vi\-sor $a\in A$ such that the sheaves~$\RR^sf_*\LL|_{\Spec(A[\frac{1}{a}])}$ are constructible for all $s\geq 0$.
\end{lem}

\begin{proof}
As $F$ is reduced (see \cite[Proposition~4.41~(1)]{Chromatic}),  so is $A$. Write~$A$ as the directed colimit of its (reduced and noetherian) finitely generated subrings.  By \cite[Lemmas \href{https://stacks.math.columbia.edu/tag/01ZM}{01ZM} and \href{https://stacks.math.columbia.edu/tag/09YU}{09YU}]{SP}, there exist a reduced noetherian ring $A_0$, a morphism $u:B:=\Spec(A)\to B_0:=\Spec(A_0)$,  a morphism of finite presentation~${f_0:X_0\to B_0}$, and a constructible \'etale sheaf $\LL_0$ on $X_0$, such that $f$, $X$ and~$\LL$ are induced from~$f_0$,~$X_0$ and $\LL_0$ by base change by $u$. 

By Deligne's generic base change theorem \cite[Th\'eor\`eme 1.9]{Finitude} and noetherian induction, there exists a partition $B_0=\sqcup_{1\leq j\leq n} B_{0,j}$ of $B_0$ into finitely many locally closed subsets such that,  if one lets $f_{0,j}:X_{0,j}\to B_{0,j}$ denote the base change of $f_0$ by the inclusion morphism~${B_{0,j}\to B_0}$,  then the sheaf $\RR^s(f_{0,j})_*\LL$ is constructible and its formation commutes with base change for all $s\geq 0$. 

For $0\leq j\leq n$, define $B_j:=u^{-1}(B_{0,j})$. They are constructible subsets of $B$.  Since the constructible and the Zariski topology of $\Spec(F)$ coincide (as follows from \cite[Lemma \href{https://stacks.math.columbia.edu/tag/092F}{092F} (2)$\Rightarrow$(3) and Lemma \href{https://stacks.math.columbia.edu/tag/04MG}{04MG} (7)$\Rightarrow$(4)]{SP}), 
%see Lemma 0905 (8)
the inverse images of the $B_j$ in $\Spec(F)$ are quasi-compact open subsets.
View $F$ as the colimit of the rings $A[\frac{1}{a}]$, where $a$ runs over all non\-zero\-di\-vi\-sors of~$A$.  By \cite[Lemma \href{https://stacks.math.columbia.edu/tag/01Z4}{01Z4}]{SP}, we may assume, after possibly inverting some non\-zero\-di\-vi\-sor $a\in A$, that the $B_j$ are quasi-compact open subsets of $B$.  The properties of the partition $B_0=\sqcup_{1\leq j\leq n} B_{0,j}$ now imply that the $(\RR^sf_*\LL)|_{B_j}$ are constructible, which concludes the proof.
\end{proof}

\begin{rems}
(i) Total rings of fractions of reduced rings are not always absolutely flat (a counterexample is constructed in \cite[Proposition 10]{Quentel}).
%This article suggests maybe better behaved variants of total rings of fractions.
The hypotheses in Lemmas \ref{lemloclibre} and \ref{lemconstr} are therefore more stringent than asking for $A$ to be reduced.

(ii) The direct image of a constructible \'etale sheaf by a morphism of finite presentation is not constructible in general.  Indeed, one can find a ring $A$ and an element~$a\in A$ such that the closure of $\Spec(A[\frac{1}{a}])$ in $\Spec(A)$ is not constructible (see \cite[Exemple et Proposition 1]{Gamboa}). It follows that the direct image of a constant \'etale sheaf by the open immersion $\Spec(A[\frac{1}{a}])\to \Spec(A)$ is not constructible. 

(iii) Lemma \ref{lemconstr} implies that the higher direct images of a constructible \'etale sheaf by a morphism of finite presentation $f:X\to\Spec(A)$ are constructible if $A$ is absolutely flat (note that all non\-zero\-di\-vi\-sors in absolutely flat rings are invertible as a consequence of \cite[Proposition 4.41 (4)]{Chromatic}).
\end{rems}

\subsection{Killing singular cohomology classes generically in \'etale covers}

The next lemmas are key to the proof of the generic comparison theorem in \S\ref{parcompth}.

\begin{lem}
\label{lembilanvan1}
Let $S$ be a reduced Stein space.  Let $X$ be an affine \'etale~$\cO(S)$\nobreakdash-scheme. Fix $m\geq 1$.  Any class~${\alpha\in H^1(X^{\an},\Z/m)}$ vanishes in $H^1((X')^{\an},\Z/m)$ for some non\-zero\-di\-vi\-sor ${a\in\cO(S)}$ and some surjective \'etale morphism of finite presentation~${X'\to X_{\cO(S)[\frac{1}{a}]}}$.
\end{lem}

\begin{proof}
The class $\alpha$ corresponds to a topological covering~${p:T\to X^{\an}}$ of degree~$m$.  By Theorem \ref{thmconstr}, after possibly inverting some non\-zero\-di\-vi\-sor~$a\in\cO(S)$, this topological covering is the analytification of a finite \'etale covering~$Y\to X$ of degree~$m$.  Set $X':=\Isom_X(Y,\Z/m)$.  The structural morphism $f:X'\to X$ is finite \'etale and surjective.  Since the finite \'etale covering~$Y\to X$ becomes trivial after base change by $f$,  the topological covering $p$ becomes trivial after base change by $f^{\an}$, and hence $(f^{\an})^*\alpha=0$.
\end{proof}

\begin{lem}
\label{lembilanvan2}
Let $S$ be a reduced Stein space of dimension $\leq 2$.  Let $X$ be an affine \'etale~$\cO(S)$-scheme. Fix $s>0$ and $m\geq 1$.  Then any class~${\alpha\in H^s(X^{\an},\Z/m)}$ vanishes in $H^s((X')^{\an},\Z/m)$ for some non\-zero\-di\-vi\-sor ${a\in\cO(S)}$ and some surjective \'etale morphism of finite presentation~${X'\to X_{\cO(S)[\frac{1}{a}]}}$.
\end{lem}

\begin{proof}

The case $s=1$ follows from Lemma \ref{lembilanvan1}.

Assume now that $s=2$.  
Let $c\in\cO(S)$ and $q:\whS\to S$ be as in Lemma~\ref{etalegeneric}.  Replacing $S$ with~$\whS$ (which is legitimate thanks to Lemma \ref{lemBing}) and the scheme~$X$ with ${X_{\cO(S)[\frac{1}{c}]}=\Spec(\cO(\whS)[\frac{1}{c}])}$, we may assume that $X=\Spec(\cO(S)[\frac{1}{c}])$ for some non\-zero\-di\-vi\-sor $c\in\cO(S)$. 
By Corollary \ref{corkillram}, there exists an al\-ter\-ation~${p:T\to S}$ of degree $m$ such that $p^*\alpha\in H^2(p^{-1}(X^{\an}),\Z/m)$ extends to a class~${\beta\in H^2(T,\Z/m)}$.  
Consider the Stein factorization~$T\to\oT\xrightarrow{\op} S$ of $p$.  
By Corollary \ref{corGLefschetz11}, the class~$\beta$ vanishes in restriction to the complement of some nowhere dense closed analytic subset of~$T$. We deduce the existence of a non\-zero\-di\-vi\-sor~$a\in\cO(S)$ such that the class~$\op^*\alpha\in H^2(\oT\setminus \{c=0\},\Z/m)$ vanishes in $H^2(\oT\setminus \{ac=0\},\Z/m)$.  Lemma \ref{lemgenfinietaled} ensures, after possibly changing~$a$, that~$\cO(\oT)[\frac{1}{ac}]$ is a finite \'etale $\cO(S)[\frac{1}{ac}]$-algebra of degree $m$. Setting $X':=\Spec(\cO(\oT)[\frac{1}{ac}])$ concludes.

Assume finally that $s\geq 3$.  Since $X$ is affine and \'etale over $\cO(S)$,  its analytification $X^{\an}$ is Stein of dimension $\leq 2$. We deduce from \cite[Satz]{Hamm} that~$X^{\an}$ has the homotopy type of a CW complex of dimension $\leq 2$, and hence that~$H^s(X^{\an},\Z/m)=0$. One can therefore take $X'=X$.
\end{proof}

\subsection{The comparison theorem} 
\label{parcompth}

\begin{thm}
\label{compC}
Let $S$ be a reduced Stein space of dimension $\leq 2$.  
Let $X$ be an $\cO(S)$-scheme of finite presentation.
Let $\LL$ be a constructible \'etale sheaf on $X$.  If one lets~$a\in\cO(S)$ run over all non\-zero\-di\-vi\-sors, the comparison morphisms
\begin{equation}
\label{genisoC}
\underset{a}{\colim }\,H^k_{\et}(X_{\cO(S)[\frac{1}{a}]},\LL)\to \underset{a}{\colim}\, H^k((X_{\cO(S)[\frac{1}{a}]})^{\an},\LL^{\an})
\end{equation}
are isomorphisms for all $k\geq 0$.
\end{thm}

\begin{proof}
We proceed in several steps.

\setcounter{Step}{0}
\begin{Step}
\label{step1}
We reduce to the case where $X=\Spec(\cO(S))$.
\end{Step}

Let $f:X\to\Spec(\cO(S))$ be the structural morphism.   By Lemma \ref{lemconstr} (which applies since $\cM(S)$ is absolutely flat, see \S\ref{parqf}), there is a non\-zero\-di\-vi\-sor $b\in\cO(S)$ such that the sheaves $\RR^sf_*\LL|_{\Spec(\cO(S)[\frac{1}{b}])}$ are constructible. 
Both sides of (\ref{genisoC}) are computed by (colimits of) Leray spectral sequences, whose~$E_2^{r,s}$ terms read 
\begin{equation}
\label{E2pq}
\underset{a}{\colim }\,H^r_{\et}(\Spec(\cO(S)[\tfrac{1}{a}]),\RR^sf_*\LL)\textrm{ and }\underset{a}{\colim }\,H^r(S\setminus\{a=0\},\RR^sf^{\an}_*(\LL^{\an})).
\end{equation}
The natural morphisms $(\RR^sf_*\LL)^{\an}\to \RR^sf^{\an}_*(\LL^{\an})$ are isomorphisms by \cite[Theorem 3.7]{Stein}. The comparison morphisms between the $E_2^{r,s}$ terms (\ref{E2pq}) are therefore isomorphisms by the $X=\Spec(\cO(S))$ case of Theorem \ref{compC} applied to the constructible sheaves on $\Spec(\cO(S))$ obtained by extending the $\RR^sf_*\LL|_{\Spec(\cO(S)[\frac{1}{b}])}$ by zero.  It follows that the comparison morphisms (\ref{genisoC}) are isomorphisms.

\begin{Step}
\label{step2}
We further reduce to the case where $\LL=\Z/m$ for some $m\geq 1$.
\end{Step}

By Lemma \ref{lemloclibre}, there is a non\-zero\-di\-vi\-sor $b\in\cO(S)$ such that $\LL|_{\Spec(\cO(S)[\frac{1}{b}])}$ is finite locally constant.  
Let $(A_i)_{i\in I}$ be finitely many abelian groups such that each geometric fiber of $\LL|_{\Spec(\cO(S)[\frac{1}{b}])}$ are isomorphic to exactly one of the $A_i$. 
Let~$f:X\to \Spec(\cO(S)[\frac{1}{b}])$ be the surjective finite \'etale morphism representing the finite locally constant \'etale sheaf $\bigsqcup_{i\in I}\Isom_{\Spec(\cO(S)[\frac{1}{b}])}(\LL|_{\Spec(\cO(S)[\frac{1}{b}])}, A_i)$ (see~\cite[Lemma \href{https://stacks.math.columbia.edu/tag/03RV}{03RV}]{SP}).
Our choice of $f$ implies that the constructible sheaf ${\M:=f^{*}(\LL|_{\Spec(\cO(S)[\frac{1}{b}])})}$ is Zariski-locally constant. In addition, the natural morphism $\LL|_{\Spec(\cO(S)[\frac{1}{b}])}\to f_*\M$ is injective, and its cokernel is finite locally constant (see \cite[Lemmas \href{https://stacks.math.columbia.edu/tag/095B}{095B} and \href{https://stacks.math.columbia.edu/tag/03RX}{03RX} (2)]{SP}). 

Iterating this procedure, we construct a resolution of $\LL|_{\Spec(\cO(S)[\frac{1}{b}])}$ by sheaves of the form $f_*\M$ with $f:X\to \Spec(\cO(S)[\frac{1}{b}])$ finite \'etale, and $\M$ constructible and Zariski-locally constant. Making use of the spectral sequence associated with such a resolution, we reduce to the case where $\LL$ is the extension by zero of a sheaf of the form $f_*\M$. Partitioning $X$ into finitely many open subsets over which $\M$ is constant, we further reduce to the case where $\M$ is constant.

By Lemma \ref{etalegeneric}, we may assume, possibly after changing $b$, that there exists an analytic covering $q:\whS\to S$ such that $X=\Spec(\cO(\whS)[\frac{1}{b}])$. As the higher direct images of morphisms induced by $f$ and $f^{\an}$ vanish by \cite[Proposition \href{https://stacks.math.columbia.edu/tag/03QP}{03QP}]{SP} and \cite[III, Theorem 6.2]{Iversen} respectively (and making use of Lemma \ref{lemBing}), we may replace $S$ with $\whS$, and $\LL$ with a constant sheaf on $\Spec(\cO(\whS))$ that has the same stalks as $\M$.  One can then further assume that $\LL=\Z/m$ for some $m\geq 1$.

\begin{Step}
\label{step3}
We deal with case where $X=\Spec(\cO(S))$ and $\LL=\Z/m$ for some $m\geq 1$.
\end{Step}

After maybe normalizing $S$ (which is legitimate thanks to Lemma \ref{lemgenfinietaled} applied to the normalization morphism), we may assume that $S$ is normal.

The colimit over all non\-zero\-di\-vi\-sors $a\in\cO(S)$ of the Leray spectral sequences of the site morphisms $\varepsilon:(S\setminus \{a=0\})_{\cl}\to(\Spec(\cO(S)[\frac{1}{a}]))_{\et}$ (see \S\ref{paranal}) reads
\begin{equation*}
E_2^{r,s}=\underset{a}{\colim }\,H^r_{\et}(\Spec(\cO(S)[\tfrac{1}{a}]),\RR^s\varepsilon_*\Z/m) \implies \underset{a}{\colim }\,H^{r+s}(S\setminus \{a=0\},\Z/m).
\end{equation*}
It follows from Corollary \ref{compH0} that the natural sheaf morphism $\Z/m\to\varepsilon_*\Z/m$ is an isomorphism. To deduce that (\ref{genisoC}) is an isomorphism from the above spectral sequence, it therefore suffices to show that 
$\underset{a}{\colim }\,H^r_{\et}(\Spec(\cO(S)[\tfrac{1}{a}]),\RR^s\varepsilon_*\Z/m)=0$ for all $r\geq 0$ and $s>0$.  Computing \'etale cohomology using hypercoverings (see~\cite[Lemma \href{https://stacks.math.columbia.edu/tag/09VZ}{09VZ}]{SP}), it suffices to show that $\underset{a}{\colim }\,H^0_{\et}(X_{\cO(S)[\frac{1}{a}]},\RR^s\varepsilon_*\Z/m)=0$ for any affine \'etale $\cO(S)$-scheme $X$ and any $s>0$.  
By definition of the higher direct images~$\RR^s\varepsilon_*\Z/m$, this amounts 
exactly to Lemma~\ref{lembilanvan2}.
\end{proof}

\begin{rems}
\label{remcomp}
(i) 
Theorem \ref{compC} fails in general if $S$ is not reduced, already for $k=0$ and $\LL=\Z/2$. To see it, let $S$ and $X$ be as in Remark \ref{remH0} (i).  Note that all the non\-zero\-di\-vi\-sors of $\cO(S)$ are invertible.   Then the left-hand side of (\ref{genisoC}) is nonzero because $X\neq\varnothing$, but the right-hand side of (\ref{genisoC}) vanishes because $X^{\an}=\varnothing$.
%Other generically reduced counterexample by adding embedded cycles at all the punctures of the example of Remark \ref{remH1}.

(ii) Theorem \ref{compC} also fails in general if $X$ is only assumed to be of finite type over $\cO(S)$, already for $k=0$ and $\LL=\Z/2$, as the example of Remark \ref{remnotfp} shows.

(iii)
The hypothesis that $\LL$ is constructible in Theorem \ref{compC} cannot be removed either. To see it,  let $S$ be a reduced countable Stein space,  define~${X:=\Spec(\cO(S))}$, set $k=0$,  and let $\LL$ be the skyscraper \'etale sheaf on $X$ with stalk~$\Z/2$ at some maximal ideal of $\cO(S)$ associated with a nonprincipal ultrafilter of $S$. Then~$\LL^{\an}=0$, so the right-hand side of (\ref{genisoC}) vanishes but the left-hand side does not. 

(iv)
Theorem \ref{compC} would fail in general if one did not take the colimit over all non\-zero\-di\-vi\-sors $a\in\cO(S)$, already for $k=0$ and $\LL=\Z/2$.  To see it, let $S$ and $X$ be as in Remark \ref{remH0} (ii).  Then $X$ is connected, so $H^k_{\et}(X,\Z/2)=\Z/2$, but $X^{\an}$ has two connected components, so $H^k(X^{\an},\Z/2)=(\Z/2)^2$.  Another counterexample, with $k=1$ and $\LL=\Z/2$ (and $S$ normal and connected), appears in Remark \ref{remH1}~(i).

(v)
Despite the examples of (i) and (iv), it is conceivable that the comparison morphism $H^k_{\et}(\Spec(\cO(S)),\Z/m)\to H^k(S,\Z/m)$ is an isomorphism for all~${k\geq 0}$, all $m\geq 1$, and all finite-dimensional Stein spaces $S$. This is so for~${k=0}$ (by Remark~\ref{remH0}~(iii)) and $k=1$ (by Remark \ref{remH1} (ii)).

(vi)
We do not know if Theorem \ref{compC} still holds when $S$ has dimension $\geq 3$.
\end{rems}

\subsection{The \texorpdfstring{$G$}{G}-equivariant comparison theorem}

Here is a $G$-equivariant enhancement of Theorem \ref{compC}.

\begin{thm}
\label{compG}
Let $S$ be a reduced $G$-equivariant Stein space of dimension $\leq 2$.  
Let~$X$ be an $\cO(S)^G$-scheme of finite presentation.
Let $\LL$ be a constructible \'etale sheaf on~$X$.  If one lets~$a\in\cO(S)^G$ run over all non\-zero\-di\-vi\-sors, the comparison morphisms
\begin{equation}
\label{genisoG}
\underset{a}{\colim }\,H^k_{\et}(X_{\cO(S)^G[\frac{1}{a}]},\LL)\to \underset{a}{\colim}\, H_G^k((X_{\cO(S)^G[\frac{1}{a}]})^{\an},\LL^{\an})
\end{equation}
are isomorphisms for all $k\geq 0$.
\end{thm}

\begin{proof}
The $G$-equivariant site morphisms $\varepsilon:((X_{\cO(S)[\frac{1}{a}]})^{\an})_{\cl}\to (X_{\cO(S)[\frac{1}{a}]})_{\et}$ induce morphisms between the Hochschild--Serre spectral sequences 
\begin{alignat}{4}
&E_2^{r,s}=H^r(G,H^s_{\et}(X_{\cO(S)[\frac{1}{a}]},\LL))&\implies& H^{r+s}_{\et}(X_{\cO(S)^G[\frac{1}{a}]},\LL)\textrm{\hspace{.3em} and }\nonumber\\
&E_2^{r,s}=H^r(G,H^s((X_{\cO(S)[\frac{1}{a}]})^{\an},\LL^{\an}))&\implies& H^{r+s}_G((X_{\cO(S)^G[\frac{1}{a}]})^{\an},\LL^{\an})\nonumber
\end{alignat}
(for which see \cite[Remark~10.9]{Scheiderer}).
After taking the colimit over all non\-zero\-di\-vi\-sors $a\in\cO(S)$,  these morphisms are isomorphisms on page $2$ by Theorem \ref{compC},  and hence isomorphisms on the abutment.
\end{proof}

\section{Cohomological dimension and real spectra}
\label{seccohodim}

The \textit{\'etale cohomological dimension} $\cd(X)$ of a quasi-compact and quasi-separated scheme $X$ is the largest integer $n$ such that there exists a torsion \'etale sheaf~$\LL$ on $X$ with $H^n_{\et}(X,\LL)\neq 0$ (or $+\infty$ if there is no upper bound on these integers). Combining \cite[Lemmas \href{https://stacks.math.columbia.edu/tag/03SA}{03SA}~(2) and \href{https://stacks.math.columbia.edu/tag/03Q5}{03Q5}]{SP} shows that one can restrict to constructible \'etale sheaves on $X$.
We let $X_{\rr}$ be the \textit{real spectrum} of a scheme $X$ (constructed by gluing real spectra of coordinate rings of affine charts of $X$, see \cite[(0.4)]{Scheiderer}).

\subsection{\'Etale cohomological dimension}

\begin{thm}
\label{etcd}
Let $S$ be a reduced $G$-equivariant Stein space of dimension $\leq 2$.
Let~$X$ be an affine $\cO(S)^G$-scheme of finite presentation. The following assertions are equivalent.
\begin{enumerate}[(i)]
\item $\cd(X_{\cM(S)^G})<\infty$;
\item $\cd(X_{\cM(S)^G})\leq \dim(X^{\an})$;
\item the real spectrum $(X_{\cM(S)^G})_{\rr}$ is empty;
\item there exists a non\-zero\-di\-vi\-sor $a\in\cO(S)^G$ such that $(X^{\an}\setminus\{a=0\})^G=\varnothing$.
\end{enumerate}
\end{thm}

\begin{proof}
As $X$ is a closed subscheme of $\A^N_{\cO(S)^G}$ for some $N$, the complex space~$X^{\an}$ is a closed subspace of $S\times \C^N$,  and hence is finite-dimensional. This shows~(ii)$\Rightarrow$(i).

If $\LL$ is a constructible sheaf on $X_{\cO(S)^G[\frac{1}{b}]}$ for a non\-zero\-di\-vi\-sor $b\in\cO(S)^G$, then
\begin{equation}
\label{cohogenerique}
H^k_{\et}(X_{\cM(S)^G},\LL)\hspace{-.14em}=\hspace{-.14em}\underset{a}{\colim }\hspace{.1em}H^k_{\et}(X_{\cO(S)^G[\frac{1}{ab}]},\LL)\hspace{-.14em}=\hspace{-.14em}\underset{a}{\colim }\hspace{.1em}H^k_G((X_{\cO(S)^G[\frac{1}{ab}]})^{\an},\LL^{\an})
\end{equation}
for all $k\geq 0$, where $a\in\cO(S)^G$ runs over all non\-zero\-di\-vi\-sors (the first equality follows from \cite[Lemma~\href{https://stacks.math.columbia.edu/tag/03Q6}{03Q6}]{SP} and the second from Theorem \ref{compG}).

Assume now that (iv) does not hold.  Fix $k\geq 0$. Consider (\ref{cohogenerique}) with $b=1$ and~$\LL=\Z/2$.  For any non\-zero\-di\-vi\-sor $a\in\cO(S)^G$,  the image of the nonzero class of $H^k_G(\pt,\Z/2)=\Z/2$ in 
$H^k_G((X_{\cO(S)^G[\frac{1}{a}]})^{\an},\Z/2)=H^k_G(X^{\an}\setminus\{a=0\},\Z/2)$ is nonzero because its restriction to any point in $(X^{\an}\setminus\{a=0\})^G$ is nonzero. It follows that it induces a nonzero class in the right-hand side of (\ref{cohogenerique}). We deduce that~$H^k_{\et}(X_{\cM(S)^G},\Z/2)$ is nonzero. As $k$ was arbitrary, assertion (i) does not hold.

We now show that (iv)$\Rightarrow$(ii). To this effect,  assume that (iv) holds and fix a constructible \'etale sheaf~$\LL$ on $X_{\cM(S)^G}$. 
Viewing $\cM(S)^G$ as the colimit of the rings $\cO(S)^G[\frac{1}{b}]$ over all non\-zero\-di\-vi\-sors $b\in\cO(S)^G$ and using \cite[Lemma \href{https://stacks.math.columbia.edu/tag/09YU}{09YU}]{SP} shows the existence of a non\-zero\-di\-vi\-sor $b\in\cO(S)^G$ and of a constructible \'etale sheaf $\LL'$ on~$X_{\cO(S)^G[\frac{1}{b}]}$ inducing $\LL$ by base change.
It follows from (\ref{cohogenerique}) applied to $\LL'$ that
\begin{equation}
\label{Steinprincipal}
H^k_{\et}(X_{\cM(S)^G},\LL)=\underset{a}{\colim }\, H^k_G(X^{\an}\setminus\{ab=0\},(\LL')^{\an}),
\end{equation}
where $a\in\cO(S)^G$ runs over all non\-zero\-di\-vi\-sors. 
We noted above that $X^{\an}$ is a closed subspace of $S\times \C^N$ for some $N\geq 0$. It is therefore Stein. Consequently, so is $X^{\an}\setminus\{ab=0\}$ (the graph of $(ab)^{\an}$ realizes it as a closed subspace of
$X^{\an}\times\C^*$).
In view of our hypothesis (iv),  we deduce from \cite[Theorem 2.6]{Artinvanishing} and (\ref{Steinprincipal}) 
that $H^k_{\et}(X_{\cM(S)^G},\LL)=0$ for all $k>\dim(X^{\an})$.  This shows that (ii) holds.

The implication (i)$\Rightarrow$(iii) follows from \cite[Remark 7.5]{Scheiderer}.  Conversely, suppose that assertion (iii) holds.   By the already proven implication (iv)$\Rightarrow$(i) applied to the $\cO(S)^G$\nobreakdash-scheme $X':=X\times_{\Spec(\R)}\Spec(\C)$, one has $\cd(X'_{\cM(S)^G})<\infty$.   
Applying \cite[Corollary 7.21]{Scheiderer} now shows that (i) holds and completes the proof.
\end{proof}

\subsection{Cohomological dimension of fields of meromorphic functions}

The following theorem is a consequence of Theorem \ref{etcd}.  

\begin{thm}
\label{thcohodimR}
Let $S$ be a normal $G$-equivariant Stein surface with $S/G$ connected.  
The following assertions are equivalent:
\begin{enumerate}[(i)]
\item the field $\cM(S)^G$ has finite cohomological dimension;
\item the field $\cM(S)^G$ has cohomological dimension $2$;
\item the field $\cM(S)^G$ admits no field orderings;
\item $S^G$ is a discrete subset of $S$.
\end{enumerate}
\end{thm}

\begin{proof}
To deduce the result from Theorem \ref{etcd} applied with $X=\Spec(\cO(S)^G)$,  one only needs to make the following two remarks.

The first remark is that $\cd(\cM(S)^G)\geq 2$.  To see it, let $T$ be a connected component of~$S$. Since $\cM(T)$ is equal to $\cM(S)^G$ or to $\cM(S)^G[\sqrt{-1}]$, it suffices to show that~$\cd(\cM(T))\geq 2$ (use \cite[I.3.3, Proposition 14]{CG}).  
Fix $t\in T$. The local morphism $\cO(T)_t\to\cO_{T,t}$ is a faithfully flat morphism of noetherian rings (see \cite[Lemma 1.11]{Stein}). As the maximal ideal of $\cO(T)_t$ generates the maximal ideal of $\cO_{T,t}$ because $T$ is Stein,  one has $\dim(\cO(T)_t)=\dim(\cO_{T,t})=2$ by \cite[Lemma~\href{https://stacks.math.columbia.edu/tag/00ON}{00ON}]{SP}.
That $\cd(\cM(T))\geq 2$ now follows from \cite[X, Corollaire~2.5]{SGA43}.

The second remark is that $S^G$ is discrete if and only if there is a non\-zero\-di\-vi\-sor~$a\in\cO(S)^G$ vanishing on $S^G$. The direct implication follows at once from the fact that $S$ is Stein. The converse is a consequence of the discreteness of the singular locus of $S$ (as $S$ is normal of dimension $2$) and of Lemma \ref{lemrealpoints} (ii).
\end{proof}

Let us record the non-$G$-equivariant version of Theorem \ref{thcohodimR}.

\begin{thm}
\label{thcohodim}
Let $S$ be a connected normal Stein surface. Then the field $\cM(S)$ has cohomological dimension $2$.
\end{thm}

\begin{proof}
Apply Theorem \ref{thcohodimR} to the $G$-equivariant Stein surface $S\sqcup S^{\sigma}$.
\end{proof}

\subsection{Real spectra}

Let~$S$ be a $G$-equivariant Stein space.  For any $\cO(S)^G$\nobreakdash-scheme of finite presentation $X$,
the bijection $(X^{\an})^G\isoto X(\R)$ (see Lemma \ref{lemRpoints}) allows us to view~$(X^{\an})^G$ as a subset of $X_{\rr}$.

\begin{thm}
\label{thorder}
Let $S$ be a reduced $G$-equivariant Stein space of dimension $\leq 2$.
Let~$X$ be an $\cO(S)^G$-scheme of finite presentation.  Then the closure of $(X^{\an})^G$ in~$X_{\rr}$ contains $(X_{\cM(S)^G})_{\rr}$.
\end{thm}

\begin{proof}
Since the assertion in local on $X$, we may assume that $X=\Spec(A)$ is affine.  Suppose that there is an open subset $U\subset X_{\rr}=\Sper(A)$ meeting $(X_{\cM(S)^G})_{\rr}$ but not intersecting $(X^{\an})^G$. We may assume that there exist $a_1,\dots, a_m\in A$ such that~$U=\{(\kp,\prec)\in\Sper(A)\mid a_1,\dots,a_m\succ 0\}$.

Define $B:=A[x_1,\dots,x_m,\frac{1}{x_1},\dots,\frac{1}{x_m}]/\langle x_i^2-a_i\rangle$ and set $Y:=\Spec(B)$.  On the one hand, any point of $(\kp,\prec)\in U$ lifts to $Y_{\rr}$ (because the inequality $a_i\succ 0$ implies that the $a_i$ are squares in the real closure of $(\Frac(A/\kp),\prec))$. It follows that~$(Y_{\cM(S)^G})_{\rr}\neq\varnothing$.
On the other hand, if $y\in(Y^{\an})^G\isoto Y(\R)$ (see Lemma~\ref{lemRpoints}), one of the $a_i$ would have to be negative on $y$,  which is impossible since $a_i$ is a square in $B$. So $(Y^{\an})^G=\varnothing$.
This contradicts Theorem \ref{etcd} (iii)$\Leftrightarrow$(iv).
\end{proof}

We highlight the following particular case of Theorem \ref{thorder}.

\begin{thm}
\label{thorder2}
Let $S$ be a normal $G$-equivariant Stein surface with $S/G$ connected. 
Then the closure of $S^G$ in $\Sper(\cO(S)^G)$ contains $\Sper(\cM(S)^G)$.
\end{thm}

\begin{rems}
(i) Theorem \ref{thorder} fails in general if $S$ is not reduced, even when ${X=\Spec(\cO(S)^G)}$. We now show that $S_0:=\bigsqcup_{n\in\N}\Spec(\R[x]/(x^{n+1}))^{\an}$ yields such an example. One has~$\cM(S_0)^G=\cO(S_0)^G=\prod_{n\in\N}\R[x]/(x^{n+1})$.
Let~$\cU$ be a nonprincipal ultrafilter on $\N$.  
We let $\prod_{\cU}$ denote the ultraproduct  construction (with respect to~$\cU$).
%See Moret-Bailly for correct definition.
The localization~$\R[x]_{(x)}$ of $\R[x]$ at $(x)$ is a discrete valuation ring. It therefore follows from \L o\'{s}' theorem (see \eg \cite[p.~191]{BDLvdD}) that $A:=\prod_{\cU}\R[x]_{(x)}$ is a valuation ring with value group $\Gamma:=\prod_{\cU}\Z$ whose residue field $A/(x)=\prod_{\cU}\R$ is real closed.  Set $y_n:=x^{n+1}$ and $y:=(y_n)_{n\in\N}$, so $A/(y)=\prod_{\cU}\R[x]/(x^{n+1})$. 
As is explained in \cite[p.~192]{BDLvdD},  the ordered group $\Gamma$ contains $\Z$ as a convex subgroup.
Let~$\kp\subset A$ be the prime ideal consisting of those elements of $A$ whose valuation is not in $\Z$.
Since $y\in \kp$, one can view $\kp$ as a prime ideal of $A/(y)$, hence of $\cO(S_0)^G$.  In addition,  since $A/\kp$ is a valuation ring, the field $\Frac(A/\kp)$ admits a field ordering (see \eg \cite[Lemma 3.7]{Lamreal}). Let $\xi$ be the corresponding point of~$\Sper(\cO(S_0)^G)$. 
Now the element $x\in\cO(S_0)^G$ vanishes on $S_0^G$, but is nonzero on~$\xi$ because $x\notin\kp$.  It follows that $S_0^G$ cannot be dense in $\Sper(\cO(S_0)^G)$.

(ii) In the setting of Theorem \ref{thorder}, it is natural to wonder whether the stronger statement that $(X^{\an})^G$ is dense in $X_{\rr}$ holds. This is not the case in general if~$S$ is reduced curve or a normal surface, and $X=\Spec(\cO(S)^G)$.  To see it,  note that the $G$-equivariant complex space $S_0$ constructed in (i) can be $G$\nobreakdash-equivariantly embedded in a $G$-equivariant Stein space $S$ that can be chosen to be either a reduced curve or a normal surface, in such a way that induced map $S_0^G\to S^G$ is bijective.

(iii) In contrast with (ii), we do not know whether $S^G$ is dense in $\Sper(\cO(S)^G)$ if~$S$ is a nonsingular $G$-equivariant Stein surface.
%In dim 1, see [Andradas, Becker, Proposition 6.1].

(iv) We refer to \cite{AB} for a detailed study of $\Sper(\cO(\R))$ (where $\R$ is viewed as a real-analytic manifold).  In particular, the example presented in (i) is closely related to \cite[Examples 3.4 and~6.2]{AB}.

(v) Theorem \ref{thorder} also fails if $X$ is only assumed to be of finite type over~$\cO(S)^G$.  To see it, set $S:=\bigsqcup_{n\in\N}\Spec(\R)^{\an}$, so $\cO(S)^G=\cM(S)^G=\prod_{n\in\N}\R$. 
Let $\cU$ be a nonprincipal ultrafilter on $\N$. Let $\km$ be the kernel of the projection ${\prod_{n\in\N}\R\to \prod_{\cU}\R}$.
Set $X:=\Spec(\cO(S)^G/\km)$. Then~$X^{\an}=\varnothing$, but since $\cO(S)^G/\km=\prod_{\cU}\R$ is a real closed field,  one has~$X_{\rr}\neq\varnothing$.
\end{rems}

Theorem \ref{thorder2} shows that $S^G$ controls $\Sper(\cM(S)^G)$ (when $S$ has dimension~$\leq 2$).  The following (much easier) lemma provides us with a converse (for any $S$).

\begin{lem}
\label{lemgenerization}
Let $S$ be a reduced $G$-equivariant Stein space.
Then any nonsingular point of $S^G$ admits a generization in $\Sper(\cO(S)^G)$ that belongs to $\Sper(\cM(S)^G)$.
\end{lem}

\begin{proof}
Let $s\in S^G$ be a nonsingular point.  Replacing $S$ with its irreducible component through $s$, we may assume that $S$ is irreducible.  The local ring $\cO(S)_s$ is regular (apply \cite[Lemma 1.13]{Stein} with $K=\{s\}$).
As $\cO(S)_s=(\cO(S)_s)^G[\sqrt{-1}]$,  we deduce from
\cite[Lemma \href{https://stacks.math.columbia.edu/tag/07NG}{07NG}]{SP} that so is $(\cO(S)_s)^G$. The lemma now follows from 
\cite[Lemma 2.3]{Henselian} as $\cM(S)^G=\Frac((\cO(S)_s)^G)$.
\end{proof}

\section{Applications to the period-index problem}
\label{secpi}

\subsection{Computing the Brauer group}

If $F$ is a field,   the subgroup~$\Br(F)_0$ of~$\Br(F)$ was defined in (\ref{Brnul}) to be $\Ker\big[\Br(F)\to \prod_{\xi\in\Sper(F)}\Br(F_{\xi})\big]$.  We also recall that if~$E$ is a topological space on which $G$ acts and $A$ is a $G$-module,  then the subgroup
$H^k_G(E,A)_0$ of $H^k_G(E,A)$ is $\Ker\big[H^k_G(E,A)\to \prod_{x\in E^G}H^k_G(x,A)\big]$ (see~\S\ref{parGeq}).

\begin{lem}
\label{BrauerG}
Let $S$ be a normal $G$-equivariant Stein surface with $S/G$ connected.  Fix $m\geq 1$. Letting  $a\in\cO(S)^G$ runs over all non\-zero\-di\-vi\-sors, one has
\begin{align}
\label{BrauercalculG}
&\Br(\cM(S)^G)[m]=\underset{a}{\colim}\,H^2_G(S\setminus\{a=0\},\Z/m(1))\\
\label{Brauercalcul0}
\textrm{and }\hspace{.5em}&\Br(\cM(S)^G)_0[m]=\underset{a}{\colim}\,H^2_G(S\setminus\{a=0\},\Z/m(1))_0.
\end{align}
\end{lem}

\begin{proof}
One computes that
\begin{equation}
\label{calculBr}
\begin{alignedat}{3}
\Br(\cM(S)^G)[m]&=H^2_{\et}(\Spec(\cM(S)^G),\Z/m(1))\\
&=\underset{a}{\colim}\,H^2_{\et}(\Spec(\cO(S)^G[\tfrac{1}{a}]),\Z/m(1))\\
&=\underset{a}{\colim}\,H^2_G(S\setminus\{a=0\},\Z/m(1)),
\end{alignedat}
\end{equation}
where the first equality follows from the Kummer exact sequence, the second from \cite[Lemma~\href{https://stacks.math.columbia.edu/tag/03Q6}{03Q6}]{SP}, and the third from Theorem \ref{compG}. This shows (\ref{BrauercalculG}).

Fix a non\-zero\-di\-vi\-sor $a\in\cO(S)^G$.  Scheiderer has defined a sheaf $\RR^2\rho\,(\Z/m(1))$ on $\Sper(\cO(S)^G[\tfrac{1}{a}])$ whose stalk at a point $\xi\in\Sper(\cO(S)^G[\tfrac{1}{a}])$ with associated real closed residue field~$\kappa_{\xi}$ is $H^2_{\et}(\Spec(\kappa_{\xi}),\Z/m(1))$,  such that there exists a morphism 
\begin{equation}
\label{SpecSper}
\phi:H^2_{\et}(\Spec(\cO(S)^G[\tfrac{1}{a}]),\Z/m(1))\to H^0(\Sper(\cO(S)^G[\tfrac{1}{a}]),\RR^2\rho\,(\Z/m(1)))
\end{equation}
with $\phi(\alpha)_{\xi}=\alpha|_{\kappa_{\xi}}$ in $H^2_{\et}(\Spec(\kappa_{\xi}),\Z/m(1))$ for $\alpha\in H^2_{\et}(\Spec(\cO(S)^G[\tfrac{1}{a}]),\Z/m(1))$ (see~\cite[Definition~2.5,  Proposition 3.12 b) and~c)]{Scheiderer}). 

Assume now that $a$ has been chosen so~$S\setminus \{a=0\}$ is nonsingular and fix a class~$\alpha\in H^2_{\et}(\Spec(\cO(S)^G[\tfrac{1}{a}]),\Z/m(1))$.  Since the sheaf $\RR^2\rho\,(\Z/m(1))$ is locally constant (see \cite[Proposition 17.4 b)]{Scheiderer}),  we deduce from Theorem \ref{thorder2} and Lemma~\ref{lemgenerization} that it is equivalent to require that $\phi(\alpha)$ vanishes on $S^G\setminus\{a=0\}$, or on~$\Sper(\cM(S)^G)$. Restricting to such elements $a\in\cO(S)^G$ and such classes $\alpha$ in the chain of equalities (\ref{calculBr}) yields the isomorphism~(\ref{Brauercalcul0}).
\end{proof}

\begin{cor}
\label{BrauerC}
Let $S$ be a connected normal Stein surface.  Fix $m\geq 1$. Then 
\begin{equation}
\label{Brauercalcul}
\Br(\cM(S))[m]=\underset{a}{\colim}\,H^2(S\setminus\{a=0\},\Z/m),
\end{equation}
where $a\in\cO(S)$ runs over all non\-zero\-di\-vi\-sors.
\end{cor}

\begin{proof}
Apply Lemma \ref{BrauerG} to the $G$-equivariant Stein surface $S\sqcup S^{\sigma}$.
\end{proof}

\subsection{The period-index problem on Stein surfaces}

\begin{thm}
\label{thpiC}
Let $S$ be a connected normal Stein surface.  For all $\eta\in \Br(\cM(S))$, one has $\ind(\eta)=\per(\eta)$.
\end{thm}

\begin{proof}
Set $m:=\per(\eta)$, so $\eta\in \Br(\cM(S))[m]$. Let $a\in\cO(S)$ be a non\-zero\-di\-vi\-sor and let $\alpha\in H^2(S\setminus\{a=0\},\Z/m)$ be a class inducing $\eta$ in (\ref{Brauercalcul}).
By Corollary~\ref{corkillram}, there exist an alteration $p:T\to S$ of degree $m$ with $T$ nonsingular, and a cohomology class~$\beta\in H^2(T,\Z/m)$ such that $p^*\alpha=\beta|_{T\setminus\{a=0\}}$ in $H^2(T\setminus\{a=0\},\Z/m)$.

By Corollary \ref{corGLefschetz11}, the class $\beta$ vanishes in restriction to the complement of some nowhere dense closed analytic subset of $T$. It follows that there exists a non\-zero\-di\-vi\-sor $b\in\cO(T)$ such that $p^*\alpha|_{T\setminus\{ab=0\}}$ vanishes in $H^2(T\setminus\{ab=0\},\Z/m)$.  By Corollary \ref{BrauerC}, the class $\eta$ vanishes in $\Br(\cM(T))$.  Since $\cM(T)$ is an \'etale~$\cM(S)$\nobreakdash-algebra of degree $m$ (apply~(\ref{eqext}) to the Stein factorization $T\to\oT\xrightarrow{\op}S$ of $p$), we deduce that $\ind(\eta)\leq m$. As $\per(\eta)$ always divides $\ind(\eta)$, the theorem is proved.
\end{proof}

\begin{rem}
\label{remcyclic}
The field extension~$\cM(T)$ of $\cM(S)$ constructed in the proof of Theorem \ref{thpiC} is cyclic of degree $m$ (see especially the equation of $\wT$ in the proof of Proposition \ref{killram}). It follows that all central division algebras over~$\cM(S)$ are cyclic. We refer to \cite[Theorem 2.1]{CTOP} for the analogous result in the local case.
\end{rem}

\begin{thm}
\label{thpi}
Let $S$ be a normal $G$-equivariant Stein surface with $S/G$ connected.  Then, for all $\eta\in \Br(\cM(S)^G)_0$, one has $\ind(\eta)=\per(\eta)$.
\end{thm}

\begin{proof}
We argue by induction on $m:=\per(\eta)$. 
If $m$ is odd, the theorem follows from Theorem \ref{thpiC} since $\cM(S)$ is an \'etale $\cM(S)^G$\nobreakdash-al\-ge\-bra of degree $2$. 

Assume now that $m=2$, so $\eta\in \Br(\cM(S)^G)_0[2]$.
Let~$a\in\cO(S)^G$ be a non\-zero\-di\-vi\-sor and let $\alpha\in H^2_G(S\setminus\{a=0\},\Z/2)_0$ be a class inducing $\eta$ in~(\ref{Brauercalcul0}).
By Proposition~\ref{killram2}, there exist a $G$-equivariant alteration~${p:T\to S}$ of degree~$2$ with~$T$ nonsingular, and a class $\beta\in H^2_G(T,\Z/2)$ such that $\partial(\beta)=0$ in $H^3_G(T,\Z(1))$ and~${p^*\alpha=\beta|_{T\setminus\{a=0\}}}$ in~$H^2_G(T\setminus\{a=0\},\Z/2)$.
As $\partial(\beta)=0$, the class $\beta$ lifts to~$H^2_G(T,\Z(1))$ and hence
vanishes in restriction to the complement of some $G$\nobreakdash-in\-vari\-ant nowhere dense closed analytic subset of~$T$ (see Proposition \ref{GLefschetz11}).  It follows that there exists a non\-zero\-di\-vi\-sor~${b\in\cO(T)^G}$ such that $p^*\alpha|_{T\setminus\{ab=0\}}$ vanishes in ${H^2_G(T\setminus\{ab=0\},\Z/2)}$. By Lemma \ref{BrauerG}, the class~$\eta$ vanishes in $\Br(\cM(T)^G)$. As~$\cM(T)^G$ is an \'etale~$\cM(S)^G$\nobreakdash-al\-ge\-bra of degree~$2$ (apply~(\ref{eqextG}) to the Stein factorization $T\to\oT\xrightarrow{\op}S$ of~$p$), we deduce that $\ind(\eta)\mid 2$, hence that $\ind(\eta)=2$. 

Assume finally that $m$ is even.  By the above, the class $\frac{m}{2}\cdot\eta\in \Br(\cM(S)^G)_0[2]$ is killed in a degree $2$ extension of $\cM(S)^G$ (which has the form $\cM(T)^G$ for some $G$\nobreakdash-equivariant analytic covering $p:T\to S$ (see (\ref{eqextG})). Then $p^*\eta\in\Br(\cM(T)^G)_0$ has period dividing $\frac{m}{2}$,  hence index dividing $\frac{m}{2}$ by induction.
It follows that $\ind(\eta)\mid m$, hence that $\ind(\eta)=m$.
\end{proof}

\begin{thm}
\label{corpi}
Let $S$ be a normal $G$-equivariant Stein surface with $S/G$ connected and $S^G\subset S$ discrete.  Then, for all $\eta\in \Br(\cM(S)^G)$, one has $\ind(\eta)=\per(\eta)$.
\end{thm}

\begin{proof}
Theorem \ref{thcohodimR} shows that the field $\cM(S)^G$ admits no field ordering. It follows that~$\eta\in\Br(\cM(S)^G)_0$, and Theorem \ref{thpi} applies.
\end{proof}

%Unramified Brauer classes?????

\subsection{Application to the \texorpdfstring{$u$}{u}-invariant}
\label{paru}

The $u$-\textit{invariant} $u(F)$ of a field $F$ is the maximal rank of an anisotropic quadratic form over $F$ that is hyperbolic over all real closures of $F$,  or $+\infty$ if these ranks admit no upper bound (see \cite[Definition~1.1]{EL}). In the more classical particular case when $F$ admits no field orderings,  the $u$-invariant $u(F)$ is the maximal rank of an anisotropic quadratic form over $F$.

\begin{thm}
Let $S$ be a normal $G$-equivariant Stein surface with $S/G$ connected. Then $u(\cM(S)^G)=4$.
\end{thm}

\begin{proof}
We only briefly explain why this follows from our main theorems and from arguments that appear in the literature.

If $\cM(S)^G$ admits no field orderings (\resp admits a field ordering), then $\cM(S)^G$ (\resp $\cM(S)^G[\sqrt{-1}]$) has cohomological dimension $2$ by Theorem \ref{thcohodimR},  so the inequality $u(\cM(S)^G)\leq 4$ follows from the arguments of \cite[Theorem 3.4]{CTOP} (\resp of \cite[Theorem 4.4]{CTOP}) replacing the use of \cite[Theorem 2.1]{CTOP} (\resp of \cite[Theorem 4.1]{CTOP}) with an application of Theorem \ref{thpi}.

As for the easier inequality $u(\cM(S)^G)\geq 4$, it can be proven by adapting the arguments of the last paragraph of \cite[\S 6.4, Proof of Theorem 0.12]{pi}.
\end{proof}

\section{Applications to Serre's conjecture II}
\label{secSerre}

\begin{thm}
\label{thab}
Let $S$ be a connected normal Stein surface.  Then the maximal abelian extension $\cM(S)^{\ab}$ of $\cM(S)$ has cohomological dimension $1$.
\end{thm}

\begin{proof}
Set $K:=\cM(S)$. We work in a fixed algebraic closure $\oK$ of $K$.
By \cite[II.3.1, Proposition 5]{CG}, it suffices to prove that $\Br(F)=0$ for any subfield~$F$ of $\oK$ that is a finite extension of $K^{\ab}$.  By the last paragraph of the proof of \cite[Theorem 2.2]{CTOP} (or by the more elementary argument explained in \cite[Remark~2.2.1]{CTOP}),  we may assume that $F$ is Galois over $K$.

Fix $\eta\in\Br(F)$. Then there exists a subfield $L\subset F$ which is a finite Galois extension of $K$ and a class $\eta_L\in \Br(L)$ such that $\eta_L$ induces~$\eta$ by restriction.  Let~${\pi:\whS\to S}$ be the analytic covering of connected normal Stein surfaces corresponding to the field extension $K\subset L$ in (\ref{eqext}).  The Galois group $\Gamma:=\Gal(L/K)$ acts on $\whS$ by functoriality of (\ref{eqext}).
Define $m:=\per(\eta_L)$ and $N:=m|\Gamma|$.

Let $a\in\cO(\whS)$ be a non\-zero\-di\-vi\-sor and let $\alpha\in H^2(\whS\setminus\{a=0\},\Z/m)$ be a class inducing $\eta_L$ in (\ref{Brauercalcul}).  After replacing $a$ with $\prod_{\gamma\in\Gamma}\gamma^*a$, we may assume that the divisor ${\whS':=\{a=0\}}$ is $\Gamma$-invariant.
By Proposition \ref{killramab}, there exist~${f\in \cM(S)^*}$, an alteration $p:T\to\whS$ with~${\cM(T)=\cM(\whS)[f^{\frac{1}{N}}]}$,  and a class~${\beta\in H^2(T,\Z/m)}$ such that ${p^*\alpha=\beta|_{T\setminus \{a=0\}}}$ in~$H^2(T\setminus \{a=0\}),\Z/m)$.

By Corollary \ref{corGLefschetz11}, the class $\beta$ vanishes in restriction to the complement of some nowhere dense closed analytic subset of $T$. It follows that there exists a non\-zero\-di\-vi\-sor $b\in\cO(T)$ such that $p^*\alpha|_{T\setminus\{ab=0\}}$ vanishes in $H^2(T\setminus\{ab=0\},\Z/m)$.  By Corollary \ref{BrauerC}, the image of $\eta_L$ in $\Br(\cM(T))=\Br(L[f^{\frac{1}{N}}])$ vanishes.  As~$L[f^{\frac{1}{N}}]\subset F$, it follows that $\eta=0$, as wanted.
\end{proof}

\begin{thm}
\label{thSerre1}
Let $S$ be a connected normal Stein surface. Let $H$ be a simply connected semisimple algebraic group over $\cM(S)$. Then $H^1(\cM(S),H)=0$.
\end{thm}

\begin{proof}
This follows from \cite[Th\'eor\`eme 1.2 (v)]{CTGP}, whose hypotheses are satisfied by Theorems \ref{thcohodim},  \ref{thpiC} and \ref{thab}.
\end{proof}

\begin{thm}
\label{thSerre2}
Let $S$ be a normal $G$-equivariant Stein surface with $S/G$ connected and $S^G$ discrete.  Let $H$ be a simply connected semisimple algebraic group over~$\cM(S)^G$. Then ${H^1(\cM(S)^G,H)=0}$.
\end{thm}

\begin{proof}
There are finite extensions $(F_i)_{i\in I}$ of~$\cM(S)^G$ and simply connected absolutely almost simple algebraic groups $H_i$ over~$F_i$ such that $H=\prod_{i\in I} \RR_{F_i/\cM(S)^G}H_i$ (see \cite[Proposition A.5.14]{CGP}). Then $H^1(\cM(S)^G,H)=\prod_{i\in I}H^1(F_i,H_i)$ by Sha\-piro's lemma \cite[Corollaire 2.10]{BoSe}.
%cas imparfait, see [Oesterlé, Nombres de Tamagawa, IV, 2.3] 
As the $F_i$ are of the same nature as~$\cM(S)^G$ (see (\ref{eqextG})), we are reduced to the case where $H$ is absolutely almost simple. In this case,  any class in $H^1(\cM(S)^G,H)$ is trivial over the degree $2$ \'etale extension $\cM(S)$ of $\cM(S)^G$ by Theorem \ref{thSerre1}, hence is trivial by \cite[Corollaire~5.5.3]{Gillebook} (which applies because $\cM(S)^G$ has cohomological dimension $2$, see Theorem \ref{thcohodimR}).
\end{proof}

\section{Applications to Hilbert's \texorpdfstring{$17$}{17}th problem}
\label{secH17}

If $X$ is a scheme on which $2$ is invertible,  we let $\{h\}\in H^1_{\et}(X,\Z/2)$ be the image of $h\in\cO(X)^*$ by the boundary map of the exact sequence $0\to\Z/2\to \G_m\xrightarrow{2}\G_m\to 0$ of \'etale sheaves on $X$. 
The next lemma is well-known (see \cite[Lemma 1.2]{CT4}).

\begin{lem}
\label{lemsquares}
Let $F$ be a field of characteristic $\neq 2$. Fix $h\in F^*$. The following are equivalent:
\begin{enumerate}[(i)]
\item the element $h$ is a sum of $3$ squares in $F$;
\item the element $-1$ is a sum of $2$ squares in $F[\sqrt{-h}]$;
\item the quaternion algebra $(-1,-1)$ splits over $F[\sqrt{-h}]$;
\item the class $\{-1\}^2\in H^2_{\et}(\Spec(F[\sqrt{-h}]),\Z/2)$ vanishes.
\end{enumerate}
\end{lem}

\begin{proof}
If $h\in F^*$ is a square, then $-1$ is a square in $F[\sqrt{-h}]$ and all assertions are true. Assume otherwise.
Set $x:=\sqrt{-h}$.
Suppose that $h=a^2+b^2+c^2$ in $F$. Then
$$-1=\frac{b^2+c^2}{a^2+x^2}=\Big(\frac{ab+xc}{a^2+x^2}\Big)^2+\Big(\frac{ac-xb}{a^2+x^2}\Big)^2.$$
This proves (i)$\Rightarrow$(ii).
Conversely,  assume that $-1=(a+bx)^2+(c+dx)^2$ with ${a,b,c,d\in F}$.  This implies that $1+a^2+c^2=(b^2+d^2)\cdot f$ and $ab+cd=0$. Unless $b^2+d^2=0$ (in which case $-1=a^2+c^2$), one computes that
$$f=\frac{1+a^2+c^2}{b^2+d^2}=\Big(\frac{b}{b^2+d^2}\Big)^2+\Big(\frac{d}{b^2+d^2}\Big)^2+\Big(\frac{ad-bc}{b^2+d^2}\Big)^2.$$
As for the equivalences (ii)$\Leftrightarrow$(iii) and (iii)$\Leftrightarrow$(iv), see \eg \cite[Proposition 1.3.2]{GS} and \cite[Proposition 4.7.1]{GS}.
\end{proof}

\begin{thm}
\label{3squares}
Let $S$ be a normal $G$-equivariant Stein surface. Fix $h\in\cO(S)^G$. The following assertions are equivalent:
\begin{enumerate}[(i)]
\item there exists a closed discrete subset $\Sigma\subset S^G$ such that $h\geq 0$ on $S^G\setminus\Sigma$;
\item $h$ is a sum of squares in $\cM(S)^G$;
\item $h$ is a sum of $3$ squares in $\cM(S)^G$.
\end{enumerate}
\end{thm}

\begin{proof}
We assume, as we may, that $S/G$ is connected.

The implication (iii)$\Rightarrow$(ii) is obvious.  To prove (ii)$\Rightarrow$(i), we claim that one can take $\Sigma:=\{x\in S^G\mid x\textrm{ is a singular point of }S\}$ (it is discrete since $S$ is normal of dimension $2$). To prove the claim,  write $h=\sum_i h_i^2$  in~$\cM(S)^G$. Let $S'\subset S$ 
be the nowhere dense closed analytic subset along which one of the $h_i$ has poles. Then~$h\geq 0$ on~${S^G\setminus (S')^G}$.  If $x\in S^G\setminus\Sigma$, then $x$ belongs to the closure of $S^G\setminus (S')^G$ (by Lemma~\ref{lemrealpoints} applied to small neighborhoods of $x$), and hence~$h(x)\geq 0$.

We now prove (i)$\Rightarrow$(iii).  We may suppose that $h\neq 0$.  Let $p:T\to S$ be the  $G$-equivariant analytic covering of normal Stein surfaces associated with the \'etale $\cM(S)^G$-algebra $\cM(S)^G[x]/\langle x^2+h\rangle$ (see~(\ref{eqextG})).  
Let $\nu:\wT\to T$ be a $G$-equivariant resolution of singularities (see Proposition \ref{ressing}).
The element $x\in\cM(T)^G$ belongs to~$\widehat{\cO}_T(T)^G$ (in the notation of \cite[6, \S 4.1]{GRCoherent}) and hence to~$\cO(T)^G$ because $T$ is normal. That $h\geq 0$ on $T^G\setminus p^{-1}(\Sigma)$ and $h=-x^2$ implies that $T^G$ is included in the nowhere dense analytic subset $\{h=0\}\cup p^{-1}(\Sigma)$ of $T$, and hence that $\wT^G=\varnothing$ (see Lemma~\ref{lemrealpoints} (ii)).
Theorem~\ref{compG} yields an isomorphism
\begin{equation}
\label{applicompa}
\underset{a}{\colim }\, H^2_{\et}(\Spec(\cO(T)^G[\tfrac{1}{a}]),\Z/2)\isoto\underset{a}{\colim }\,H^2_G(T\setminus \{a=0\},\Z/2),
\end{equation}
where $a\in\cO(T)^G$ runs over all non\-zero\-di\-vi\-sors.  
Consider the image $\alpha$ of the class $\{-1\}^2\in H^2_{\et}(\Spec(\cO(T)^G),\Z/2)$ in $H^2_G(T,\Z/2)$. Define $\beta:=\nu^*\alpha\in H^2_G(\wT,\Z/2)$. 
As $\wT^G=\varnothing$, one has~$H^3_G(\wT,\Z(1))=0$ by \cite[Proposition 2.8]{Artinvanishing}.  It follows that $\beta$ lifts to a class in $H^2_G(\wT,\Z(1))$ and hence, by Proposition \ref{GLefschetz11}, that it vanishes in the complement of a $G$\nobreakdash-invariant nowhere dense closed analytic subset of $\wT$. 
 We deduce that $\alpha$ vanishes in ${H^2_G(T\setminus \{a=0\},\Z/2)}$ for some non\-zero\-di\-vi\-sor $a\in\cO(T)^G$.
 It follows from~(\ref{applicompa}) that $\{-1\}^2$ vanishes in $H^2_{\et}(\Spec(\cO(T)^G[\tfrac{1}{a}]),\Z/2)$ for some (possibly different) non\-zero\-di\-vi\-sor $a\in\cO(T)^G$, hence a fortiori in~$H^2_{\et}(\Spec(\cM(T)^G),\Z/2)$.  To conclude, apply Lemma \ref{lemsquares} (iv)$\Rightarrow$(i).
\end{proof}

The case of the function $h=-1$ is of particular interest.

\begin{cor}
Let $S$ be a normal $G$-equivariant Stein surface. The following assertions are equivalent:
\begin{enumerate}[(i)]
\item the subset $S^G$ of $S$ is discrete;
\item $-1$ is a sum of squares in $\cM(S)^G$;
\item $-1$ is a sum of $2$ squares in $\cM(S)^G$.
\end{enumerate}
\end{cor}

\begin{proof}
Use Theorem \ref{3squares} and note that if $-1$ is a sum of $3$ squares in a field $F$, then it is a sum of $2$ squares in $F$ (apply \cite[Satz 4]{Pfister} or use Lemma \ref{lemsquares} (i)$\Rightarrow$(ii)).
\end{proof}

Arguing as in the proof of \cite[Theorem 6.5]{Stein}, we obtain the following consequence in real-analytic geometry. We follow the conventions of \cite{GMT} and refer to \cite[II, Definition 1.4]{GMT} for the definitions of \textit{real-analytic spaces} and \textit{real-analytic varieties}.

\begin{thm}
\label{corsquares}
Let $M$ be a normal real-analytic variety of pure dimension $2$ and fix ${h\in\cO(M)}$. The following assertions are equivalent:
\begin{enumerate}[(i)]
\item $h\geq 0$ on $M$;
\item $h$ is a sum of squares in $\cM(M)$;
\item $h$ is a sum of $3$ squares in $\cM(M)$.
\end{enumerate}
\end{thm}

\begin{proof}
By \cite[IV, Proposition 3.8]{GMT}, 
%Pure dimensional is important, cf [op. cit., Remark 3.9].
the real-analytic variety $M$ is coherent, hence a real-analytic space.
By \cite[III, Theorems 3.6 and~3.10]{GMT}, there exist a normal $G$-equivariant Stein surface $S$, and an isomorphism $M\isoto S^G$ of real-analytic spaces such that~$M$ admits a basis of $G$-invariant Stein open neighborhoods in $S$.
By \cite[III, Proposition 1.8]{GMT}, one may replace $S$ with one such neighborhood in such a way that~$h$ extends to a holomorphic map $\th:S\to \C$. After replacing $\th$ with the function $z\mapsto(\th(z)+\overline{\th\circ \sigma(z)})/2$, we may assume that it is $G$-equivariant.  One may now apply Theorem \ref{3squares} to the function $\th$ to conclude.
\end{proof}

\bibliographystyle{myamsalpha}
\bibliography{Steinsurface}

\end{document}